\documentclass[11pt]{amsart}

\usepackage{amsmath, amssymb, amscd, cancel, graphicx, paralist, soul, stmaryrd, mathtools, amsfonts, tikz}
\usetikzlibrary{knots}
\usetikzlibrary{hobby}

\usepackage{amsmath, amssymb, amscd, cancel, graphicx, paralist, soul, stmaryrd}
\usepackage{tikz}
\usetikzlibrary{knots}
\usetikzlibrary{hobby}
\bibliography{refs}
\usepackage{amsthm}

\usepackage{mathdots}

\usepackage{hyperref}

\usepackage{tikz}

\usepackage[all]{xy}
\usepackage{multirow}
\usepackage{longtable}
\usepackage{array}

\newtheorem{thm}{Theorem}[section]
\newtheorem{conj}[thm]{Conjecture}
\newtheorem{cor}[thm]{Corollary}
\newtheorem{lem}[thm]{Lemma}

\newtheorem{prop}[thm]{Proposition}

\newtheorem{defin}[thm]{Definition}

\newtheorem{rmk}[thm]{Remark}

\def\Char{{\mathrm{Char}}}

\def\int{{\text{int}}}

\def\mod{{\textup{mod} \;}}

\def\rk{{\mathrm{rk}}}

\begin{document}

\title[On lens space surgeries from the Poincar\'e homology sphere]%
{On lens space surgeries from the Poincar\'e homology sphere}

\begin{abstract}
    Building on Greene's changemaker lattices, we develop a lattice embedding obstruction to realizing an L-space bounding a definite 4-manifold as integer surgery on a knot in the Poincar\'e homology sphere. As the motivating application, we determine which lens spaces are realized by $p/q$-surgery on a knot $K$ when $p/q > 2g(K) -1$. Specifically, we use the lattice embedding obstruction to show that if $K(p)$ is a lens space and $p \geq 2g(K)$, then there exists an equivalent surgery on a Tange knot with the same knot Floer homology groups; additionally, using input from Baker, Hedden, and Ni, we identify the only two knots in the Poincar\'e homology sphere that admit half-integer lens space surgeries. Thus, together with the Finite/Cyclic Surgery Theorem of Boyer and Zhang, we obtain the corollary that lens space surgeries on hyperbolic knots in the Poincar\'e homology sphere are integral.
\end{abstract}
\author[Jacob Caudell]{Jacob Caudell}

\address{Department of Mathematics, Boston College\\ Chestnut Hill, MA 02467}

\email{caudell@bc.edu}

\maketitle

\section{Introduction}
\subsection{Background.}
\emph{Dehn surgery}---the cut and paste operation whereby one solid torus is removed from a 3-manifold and replaced by another---has been one of the most well-studied and befuddling constructions in low-dimensional topology since Dehn first performed surgery in 1910. A classical result (\cite{Lic62},\cite{Wal60}) states that for any two 3-manifolds $Y$ and $Y'$, there is a link $L$ in $Y$ admitting a Dehn surgery to $Y'$, but determining whether $Y'$ can be obtained by surgery on a \emph{knot} $K$ in $Y$---and if so then for which knots in $Y$---is notoriously difficult in general. In certain cases---e.g. when $Y$ is $S^3$ and $Y'$ is reducible, contains an incompressible torus, or has a finite fundamental group---general constructions motivated by low-complexity examples have led to conjecturally complete accounts of all such surgeries. At present, we treat the case where $Y$ is the Poincar\'e homology sphere $\mathcal P$, oriented as the boundary of the negative definite $E_8$ plumbing, and $Y'$ is homeomorphic to a connected sum of exactly one or two lens spaces. Note that the scenario where $Y'$ is instead a connected sum of three or more lens spaces was dispatched in \cite{Cau21}, and the main obstruction formulated there serves as the foundation for the present work.

To set up the statement of our main results, we recall some conventions of Dehn surgery. For $K \subset \mathcal P$ with exterior $M_K : = \mathcal P \setminus \overset{\circ}{\nu}(K)$, identify $\mathbb Q \cup \{1/0\}$ with the set of \emph{slopes} on $\partial M_k$ by $p/q \mapsto p[\mu] + q [\lambda]$, where $\mu$ is the meridian of $K$ and $\lambda$ is the \emph{Seifert longitude} of $K$ oriented such that $\mu \cdot \lambda = 1$, i.e. the unique slope on $\partial \nu(K)$ that bounds in $M_K$ with $\mu \cdot \lambda = 1$. Denote by $K(p/q)$ the manifold obtained by identifying the boundaries of an abstract solid torus $D^2 \times S^1$ and $M_K$ such that $\partial D^2 \times \{\text{pt}\}$ is identified with the slope $p/q$, and call this manifold $p/q$-surgery on $K$. Practitioners of Heegaard Floer homology will recognize that for $K(p/q)$ to be a connected sum of lens spaces, we must have $p/q \geq 2g(K)-1$, as $\mathcal P$ and connected sums of lens spaces are all \emph{Heegaard Floer L-spaces}---that is, they satisfy $\text{rk }\widehat{HF}(Y) = |H_1(Y; \mathbb Z)|$. Practitioners of surface intersection graph techniques will recognize that if $p\geq 2g(K)$ and $K(p)$ is reducible, then $K$ is the $(r,s)$-cable of a knot $\kappa$, where $p = rs$, $s\geq 2$, and $K(p) \cong \kappa(r/s) \# L(s,r)$ \cite{MS03} (cf. \cite{Hof98}).

\subsection{Main results.} We now state our main results, which resolve the question of which knots in $\mathcal P$ admit half-integer lens space surgeries---thereby completely resolving the question of which knots admit non-integer lens space surgeries---and the question of which lens spaces are realized by $\geq 2g(K)$-surgery on a knot $K \subset \mathcal P$. 

\subsubsection{Half-integer surgeries.} 
Figures 1 and 2 present two knots in $\mathcal P$, $E$ and $C$, with half-integer lens space surgeries.

\begin{figure}
    \centering
\begin{tikzpicture}[scale = .75]
    \begin{knot}[flip crossing = 2, flip crossing = 3, flip crossing = 6, flip crossing = 8, flip crossing = 10]
    \strand[ultra thick] (0,0) to[curve through = {(2,2)..(0,4)..(-2,2)}] (0,0);
    \strand[ultra thick] (-3.2,2) to[curve through = {(-2.2,1)..(-1.2,2)..(-2.2,3)}] (-3.2,2);
    \strand[ultra thick] (0,.8) to[curve through = {(1, -.2)..(0,-1.2)..(-1,-.2)}] (0,.8);
    \strand[ultra thick] (3.2,2) to [curve through = {(2.2,3)..(1.2,2)..(2.2,1)}] (3.2,2);
    \strand[thick, red] (-4+.1 , 2) to [curve through = {(-3.2,1.2+.1)..(-2.4-.1,2)..(-3.2,2.8-.1)}] (-4+.1,2);
    \node at (0,4.25) {$\Large 0$};
    \node at (-3.4,2) {$\Large 3$};
    \node[red] at (-4,2.8) {$\Large E$};
    \node[red] at (-4.2,2) {$\Large \frac{1}{2}$};
    \node at (0,-1.45) {$\Large -2$};
    \node at (3.4,2) {$\Large 5$};
    \end{knot}
\end{tikzpicture}
    \caption{Performing a slam-dunk of $E$ through the 3-framed unknot followed by a slam-dunk of the resulting 1-framed unknot through the central 0-framed unknot yields a linear surgery diagram for $L(7,6)$. Thus, $E\subset \mathcal P$, the exceptional fiber of order $3$, admits a half-integer surgery to $L(7,6)$.}
\end{figure}
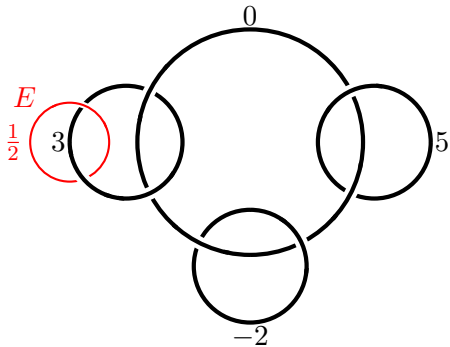

\begin{figure}
    \centering
    \begin{tikzpicture}
    \begin{knot}[clip width = 3, flip crossing = 2, flip crossing = 3, flip crossing = 5, flip crossing = 7, flip crossing = 10, flip crossing = 12, flip crossing = 14, flip crossing = 15]
    \strand[ultra thick] (3,0) to[curve through={(4,1)..(3,2)..(2,1)}] (3,0);
    \strand[ultra thick] (4.6,0) to[curve through={(5.6,1)..(4.6,2)..(3.6,1)}] (4.6,0);
    \strand[ultra thick] (6.2,0) to[curve through={(7.2,1)..(6.2,2)..(5.2,1)}] (6.2,0);
    \strand[ultra thick] (7.8,0) to[curve through={(8.8,1)..(7.8,2)..(6.8,1)}] (7.8,0);
    \strand[ultra thick] (9.4,0) to[curve through={(10.4,1)..(9.4,2)..(8.4,1)}] (9.4,0);
    \strand[ultra thick] (3,-1.6) to [curve through = {(4, -.6)..(3,.4)..(2,-.6)}] (3,-1.6);
    \strand[thick, blue] (1.2,1) to[curve through ={(1.8,.2)..(2.35, .7)..(1.6,1)..(2.35,1.3)..(1.8,1.8)}] (1.2,1);
    \node[blue] at (1.3,2) {$\Large C^*$};
    \node[blue] at (.6,1) {$\Large -3/2$};
    \node at (3,2.2) {$\Large -4$};
    \node at (4.6, 2.2) {$\Large -2$};
    \node at (6.2,2.2) {$\Large -2$};
    \node at (7.8,2.2) {$\Large -2$};
    \node at (9.4, 2.2) {$\Large -2$};
    \node at (2.1, -1.6) {$\Large -2$};
    \end{knot}
    \end{tikzpicture}
    \caption{After modifying the surgery diagram by adding a $-2$-framed meridian to $C^*$ and changing the framing of $C^*$ to $-2$, a sequence of blow-ups and blow-downs takes the surgery diagram to the negative definite $E_8$ plumbing. Thus, there is a knot $C \subset \mathcal P$ with a half-integer surgery to $L(27,16)$ whose surgery dual is $C^*$.}
\end{figure}
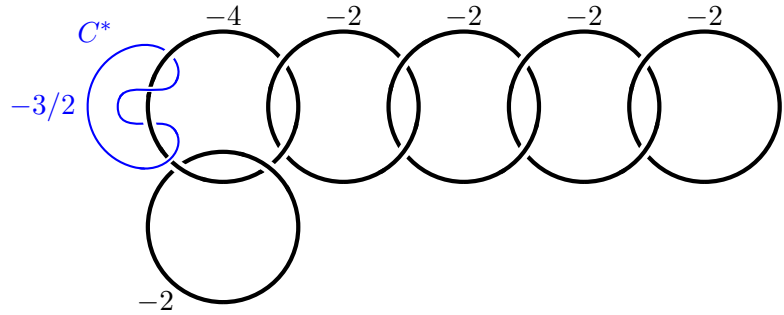

\begin{thm}\label{thm:main}
The two knots $E$ and $C$ are the only knots in $\mathcal P$ with half-integer lens space surgeries. 
\end{thm}

We point out that $E(7/2) \cong L(7,6)$ and $C(27/2) \cong L(27,16)$. It follows that $14$-surgery on the $(7,2)$-cable of $E$ is homeomorphic to $L(7,6)\#L(2,1)$, and that $54$-surgery on the $(27,2)$-cable of $C$ is homeomorphic to $L(27,16)\#L(2,1)$. In fact, the cabling construction ensures that the $2p$-surgery on the $(p,2)$-cable of a knot $K$ is homeomorphic to $K(p/2) \# L(2,1)$. In order to obstruct half-integer lens space surgeries, we turn our attention to obstructing integer surgeries to 3-manifolds of the form $L(p,q) \# L(2,1)$. The upshot of this is twofold: first, our lattice embedding obstruction does not deal directly with non-integer surgeries, though some modifications may be made to make it amenable to such surgeries, as Gibbons did with $p/q$-changemakers in \cite{Gib}, and second, any knot $K$ where $K(2p)$ is homeomorphic to $L(p,q) \# L(2,1)$ must satisfy $2p\geq 2g(K)$, so in fact $K$ is a cable knot. Moreover, $K$ must be the $(p,2)$-cable of knot whose $p/2$-surgery yields $L(p,q)$, since $2/p$-surgery certainly never yields a lens space for $p > 2$--then, the genus of the knot would have to be $0$, but then the knot is contained in a $3$-ball, thus the exterior of the knot is reducible. En route to proving Theorem \ref{thm:main}, we prove the following.

\begin{thm}\label{thm:mainrephrased}
If $K(2p) \cong L(p,q) \# L(2,1)$, then $(p,q) \in \{(27,16), (7,6)\}$ and $g(K) = p$.
\end{thm}

Recall that $\mathcal P$ is the branched double cover of the pretzel knot $P(-2, 3, 5)$, and that the lens space $L(p,q)$ is the branched double cover of the 2-bridge knot $K(p,q)$. Recall also that, by the Montesinos trick, if the knot $K_1 \subset S^3$ may be obtained from the knot $K_0$ by changing a crossing in a planar diagram for $K_0$, then there is a knot in the branched double of $K_0$ that admits a half-integer surgery to the branched double cover of $K_1$. We immediately obtain the following corollary of Theorem \ref{thm:main}.

\begin{cor}
    The only 2-bridge knots in $S^3$ that admit crossing changes to $P(-2,3,5)$ are $K(7,6)$ and $K(27,16)$.\qed
\end{cor}

When paired with the Finite/Cyclic Surgery Theorem of Boyer--Zhang \cite{BZ96}, Theorem \ref{thm:main} gives a remarkable corollary characterizing knots in $\mathcal P$ with non-integral lens space surgeries. The following is a specialization of the Finite/Cyclic Surgery Theorem relevant to the case at hand.

\begin{thm}[Theorem 1.1 (2) of \cite{BZ96}]\label{thm:FST}
Let $K\subset \mathcal P$ and suppose that $K(p/q)$ is a lens space. If $q\geq 2$, then exactly one of the following holds:
\begin{enumerate}
    \item $K$ is an exceptional fiber in a Seifert fibering of $\mathcal P$;
    \item $K$ is a cable of an exceptional fiber in a Seifert fibering of $\mathcal P$; or
    \item the interior of $M_K$ admits a hyperbolic metric with respect to which $\partial M_K$ is totally geodesic and $q = 2$. \qed
\end{enumerate}
\end{thm}

In the case of (3), the knot $K$ is said to be \emph{hyperbolic}. In light of Theorem \ref{thm:FST}, we obtain the following corollaries.

\begin{cor}\label{cor:integral}
Lens space surgeries on hyperbolic knots in $\mathcal P$ are integral.
\end{cor}

\begin{proof}
By Theorem \ref{thm:FST}, a non-integer lens space surgery on a hyperbolic knot must be a half-integer surgery. By Theorem \ref{thm:main}, the only knots in $\mathcal P$ with half-integer lens space surgeries are $E$, which is the exceptional fiber of order $3$, and $C$, which is a cable of an exceptional fiber. Neither of these knots is hyperbolic. 
\end{proof}

\begin{cor}
Let $K\subset \mathcal P$ and suppose that $K(p/q)$ is a connected sum of exactly two lens spaces. If $p/q>2g(K)-1$, then $K$ is either an exceptional fiber, or a once- or twice-iterated cable thereof. 
\end{cor}
\begin{proof}
By Lemma 14 of \cite{Cau21}, if $K(p/q)$ is reducible, $q\geq 2$, and $p/q > 2g(K)-1$, then $K$ is an exceptional fiber. Otherwise, we have $q =1$ and $p\geq 2g(K)$, and therefore $K$ is a cable knot. Together, Theorems \ref{thm:main} and \ref{thm:FST} imply that $K$ is either a cable of an exceptional fiber or a cable of a cable of an exceptional fiber. 
\end{proof}

Interestingly enough, the strategy we use to prove Theorem \ref{thm:mainrephrased} cannot on its own be used to characterize the knots $K \subset \mathcal P$ where $K(2g(K)-1)$ is a connected sum of exactly two lens spaces. We are at present unable to to produce such a knot, nor are we able to prove that no such knot exists. However, we deduce an obstruction toward that end.

\begin{thm}\label{thm:twogminusonereducible}
Let $K \subset \mathcal P$, $p = 2g(K)-1$, and suppose that $K(p)$ is a connected sum of exactly two lens spaces. Then there is a torus knot, or a cable thereof, $T \subset S^3$ with $K(p) \cong T(p)$, and
\[\Delta_K(t) = \Delta_T(t) -(t^{(p-1)/2} + t^{-(p-1)/2}) + (t^{(p+1)/2} + t^{-(p+1)/2}).\]
\end{thm}

\subsubsection{Integer surgeries.}

In \cite{Ber}, Berge produced an elegant construction of lens space surgeries on knots in $S^3$. Let $K \subset S^3$ lie on a genus two Heegaard surface $\Sigma$, and suppose that $K$ represents a \emph{primitive} element in the fundamental group of each of the two handlebodies bounded by $\Sigma$, in which case $K$ is said to be \emph{doubly primitive}. The surface $\Sigma$ induces an integral slope on $\partial M_K$: $\Sigma \cap \partial M_K$ consists of two anti-parallel curves that represent the same slope $p$. Berge showed that the result of $p$-surgery on $K$ is a lens space and tabulated a list of doubly primitive knots in $S^3$, the specific collection of which are known as the \emph{Berge knots}. The \emph{Berge conjecture} posits that every integer surgery on a knot in $S^3$ arises from this construction. This conjecture has received considerable attention since Berge's observation more than thirty years ago. At the time of this writing, one of the furthest strides toward proving the Berge conjecture is the resolution of the \emph{lens space realization problem} \cite{Gre13}, where Greene proved the following.

\begin{thm}[Theorem 1.2 of \cite{Gre13}]
If $K \subset S^3$, $p$ is a positive integer, and $K(p) \cong L(p,q)$, then there is a Berge knot $B \subset S^3$ such that $B(p) \cong L(p,q)$ and $K$ has the same knot Floer homology as $B$. \qed
\end{thm}

As Berge noted, it is often convenient to take the perspective of the surgery dual to a doubly primitive knot $K^* \subset K(p) \cong L(p,q)$, the core of the surgery solid torus. Following convention, we refer to the dual of a Berge knot also as a Berge knot, though the ambient manifold serves to prevent confusion. Berge showed in \cite[Theorem 2]{Ber} that $K^*$ is a \emph{simple} knot in $L(p,q)$---that is, letting $H_0 \cup_{T^2} H_1$ denote a genus one Heegaard of splitting of $L(p,q)$, $K^*$ admits an isotopy in $L(p,q)$ such that $K^*\cap H_i$ is contained in a compression disk for $H_i$, $i \in \{0,1\}$, and there is a unique such knot $K^*$ in each homology class in $L(p,q)$. Thus, a Berge knot in a lens space is determined by its homology class.  

Of course, one readily observes that Berge's construction is not unique to $S^3$, and it is readily adapted to produce lens spaces surgeries on knots in any $3$-manifold admitting a genus two Heegaard splitting---consider, for instance, the Poincar\'e homology sphere. In \cite{Tan09}, Tange produced a list of simple knots in lens spaces admitting integer surgeries to $\mathcal P$ that we call \emph{Tange knots}, which, by the same ambiguity as with Berge knots, we also call their surgery duals in $\mathcal P$, which are doubly-primitive. In \cite{Tan10}, Tange made partial progress towards proving that any lens space $L(p,q)$ realized by integer surgery on a knot $K \subset \mathcal P$ with $2g(K)\leq p$ is realized by surgery on a Tange knot. Here, we take up the strategy suggested in the remark following \cite[Conjecture 1.10]{Gre13}, to prove the following.

\begin{thm}\label{thm:integerrealization}
    If $K(p)\cong L(p,q)$ and $p\geq 2g(K)$, then there is a Tange knot $T$ such that $T(p)\cong L(p,q)$ and $K$ has the same knot Floer homology as $T$. 
\end{thm}

We believe that one may furthermore show, following the argument in \cite[Section 2]{Gre13}, that the surgery duals of $K$ and $T$ in the statement of the theorem are homologous in $H_1(L(p,q))$, though we omit details in this work. 

A simple knot $K \subset L(p,q)$ is said to be \emph{Floer simple}, in the sense that $\rk(\widehat{HFK}(L(p,q),K) = \rk\widehat{HF}(L(p,q)) = p$. In \cite{Hed}, Hedden constructed a knot $T_L$ (and its mirror image $T_R$) in each lens space $L(p,q)$ with the property that $\rk \widehat{HFK}(L(p,q),T_L) = p + 2$. In \cite{Ras07}, Rasmussen showed that the rank of the knot Floer homology of a knot $K \subset L(p,q)$ with an integer surgery to an L-space homology sphere is at most $p + 2$, and he furthermore certified that, up to orientation reversal, $\mathcal P$ is the only L-space homology sphere realized by surgery on a Hedden knot in $L(p,q)$ for $p \leq 38$. In \cite{Bak20}, Baker also constructed an infinite family of knots in $\mathcal P$ with integer lens space surgeries. Baker's construction identifies band surgeries that take the pretzel knot diagram $P(-2,3,5)$ to a two-bridge knot; by the Montesinos trick, these band surgeries lift to integer lens space surgeries on knots in $\mathcal P$. Notably, all but finitely many of the knots arising from Baker's construction are tunnel number two, and therefore none of these are doubly primitive, and, upon further inspection, none of the lens spaces realized by surgery on Baker knots are realized by surgery on either Tange or Hedden knots. While the lens space surgery slopes of Hedden and Baker knots were known at the time of their construction, we remark that a corollary of Theorem \ref{thm:integerrealization} is that if $K \subset \mathcal P$, $K(p) \cong L(p,q)$, and $L(p,q)$ is not surgery on a Tange knot (e.g. $K$ is a Hedden or a Baker knot), then $p = 2g(K)-1$. Obstructing $2g(K)-1$ lens space surgeries falls outside of the purview of $E_8$-changemakers, though we obtain an alternate proof of \cite[Theorem 8]{Bak20}.

\begin{thm}\label{2gminusoneirreducible}
    If $K(p) \cong L(p,q)$ for $p = 2g(K)-1$, then there is a Berge knot $B\subset S^3$ such that $B(p) \cong L(p,q)$ and 
    \[\Delta_K(t) = \Delta_B(t) -(t^{(p-1)/2} + t^{-(p-1)/2}) + (t^{(p+1)/2} + t^{-(p+1)/2}).\]
\end{thm}

\subsection{Changemakers.}

In an influential pair of papers, Greene observed that by combining Donaldson's Diagonalization Theorem and the data of Heegaard Floer \emph{correction terms} (or \emph{d-invariants}), one may address the topological problem of realizing, for example, a lens space as surgery on a knot in $S^3$ by means of a combinatorial heuristic. This heuristic, called a \emph{changemaker lattice embedding}, is derived from the relationship between the Alexander polynomial of a knot in $S^3$: 
\begin{equation}
    \Delta_K(T) = \sum_{i = 0}^\infty a_i (T^i + T^{-i}),
\end{equation} the \emph{torsion coefficients} of $K$:
\begin{equation}
    t_i(K) = \sum_{j = 1}^\infty j\cdot a_{|i|+j},
\end{equation} the correction terms of $K(p)$: 
\begin{equation}
    \{d(K(p), \mathfrak t) \colon\mathfrak t \in \text{Spin}^c(K(p))\},
\end{equation}and the lengths of \emph{characteristic elements} in a suitably chosen negative-definite unimodular lattice $L$:
\begin{equation}
    \text{Char}(L) = \{\mathfrak c \in L \colon \langle \mathfrak c, v\rangle \equiv \langle v, v\rangle \ \mod 2 \text{ for all } v \in L\}.
\end{equation}
In the case of studying lens space surgeries on knots in $S^3$, Greene uses Donaldson's Diagonalization theorem to pin down the lattice $L$ precisely as $-\mathbb Z^{n+1}$, in which case: 
\begin{equation}
    \text{Char}(-\mathbb Z^{n+1}) =  \Big\{\sum_{i= 0}^n\mathfrak c_id_i \colon \mathfrak c_i \equiv 1 \ \mod 2\Big\}
\end{equation} for $\{d_0, \ldots, d_n\}$ an orthonormal basis of $-\mathbb Z^{n+1}$. Greene showed that if $L(p,q)$ is integer surgery on a knot in $S^3$, then the canonical \emph{linear lattice} $\Lambda(p,q)$ is the orthogonal complement to some vector $\sigma \in - \mathbb Z^{n+1}$ with the following curious property.

\begin{defin}\label{def:changemaker}
A vector $\sigma = (\sigma_0, \ldots, \sigma_n) \in -\mathbb Z^{n+1}$ is said to be a \emph{changemaker} if $0\leq \sigma_0 \leq \cdots \leq \sigma_n$, and for all $1\leq i\leq n$
\[\sigma_i \leq 1 + \sum_{j = 0}^{i-1}\sigma_j.
\]
This is equivalent to and derived from the condition that 
$$\{\langle \mathfrak c, \sigma\rangle \colon \mathfrak c \in \{\pm 1\}^{n+1}\} = \{j \in [-|\sigma|_1, |\sigma|_1] \colon j \equiv |\sigma|_1 \ \mod 2\},$$
where $|\sigma|_1$ denotes the $L_1$-norm of $\sigma$.
\end{defin}

\begin{rmk}
    The latter of the two conditions in Definition \ref{def:changemaker} is what generalizes to $E_8$.
\end{rmk}

More precisely, Greene showed the following.

\begin{thm}\label{thm:changemakerembedding}
Let $K \subset S^3$ and suppose that $K(p) \cong L(p,q)$. Then $\Lambda(p,q) \cong (\sigma)^\perp$ for some changemaker $\sigma \in \mathbb Z^{n+1}$. Moreover, 
\[2g(K) = p- |\sigma|_1,\]
and if $K'(p) \cong L(p,q)$, then $\widehat{HFK}(K')\cong \widehat{HFK}(K)$.
\end{thm}
\begin{proof}[Sketch of proof]
Suppose that $K(p) \cong L(p,q)$. Let $W$ be the orientation reversal of the trace cobordism of $p$-surgery on $K$, and let $X:=X(p,q)$ be the \emph{linear} negative definite 4-manifold bounded by $L(p,q)$ presented by the Kirby diagram in Figure 4. Form the oriented, negative definite 4-manifold $Z$ by gluing $W$ to $X$ along $L(p,q)$. Since $\partial Z = S^3$, Donaldson's Theorem gives us an embedding $Q_W \oplus Q_X \hookrightarrow -\mathbb Z^{n+1}$ for $n = \text{rk} Q_X$. Writing the image of the generator of $Q_W$ under this embedding as $\sigma = (\sigma_0, \ldots, \sigma_n)$, we have that
\begin{equation}\label{eq:pairingwithtorsion}
    \mathfrak c^2 + n+1 \leq -8t_i(K),
\end{equation}
for all $|i|\leq p/2$ and any characteristic element $\mathfrak c$ in $-\mathbb Z^{n+1}$ with $\langle c, \sigma\rangle + p \equiv 2i \ \mod 2p$. Furthermore, for all $|i|\leq p/2$, there is some $\mathfrak c$ attaining equality in (\ref{eq:pairingwithtorsion}). 

That $\sigma$ is a changemaker is derived from (\ref{eq:pairingwithtorsion}), following the preliminary observation that, for $K$ an L-space knot, the sequence $(t_0(K), t_1(K), \ldots )$ is a non-increasing sequence of non-negative integers with $t_i(K) = 0$ if and only if $i\geq g(K)$. Therefore, \[2g(K) = p - \max\{\langle \mathfrak c, \sigma \rangle \colon \mathfrak c^2 = -(n+1)\} = p - |\sigma|_1.\] Similarly, we may compute $g^{(i)}(K) = \min\{j\geq 0 \colon t_j(K) = i\}$ for $i \in \{0, \ldots, t_0(K)\}$ by the formula
\[2g^{(i)}(K) = p - \max\{\langle \mathfrak c, \sigma\rangle \colon \mathfrak c^2 = -(n+1) -8i\},\]
and thereby completely recover $\Delta_K(T)$ from $\sigma$. The knot Floer homology of an L-space knot in an integer homology sphere L-space is completely determined by $\Delta_K(T)$, and therefore if $K(p) \cong K'(p) \cong L(p,q)$, then $\widehat{HFK}(K) \cong \widehat{HFK}(K')$ as bigraded abelian groups. \end{proof}

It follows that if for $K \subset S^3$ we have $K(p) \cong L(p_1,q_1)\# L(p_2, q_2)$, then $p \geq 2g(K)$, and therefore $K$ is a cable of knot $\kappa \subset S^3$ admitting a non-integral surgery to a lens space. Our proof of Theorems \ref{thm:twogminusonereducible} and \ref{2gminusoneirreducible} rests on the following Proposition (cf. \cite[Theorem 3]{Cau21}).

\begin{prop}\label{prop:twogminusoneembedding}
Let $K \subset \mathcal P$ and $p = 2g(K)-1$. If $K(p) \cong L(p_1,q_1)\#L(p_2,q_2)$, then the intersection form of the boundary connected sum of $X(p_1,q_1)$ and $X(p_2,q_2)$ is the orthogonal complement to a changemaker $\sigma\in -\mathbb Z^{n+1}$. \qed
\end{prop}

\begin{proof}[Proof of Theorems \ref{thm:twogminusonereducible} and \ref{2gminusoneirreducible}]
If $K(p)$ is a lens space or the connected sum of exactly two lens spaces, then let $\sigma \in -\mathbb Z^{n+1}$ be the changemaker associated to the surgery. As observed in \cite{Gre15, Gre13, McC,Cau21}, $\sigma$ determines $\Delta_K(t)$. Now, if $K(p)$ is the connected sum of a pair of lens spaces, the resolution of the cabling conjecture for connected sums of lens spaces in \cite{Gre15} gives us a knot $K'\subset S^3$ such that $K'(p) \cong K(p)$ and $K'$ is either a torus knot or a cable of a torus knot. On the other hand, if $K(p) \cong L(p,q)$, then the resolution of the lens space realization problem in \cite{Gre13} gives us a knot $K'\subset S^3$ such that $K'(p) \cong L(p,q)$ and $K'$ is a Berge knot. In both cases, the changemaker $\sigma$ determines $\Delta_{K'}(t)$, and by the assumption that $p = 2g(K) -1$, we see that 
$$\Delta_K(t) = \Delta_{K'}(t) -(t^{(p-1)/2} + t^{-(p-1)/2}) + (t^{(p+1)/2} + t^{-(p+1)/2}).$$\end{proof}

\subsection{$E_8$-changemakers.} We take the changemaker lattice construction as inspiration in devising an obstruction to realizing an L-space bounding a \emph{sharp} $4$-manifold as $p\geq 2g(K)$ surgery on a knot $K\subset \mathcal P$. In Section \ref{sec:workingin}, we develop the following notion of an $E_8$-\emph{changemaker}---the appropriate generalization of a changemaker in the lattice $-E_8 \oplus -\mathbb Z^{n+1}$---and show that if $p$-surgery on a knot $K \subset \mathcal P$ is an L-space that bounds a sharp (Definition \ref{def:sharp}) 4-manifold $X$, then $Q_X$ is isomorphic to the orthogonal complement to some $E_8$-changemaker in $-E_8\oplus -\mathbb Z^{n+1}$ for $n =b_2(X) - 8$.

For ease of notation, we now define the notion of a \emph{parity interval}.

\begin{defin}\label{def:parityinterval}
    Let $a$ and $k$ be integers with $k \geq 0$. The \emph{parity interval} $PI(a, a + 2k)$ is the set of integers $\{a + 2j \colon 0\leq j \leq k\}$. 
\end{defin}

Recall that if $L$ is a negative-definite, unimodular lattice and $\mathfrak c \in \text{Char}(L)$, then $\langle \mathfrak c, \mathfrak c\rangle \equiv \rk L \ \mod 8$. Define $m(L) = \max \{\langle \mathfrak c, \mathfrak c\rangle \colon \mathfrak c \in \text{Char}(L)\}$. We say that a characteristic vector $\mathfrak c\in L$ is \emph{short} if $\langle \mathfrak c, \mathfrak c\rangle = m(L)$, and \emph{nearly short} if $\langle \mathfrak c, \mathfrak c\rangle = m(L) - 8$. Denote the sets of short and nearly short characteristic vectors in $L$ by $\text{short}(L)$ and $\text{Short}(L)$, respectively. For a vector $v \in L$, let $c(v) := \max\{\langle \mathfrak c, v\rangle \colon \mathfrak c \in \text{short}(L)\}$ and let $C(v) := \max\{\langle \mathfrak c, v\rangle \colon \mathfrak c \in \text{Short}(L)\}$.

We are now ready to define the notion of an $E_8$-\emph{changemaker}.

\begin{defin}\label{def:e8changemaker}
    A vector $\tau = (s, \sigma) \in -E_8 \oplus -\mathbb Z^{n+1}$ is said to be an $E_8$-changemaker if 
    \begin{enumerate}
        \item $PI(-c(\tau), c(\tau)) = \{\langle \mathfrak c, \tau\rangle \colon \mathfrak c \in \text{short}(-E_8 \oplus -\mathbb Z^{n+1})\}$, and
        \item $PI(c(\tau) + 2, C(\tau)) \subset \{\langle \mathfrak c, \tau\rangle \colon \mathfrak c \in \text{Short}(-E_8 \oplus -\mathbb Z^{n+1})\}$.
    \end{enumerate}
    If $L \cong (\tau)^\perp$ for some $E_8$-changemaker $\tau$, then $L$ is said to be an $E_8$-\emph{changemaker lattice}.
\end{defin}

The combinatorial constraints defining the $E_8$-changemaker condition naturally extend those defining the changemaker condition. In fact, writing $\tau = (s, \sigma) \in -E_8 \oplus -\mathbb Z^{n+1}$ and noting that $\text{short}(-E_8 \oplus - \mathbb Z^{n+1}) = \text{short}(-E_8)\oplus\text{short}(-\mathbb Z^{n+1}) = \{0\} \oplus \{\pm 1\}^{n+1}$, in order for $\tau$ to be an $E_8$-changemaker the first thing one finds is that $\sigma$ is a changemaker. Furthermore, the $E_8$-changemaker $\tau$ associated to a surgery $K(p)$ determines $\widehat{HFK}(K)$, and in particular $g(K)$, in much the same way that a changemaker $\sigma$ arising from surgery determines the associated knot's Floer homology groups.

Our lattice embedding obstruction is derived from the following theorem.

\begin{thm}\label{thm:hardeightembedding}
Let $K$ be an L-space knot in $\mathcal P$, and suppose that $p \geq 2g(K)$. If $K(p)$ bounds a sharp $4$-manifold with no torsion in $H_1(X)$, then $Q_X$ embeds as the orthogonal complement to an $E_8$-changemaker $\tau = (s,\sigma)$ and $\widehat{HFK}(\mathcal P, K)$ is determined by $\tau$. In particular, $2g(K) = p - |\sigma|_1$.
\end{thm}

We then develop the theory of $E_8$-changemaker lattices in order to implement the obstruction implicit in Theorem \ref{thm:hardeightembedding}, and arrive at the following two theorems.

\begin{thm}\label{thm:twosummandhardeightembedding}
Suppose that $\Lambda(2,1)\oplus \Lambda(p,q) \cong (\tau)^\perp$ for some $E_8$-changemaker $\tau$. Then $(p,q) = (7,6)$ or $(27,16)$. Furthermore, if $K(2p) \cong L(2,1) \# L(p,q)$, then $g(K) = p$.
\end{thm}

\begin{thm}\label{thm:integersurgeryrealization}
If $\Lambda(p,q) \cong (\tau)^\perp$ for some $E_8$-changemaker $\tau$, then there is a Tange knot $T$ such that $T(p) \cong L(p,q)$. 
\end{thm}

\begin{proof}[Proof of Theorem \ref{thm:integerrealization}]
Let $K \subset \mathcal P$, and suppose that $p\geq 2g(K)$ and $K(p)\cong L(p,q)$. Then $K(p)$ bounds a sharp 4-manifold $X := X(p,q)$, and Theorem \ref{thm:hardeightembedding} implies that $Q_X$ embeds as the orthogonal complement to some $E_8$-changemaker $\tau$ that completely determines $\widehat{HFK}(\mathcal P, K)$. By Theorem \ref{thm:integersurgeryrealization}, there is a Tange knot $T \subset \mathcal P$ such that $T(p) \cong L(p,q)$, so $\widehat{HFK}(\mathcal P, K) \cong \widehat{HFK}(\mathcal P, T)$.    
\end{proof}

In the current work, we demonstrate Theorem \ref{thm:integersurgeryrealization} explicitly only for $\Lambda(p,q)$ with rank $\geq 10$, but contend that a computer aided search of the 1003 non-zero $E_8$-changemakers in $-E_8 \oplus -\mathbb Z^n$ with $n \in \{0, 1, 2\}$ demonstrates the theorem to be true in the case that $7\leq \rk \Lambda(p,q) \leq 9$. There are two reasons for this omission: first, the technical lemmas we use to characterize $E_8$-changemaker embeddings of linear lattices when $n\geq 3$ break down when the diagonal summand is small, so this case requires a different approach than the general one we take; second, Rasmussen, according to the remark at the end of Section 6.3 of \cite{Ras07}, has used a computer to verify that all lens spaces $L(p,q)$ with $p\leq 100,000$ realized by integer surgery on a knot $K\subset \mathcal P$ with $2g(K)\leq p$ are realized by integer surgery on a Tange knot; by explicit computation, the author has found that linear $E_8$-changemaker lattices $\Lambda(p,q)$ with $n \in \{0,1,2\}$ satisfy $p\leq 100,000$. 

\subsection{Beyond $E_8$-changemakers.} The utility of changemakers and $E_8$-changemakers in both constructing and obstructing lens space surgeries from $S^3$ and $\mathcal P$, respectively, suggests that appropriate generalizations of their defining combinatorics may give way to tractable strategies for understanding the more general phenomenon of when Dehn surgery on a knot $K$ in an arbitrary integer homology sphere $Y$ yields the lens space $L(p,q)$. In the case that $Y$ is an L-space with $d(Y) \geq 4$, then, by considering the 4-manifold $Z$ obtained by gluing the trace cobordism from $L(p,q)$ to $Y$ to $X(p,q)$ along $L(p,q)$ in light of data from Floer homology (cf. Lemma \ref{lem:generalized_changemaker}) we are led to consider vectors $\tau$ in negative definite, unimodular lattices of the form $\Lambda \cong -L \oplus -\mathbb Z^{n+1}$, where $L$ has no vectors of length $1$, with the property that 
$$\mathfrak c^2 + \rk(\Lambda) - 4d(Y) \leq -8t_i(K)$$
for all $|i|\leq p/2$ and $\mathfrak c \in \text{Char}(\Lambda)$ such that $\langle \mathfrak c, \tau\rangle + p \equiv 2i \ \mod 2p$, and moreover over for all $|i|\leq p/2$ there is some $\mathfrak c$ which attains equality. If $p \geq 2g(K)$, then we are led to the following definition.

\begin{defin}
    Let $L$ be a positive definite unimodular lattice with no vectors of norm 1, and let $\Lambda = -L \oplus -\mathbb Z^{n}$. Let $$m(\Lambda) = \max\{\langle \mathfrak c,\mathfrak c\rangle \colon \mathfrak c \in \text{Char}(\Lambda)\}.$$ Denote by Char$_i(\Lambda)$ the characteristic elements in $\Lambda$ of norm $m(\Lambda) - 8i$, and for a vector $v \in \Lambda$, let $C_i(v)$ denote $\{\langle \mathfrak c, v\rangle\colon c \in \text{Char}_i(\Lambda)\}$, and let $c_i(v) = \max(C_i(v))$. A vector $\tau \in \Lambda$ is said to be an $L$\emph{-changemaker} if 
\begin{enumerate}
    \item $PI(-c_0(\tau), c_0(\tau)) = C_0(\tau)$, and
    \item $PI(c_{i-1}(\tau) + 2, c_i(\tau)) \subset C_i(\tau)$ for $1 \leq i \leq \frac{\rk(\Lambda) + m(\Lambda)}{8}$. 
\end{enumerate}
A lattice $\mathcal L$ is said to be an $L$\emph{-changemaker lattice} if $\mathcal L \cong (\tau)^\perp$ for some $L$-changemaker $\tau$.
\end{defin} 
\noindent We point out that the definition above coincides with Greene's definition of a changemaker if $L$ is taken to be the trivial lattice, i.e. the lattice of rank 0, and that it agrees with the above definition of an $E_8$-changemaker if $L$ is taken to be $E_8$.

At present, the only known irreducible integer homology sphere L-spaces with non-negative $d$-invariant are $S^3$ and $\mathcal P$. One conceivable strategy for producing a novel irreducible integer homology sphere L-space by way of $L$-changemakers would proceed as follows. First, fix a positive definite unimodular lattice $L$. Next, one may work out the specific combinatorics of the $L$-changemaker condition, as we do here for $L \cong E_8$. One may then attempt to identify a linear lattice $\Lambda(p,q)$ which embeds in $-L\oplus -\mathbb Z^n$ as the orthogonal complement to an $L$-changemaker $\tau$. If there is such a linear lattice, it is conceivable that there is an integer homology sphere L-space $Y$ which bounds a 4-manifold $Z$ with $Q_Z \cong -L\oplus-\mathbb Z^n$ obtained as surgery on a knot in the lens space $L(p,q)$. The data of the embedding $(\tau) \oplus \Lambda \hookrightarrow -L\oplus -\mathbb Z^n$ has the potential to produce $Y$ as in the following example. 

The author discovered the $E_8$-changemaker embedding $\Lambda(2,1) \oplus \Lambda(27, 16) \hookrightarrow -E_8$ prior to discovering the knot $C$; in fact, the explicit embedding of $\Lambda(2,1) \oplus \Lambda(27,16)$ in $-E_8$ as the sublattice generated by $\mathcal S = \{e_3, e_2, e_1-e_4, e_5, e_6, e_7, e_8\}$ was used together with the fact that $E_8$ is generated by $\mathcal S \cup \{e_4\}$ to produce a Kirby diagram of $\mathcal P$ realizing $\mathcal P$ as surgery on a knot in $L(2,1) \# L(27,16)$ (see Figure 3). 

\begin{figure}\label{fig:212716}
    \centering
    \begin{tikzpicture}
    \begin{knot}[clip width = 3, flip crossing = 2, flip crossing = 3, flip crossing = 5, flip crossing = 7, flip crossing = 10, flip crossing = 12, flip crossing = 14, flip crossing = 15]
    \strand[ultra thick] (3,0) to[curve through={(4,1)..(3,2)..(2,1)}] (3,0);
    \strand[ultra thick] (4.6,0) to[curve through={(5.6,1)..(4.6,2)..(3.6,1)}] (4.6,0);
    \strand[ultra thick] (6.2,0) to[curve through={(7.2,1)..(6.2,2)..(5.2,1)}] (6.2,0);
    \strand[ultra thick] (7.8,0) to[curve through={(8.8,1)..(7.8,2)..(6.8,1)}] (7.8,0);
    \strand[ultra thick] (9.4,0) to[curve through={(10.4,1)..(9.4,2)..(8.4,1)}] (9.4,0);
    \strand[ultra thick] (3,-1.6) to [curve through = {(4, -.6)..(3,.4)..(2,-.6)}] (3,-1.6);
    \strand[ultra thick, blue] (1.2,1) to[curve through ={(1.8,.2)..(2.35, .7)..(1.65,1)..(2.35,1.3)..(1.8,1.8)}] (1.2,1);
    \strand[ultra thick] (-.6,1) to[curve through = {(.4,2)..(1.4,1)..(.4,0)}] (-.6,1);
    \node at (.4, 2.2) {$-\Large 2$};
    \node[blue] at (1.78,2.2) {$\Large -2$};
    \node at (3,2.2) {$\Large -4$};
    \node at (4.6, 2.2) {$\Large -2$};
    \node at (6.2,2.2) {$\Large -2$};
    \node at (7.8,2.2) {$\Large -2$};
    \node at (9.4, 2.2) {$\Large -2$};
    \node at (2.1, -1.6) {$\Large -2$};
    \end{knot}
    \end{tikzpicture}
    \caption{A sequence of blow-ups and blow-downs takes this Kirby diagram to the negative definite $E_8$ plumbing. The black unknots correspond to the elements of $\mathcal S$, while the blue unknot corresponds to the simple root $e_4$.}
\end{figure}
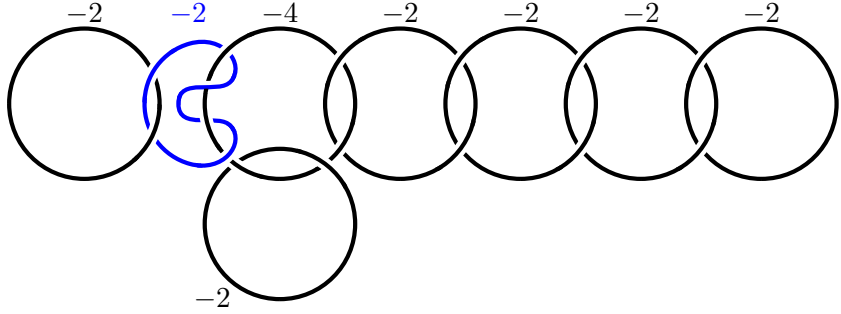

It is conceivable that there are negative-definite unimodular lattices $L$ such that no $L$-changemaker lattice is isomorphic to a sum of linear lattices. In fact, if one can show that no $L$-changemaker lattice is isomorphic to a sum of linear lattices for all lattices $L$ with $\frac{\rk(L) + m(L)}{4} = d$ then one will have shown that no integer homology sphere L-space $Y$ with $d(Y) = d$ may be obtained by surgery on a knot $K$ in a lens space $L(p,q)$ with $p \geq 2g(K)$. We issue the following conjecture, in the spirit of \cite[Conjecture 1]{Ras07}.

\begin{conj}
If $L$ is a definite unimodular lattice and $\Lambda(p,q)$ is an $L$-changemaker lattice, then $L(p,q)$ is realized by integer surgery on a knot enumerated by Berge or Tange. In particular, $L$ is the trivial lattice or $L \cong -E_8$.
\end{conj}

On the other hand, one may wonder if some appropriate generalization of the changemaker condition may shed light on surgeries from lens spaces to a non-L-space integer homology sphere $Y$. Every Heegaard genus two integer homology sphere admits a lens space surgery by the doubly primitive construction---might the bootstrapping together of Floer homology and lattice embeddings entirely capture these realization problems as it has for $S^3$ and $\mathcal P$? Since $Y$ is not an L-space, a reexamination of the input from the surgery exact triangle in Heegaard Floer homology that leads to the $L$-changemaker definition is in order, as the following example demonstrates.

Consider the lens space $L(46,15)$, which admits a surgery description in terms of a linear chain of $15$ unknots, all of which are $-2$-framed except for, say, the first one, which bears a $-4$-framing. Then, $-2$-surgery with respect to the blackboard framing on the meridian of the third unknot produces a Seifert-fibered integer homology sphere $Y$ that is not an L-space that bounds the even unimodular lattice of rank 16 called $\Gamma_{16}$. We assert, without supplying proof here, that the obvious embedding $\Lambda(46,15) \hookrightarrow \Gamma_{16}$ implicit in the surgery diagram is not a $\Gamma_{16}$-changemaker embedding. It is important to note that this does not rule out the existence of a $\Gamma_{16}$-changemaker embedding of $\Lambda(46,15)$, though the author doubts that one exists.

\subsection{Organization.}

In Section \ref{sec:lattices}, we recall preliminary observations about lattices recorded in Section 3 of \cite{Gre13}. In Section \ref{sec:inputfromfloer}, we recall input from Floer homology begun in \cite{Cau21} to address $(2g-1)$-surgeries, and bring it to bear on $p \geq 2g$-surgeries, culminating in the notion of an $E_8$-changemaker. In the first part of Section \ref{sec:workingin}, we discuss some essentials required to take a hands-on approach to working with the $E_8$ lattice that are more familiar to number theorists and lattice theorists, as communicated to the author by Daniel Allcock and Richard Borcherds, and that the author suspects are less familiar to low-dimensional topologists. In the second part of Section \ref{sec:workingin}, we develop the input from Floer homology in Section \ref{sec:inputfromfloer} in order to derive combinatorial constraints on $E_8$-changemakers and implement our lattice embedding obstruction. At the end of Section \ref{sec:workingin}, we record some basic observations about linear lattices that are $E_8$-changemaker lattices. In Section \ref{sec:identifying}, we implement the $E_8$-changemaker obstruction to characterize when a linear lattice is realized by the orthogonal complement to an $E_8$-changemaker in $-E_8 \oplus -\mathbb Z^{n+1}$ with $n \geq 2$. In Section \ref{sec:main}, we explicitly identify each of the thirty-eight families of indecomposable linear $E_8$-changemaker lattices identified in Section \ref{sec:identifying} with a Tange family, identify each of the six families of decomposable linear $E_8$-changemaker lattices identified in section \ref{sec:identifying} with a family of surgeries on cables of an exceptional fiber, and prove the main theorems. 

For readers with some familiarity with changemaker lattices, we offer a brief summary of the strategy we employ to obstruct linear lattices from admitting $E_8$-changemaker lattice embeddings. First, as in \cite{Gre13}, we take the perspective of an $E_8$-changemaker $\tau = (s, \sigma)$. In Section \ref{sec:workingin}, we prove that the orthogonal complement to $\tau$ admits an \emph{standard basis} of irreducible vectors that are either \emph{tight}, \emph{gappy}, or \emph{just right}, in much the same way that the orthogonal complement to a changemaker does, and that contains the standard changemaker basis of $(\sigma)^\perp\subset \mathbb Z^{n+1}$. Greene's analysis of changemakers whose orthogonal complements are linear lattices progresses by analyzing first the case when every standard basis element is just right, then the case when there is a gappy vector but no tight vector, and finally the case when there is a tight vector. The result of his analysis showed that the collection of changemaker lattices whose associated \emph{intersection graphs} contained no \emph{claws}, \emph{heavy triples}, or \emph{incomplete cycles} corresponds precisely to the set of orthogonal sums of linear lattices coming from connected sums of lens spaces realized either by lens space surgery on a Berge knot or a reducible surgery on a torus knot or a cable thereof. Our analysis of $E_8$-changemakers whose orthogonal complements are sums of linear lattices thus proceeds by first conditioning on whether the changemaker basis of $(\sigma)^\perp$ contains a gappy or tight vector, deducing characteristics of such bases whose intersection graphs contain no claw, heavy triples, or incomplete cycles, then elucidating how such a changemaker basis may be extended to an $E_8$-\emph{changemaker basis} without introducing any claws, heavy triples, or incomplete cycles. We comment further on our strategy at the beginning of Section \ref{sec:identifying}.

\subsection{Acknowledgments.} The author wishes to thank Daniel Allcock, Ken Baker, John Baldwin, Richard Borcherds, Steve Boyer, Josh Greene, Cameron Gordon, Yi Ni, and Motoo Tange for insightful conversations, valuable encouragement, and probing questions. The author would also like to thank the mathematics graduate student community at Boston College for fostering a warm and welcoming atmosphere, which provided an invaluable reservoir of morale.

\section{Lattices.}\label{sec:lattices}
Here we provide an executive summary, of Section 3 of \cite{Gre13}.

\subsection{Generalities.} Let $V$ be a Euclidean vector space of dimension $n$, and let $\mathcal V$ be an orthonormal basis of $V$. A \emph{lattice} $L$ is the $\mathbb Z$-module of integer linear combinations of elements of a basis $\mathcal B$ for $V$, and we say $\rk L = n$. The inner product $\langle-,-\rangle : V\otimes V \to \mathbb R$ restricts to a positive definite, symmetric, bilinear form on $L$; $L$ is said to be an \emph{integer lattice} if the image of this restricted pairing is contained in $\mathbb Z$. Equivalently, letting $B$ be the $n \times n$ matrix whose columns are the vectors in $\mathcal B$ with respect to the coordinates $\mathcal V$, $L$ is an integer lattice if every entry in the matrix $A = B \cdot B^T$ is an integer. In this case, we define the \emph{dual} of $L$,
\begin{equation*}
    L^* : = \{x \in V | \ \langle x, y \rangle \in \mathbb Z \text{ for all } y \in L\},
\end{equation*}
and the discriminant of $L$, disc$(L)$, is the index $[L^*: L]$.  Note here that the restriction of the inner product of $V$ to $L^*$ is given, with respect to the implicit basis $\mathcal B^*$ of $\text{Hom}(V,\mathbb R)\cong V$ dual to $\mathcal B$, by the matrix $A^{-1}$, thus the vectors in $V$ expressed by the columns of $A^{-1}$ with respect to the basis $\mathcal V$ form a $\mathbb Z$-basis for $L^*$. If now $|\det A| = 1$, in which case we say that $L$ is \emph{unimodular}, then $A\in \text{SL}(2, \mathbb Z)$, hence $L^*$ is an integer lattice. It follows that $L = L^*$ and that $A^{-1} : L \to L^*$ encodes this isomorphism with respect to the bases $\mathcal B$ of $L$ and $\mathcal V^* = \mathcal V$. Henceforth, we work only with integer lattices.

The \emph{norm} of a vector $v \in L$ is $|v|:=\langle v, v\rangle$. A vector $v$ is said to be \emph{reducible} if it can be written as the sum of two non-zero vectors $x$, $y \in L$ with $\langle x, y\rangle \geq 0$. A vector is \emph{irreducible} if it is not reducible. A vector $v$ is \emph{breakable} if $v = x+y$ for some $x, y \in L$ with $|x|, |y| \geq 3$ and $\langle x, y\rangle = -1$. A vector is \emph{unbreakable} if it is not breakable. A lattice $L$ is \emph{decomposable} if it can be written as an orthogonal direct sum $L = L_1 \oplus L_2$ with $L_1, L_2 \neq (0)$. A lattice is \emph{indecomposable} if it is not decomposable. 

Every integer lattice $L$ admits a basis $S = \{v_1, \ldots, v_n\}$ of irreducible vectors. Given such a basis $S$, we define its \emph{pairing graph}
\begin{equation*}
    \hat{G}(S) = (S,E), \ \ E = \{(v_i,v_j)\ | \ i \neq j \text{ and } \langle v_i,v_j\rangle \neq 0\}.
\end{equation*}

\subsection{Linear lattices} 
Let $p> q > 0$ be relatively prime integers. Note that the fraction $p/q$ admits a unique Hirzebruch-Jung continued fraction expansion
\begin{equation*}
    p/q =[a_1,a_2,\ldots,a_n]^-= a_1 - \frac{1}{a_2-\frac{1}{\ddots- \frac{1}{a_n}}}
\end{equation*}
with each $a_i \geq 2$ an integer. The \emph{linear lattice} $\Lambda(p,q)$ is the integer lattice freely generated by the \emph{vertex basis} $V = \{x_1, \ldots, x_n\}$, where 
    \[\langle x_i,x_j \rangle = 
    \begin{dcases}
        -a_i, & \text{ if } i = j; \\
        1, & \text{ if } |i-j| = 1\\
        0, & \text{ if } |i-j| >1 \\
    \end{dcases}.
    \]
Notably, the linear lattices $\Lambda(p,q)$ and $\Lambda(p',q')$ are isomorphic if and only if there exists an orientation preserving homeomorphism $L(p,q)\cong L(p',q')$. 

\begin{prop}[\cite{Ger}]
If $\Lambda(p,q) \cong \Lambda(p', q')$, then $p = p'$ and $q = q'$ or $qq' \equiv 1 (\mod p)$. \qed
\end{prop}

An interval $T$ in a linear lattice $L$ is a (possibly singleton) set of consecutive vertices $\{x_i, x_{i+1}, \ldots, x_j\}$ in the pairing graph of the vertex basis $V$ for $L$. Let $[T] \in L$ denote the vector $x_i + \ldots + x_j$. Let $T = \{x_i, \ldots, x_j\}$ and $T' = \{x_k, \ldots, x_l\}$ be distinct intervals. We say that $T$ and $T'$ \emph{share a common endpoint} if $i = k$ or $j = l$ and write $T \prec T'$ if $T\subset T'$. We say that $T$ and $T'$ are \emph{consecutive} if $k = j+1$ or $i = l+1$, and write $T \dag T'$. If $T$ and $T'$ either share a common endpoint or are consecutive, then we write $T \sim T'$ and say that $T$ and $T'$ abut. If $T$ and $T'$ do not abut, then either $T\cap T' = \emptyset$, in which case we say that $T$ and $T'$ are distant, or $T\cap T'$ is a proper non-empty subset of both $T$ and $T'$, in which case we write $T\pitchfork T'$. Note that if $T\pitchfork T'$, then the symmetric difference of $T$ and $T'$ is a union of two distant intervals. 

\begin{prop}[Corollary 3.5 of \cite{Gre13}]\label{prop:linearirreducibles} 
Let $L = \bigoplus_kL_k$ denote a sum of linear lattices.
\begin{enumerate}
    \item The irreducible vectors in $L$ take the form $\pm [T]$, where $T$ is an interval in some $L_k$;
    \item each $L_k$ is indecomposable;
    \item if $T \pitchfork T'$, then $[T\setminus T'] \pm [T'\setminus T]$ is reducible;
    \item $[T]$ is unbreakable iff $T$ contains at most one vertex basis element of norm $\geq 3$.\qed
\end{enumerate}

\end{prop}

\section{Input from Floer homology.}\label{sec:inputfromfloer}

In this section, we will use input from Floer homology to begin to define the $E_8$-changemaker property. 

\subsection{Generalizing the notion of a changemaker.}

While there are by now many flavors of Heegaard Floer homology, in this work we invoke only some of the original theory's most basic properties. Recall that to a rational homology 3-sphere $Y$ equipped with a spin$^\text{c}$ structure $\mathfrak t$, Ozv\'ath--Szab\'o associated a non-trivial $\mathbb F_2$-vector space $\widehat{HF}(Y, \mathfrak t)$ and a numerical invariant $d(Y, \mathfrak t) \in \mathbb Q$, called a \emph{correction term}. We denote by $\widehat{HF}(Y)$ the direct sum $$\bigoplus_{\mathfrak t \in \text{Spin}^\text{c}(Y)} \widehat{HF}(Y,\mathfrak t),$$ and we observe that $\dim \widehat{HF}(Y) \geq |\text{Spin}^\text{c}(Y)| = |H^2(Y)|$. We say that $Y$ is an \emph{L-space} if $\widehat{HF}(Y)$ has minimal rank, i.e. $\dim \widehat{HF}(Y) = |H^2(Y)|$, in which case $\dim \widehat{HF}(Y, \mathfrak t) = 1$ for all $\mathfrak t \in \text{Spin}^\text{c}(Y)$. 

To a 4-dimensional cobordism $W:Y_0 \to Y_1$ equipped with a spin$^\text{c}$ structure $\mathfrak s$ restricting to $\mathfrak t_0$ on $Y_0$ and $\mathfrak t_1$ on $Y_1$, the theory associates a homomorphism $\widehat{F}_{W;\mathfrak s}:\widehat{HF}(Y_0, \mathfrak t_0) \to \widehat{HF}(Y_1, \mathfrak t_1)$. Osv\'ath--Szab\'o showed that if $W$ is negative-definite, then 
\begin{equation}\label{eq:correctiontermbound}
    4d(Y_0, \mathfrak t_0) + c_1(\mathfrak s)^2 + b_2(W) \leq 4d(Y_1,\mathfrak t_1).
\end{equation}

Note that if $Y$ bounds a negative definite 4-manifold $X$, we may take $W := X \setminus B^4$ to be a negative definite cobordism from $S^3$ to $Y$ and obtain, upon noting that $d(S^3, \mathfrak t) = 0$ for the unique $\mathfrak t \in \text{Spin}^{\text{c}}(S^3)$, that
\begin{equation}\label{eq:correctiontermbound'}
    c_1(\mathfrak s)^2 + b_2(X) \leq 4d(Y,\mathfrak t),
\end{equation}
for every $\mathfrak s \in \text{Spin}^\text{c}(X)$ extending $\mathfrak t \in \text{Spin}^\text{c}(Y)$.

\begin{defin}\label{def:sharp}
A negative definite cobordism $W:Y_0 \to Y_1$ is \emph{sharp} if, for every $\mathfrak t_0 \in \text{Spin}^\text{c}(Y_0)$ and $\mathfrak t_1 \in \text{Spin}^\text{c}(Y_1)$, there exists some extension $\mathfrak s \in \text{Spin}^\text{c}(X)$ attaining equality in the bound (\ref{eq:correctiontermbound}). Furthermore, a negative definite $4$-manifold $X$ with connected boundary is said to be sharp if $W:= X\setminus B^4$ is.
\end{defin}

\begin{figure}
    \centering
    \begin{tikzpicture}[scale=.8]
    \begin{knot}[clip width = 5, flip crossing = 1, flip crossing = 3, flip crossing = 6]
    \strand[ultra thick] (0,0) to[curve through={(-1,1)..(0,2)..(1,1)}] (0,0);
    \strand[ultra thick] (1.3,0) to[curve through={(.3,1)..(1.3,2)..(2.3,1)}](1.1,0);
    \strand[ultra thick]  (2.6 + .5,1-0.86602540378) to[curve through={(2.6,0)..(1.6,1)..(2.6,2)}] (2.6 + 0.5, 1+0.86602540378);
    \node at (3.6,1){$\Large \cdots$};
    \strand[ultra thick] (3.6+.5, 1-0.86602540378) to[curve through={(4.6,0)..(5.6,1)..(4.6,2)}] (3.6+.5,1+0.86602540378);
    \strand[ultra thick] (5.9,0) to[curve through={(4.9,1)..(5.9,2)..(6.9,1)}](5.9,0);
    \node at (0,2.35) {$\Large -a_1$};
    \node at (1.3,2.35) {$\Large -a_2$};
    \node at (2.6, 2.35) {$\Large -a_3$};
    \node at (4.6, 2.35) {$\Large -a_{n-1}$};
    \node at (5.9,2.35){$\Large -a_n$};
    \end{knot}
    \end{tikzpicture}
    \caption{The canonical sharp $4$-manifold with boundary $L(p,q)$, $X(p,q)$.}
    \label{fig:Xpq}
\end{figure}
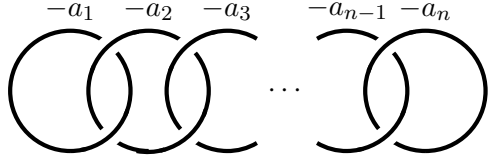

The atomic sharp $4$-manifold for us at present is the manifold $X(p,q)$ with boundary $L(p,q)$, which we now construct. Take the Hirzebruch-Jung continued fraction expansion of the reduced improper fraction $p/q = [a_1, \ldots, a_n]^-$; $X(p,q)$ is the $4$-manifold prescribed by the Kirby diagram in Figure 4. Observe that the intersection form $Q_{X(p,q)}$, which is presented by the linking matrix of the Kirby diagram in Figure 4, is isomorphic to the linear lattice $\Lambda(p,q)$. Ozsv\'ath--Szab\'o showed that $X(p,q)$ is sharp, hence the boundary connected sum $\natural_{i=1}^n X(p_i,q_i)$, whose intersection form is the orthogonal sum $\oplus_{i=1}^n \Lambda(p_i,q_i)$, is also sharp. 

Now, let $X$ be a negative definite $4$-manifold with $H_1(X)$ torsion-free and $b_2(X) = n$, and suppose that $K(p) \cong \partial X$ for some knot $K$ in an integer homology sphere L-space $Y$, e.g. $\mathcal P$, and some integer $p\geq 1$. Form the negative definite $4$-manifold $Z$ with $\partial Z \cong Y$ by identifying the orientation reversal of the trace cobordism of $p$-surgery on $K$, which we call $W$, and $X$ along their common boundary component $K(p) \cong \partial X$. Let $\tau \in Q_Z$ denote the class of the generator of $Q_W$, i.e. a Seifert surface of $K$ in $Y$ capped off with the core of the $2$-handle in $W$, under the image of inclusion into $Q_Z$. 

In \cite{Cau21}, the author showed the following.

\begin{lem}\label{lem:generalized_changemaker}
Let $K$ denote an L-space knot in an L-space integer homology sphere $Y$, and suppose that $K(p)$ bounds a smooth, negative definite 4-manifold $X$ with $H_1(X)$ torsion-free. Then
\begin{equation} \label{eq:generalized_changemaker}
    \mathfrak c^2 + (n+1)-4d(Y)\leq -8t_i(K)
\end{equation}
for all $|i|\leq p/2$ and $\mathfrak c \in \text{Char}(Q_Z)$ such that $\langle \mathfrak c, \tau\rangle + p\equiv 2i \ \mod 2p$. Furthermore, if $X$ is sharp, then for every $|i|\leq p/2$ there exists $\mathfrak c$ attaining equality in (\ref{eq:generalized_changemaker}).\qed
\end{lem}

Recall also the following theorem of Scaduto, anticipated by Fr\o yshov (\cite{Fro96}).

\begin{thm}[Corollary 1.4 of \cite{Sca18}]\label{thm:oneoftwo}
If $Z$ is a negative definite 4-manifold with no 2-torsion in its homology and $\partial Z = \mathcal P$, then $Q_Z \cong -\mathbb Z^{n} \ (n \geq 1)$ or  $Q_Z \cong -E_8 \oplus -\mathbb Z^n \ (n \geq 0)$.\qed    
\end{thm}

Supposing now that $Y=\mathcal P$, that $p\geq 2g(K)$, and that $X$ is sharp, then the intersection form of the $4$-manifold $Z = W \cup X$ is either $-\mathbb Z^{n + 1}$ or $-E_8 \oplus -\mathbb Z^{n-7}$. Then, by (\ref{eq:generalized_changemaker}) and the observation that $t_i(K) = 0$ if and only if $i \geq g(K)$, it follows that $Q_Z\not\cong -\mathbb Z^{n+1}$ since there is no vector in $\Char(-\mathbb Z^{n + 1})$ attaining the equality $\mathfrak c^2=-n+7$. We conclude, then, that $Q_Z \cong -E_8\oplus-\mathbb Z^{n-7}$. We write $\tau = (s,\sigma)\in -E_8\oplus -\mathbb Z^{n-7}$ for the image of the generator of $Q_W$ under inclusion into $H_2(Z)$ where $s \in -E_8$ and $\sigma \in -\mathbb Z^{n-7}$. Noting that $\text{short}(-E_8 \oplus -\mathbb Z^{n-7}) = \{0\} \oplus \{\pm 1\}^{n-7}$, and that $|\langle \mathfrak c, \sigma\rangle|\leq |\sigma|_1 \leq |\sigma| \leq p$, the sharpness of $X$ implies the equality of the following two sets
\begin{equation}\label{eq:changemaker_tail_truncated}
\{\langle(0,\mathfrak c'), (s, \sigma)\rangle \ \colon \ \mathfrak c' \in \{\pm 1\}^{n-7}\} = PI(2g(K) - p, 2g(K)+p).
\end{equation}
It follows from (\ref{eq:changemaker_tail_truncated}) that $\sigma$, after a suitable isometry of $-E_8 \oplus -\mathbb Z^{n-7}$ putting $\sigma$ in the first orthant of $-\mathbb Z^{n-7}$, is a changemaker, and that $2g(K) = p-c(\tau) = p - |\sigma|_1$. 

We may deduce restrictions on $s$ by considering (\ref{eq:generalized_changemaker}) for characteristic vectors in the set Short$(-E_8 \oplus -\mathbb Z^{n+1})$. We assert the following lemma now, but postpone its proof until after we have laid some groundwork for working in the $E_8$ lattice in the next section.

\begin{lem}\label{lem:Shortsmallpairings}
If $\tau = (s, \sigma)$ arises as above, and $\langle s, s\rangle \leq -4$, then $|\langle \mathfrak c, \tau \rangle|\leq p$ for all $\mathfrak c \in \text{Short}(-E_8 \oplus -\mathbb Z^{n-7})$. Moreover, for any $\mathfrak c \in \text{Short}(-E_8 \oplus -\mathbb Z^{n-7})$ and $|i|\leq p/2$, if $\langle \mathfrak c, \tau\rangle + p \equiv 2i \ \mod 2p$, then $\langle \mathfrak c, \tau \rangle + p = 2i$.
\end{lem}

\begin{rmk}\label{ref:rmkaboutsmalls}
If $\langle s, s \rangle \geq -2$, then either $s = 0$ or $\langle s, s\rangle = -2$. We will see upon explicit inspection in Section \ref{sec:workingin}, after we have a handle on working with $E_8$, that $\tau$ is in fact an $E_8$-changemaker regardless of whether or not $\langle s, s\rangle \leq -4$. 
\end{rmk}

Let $g^{(1)}(K)=\min\{i\geq 0\ \colon \ t_i(K)=1\}$. Then, assuming $X$ is sharp, Lemmas \ref{lem:generalized_changemaker} and \ref{lem:Shortsmallpairings} together imply that if $\langle s, s \rangle \leq -4$, then $2g^{(1)} = p - C(\tau)$, and
\begin{equation}\label{eq:hard_eight_inclusion}
PI(-C(\tau), -c(\tau) -2) \subset \{\langle \mathfrak c, \tau \rangle \colon \mathfrak c \in \text{Short}(-E_8 \oplus -\mathbb Z^{n-7})\},
\end{equation}
i.e. for any $j$ in the parity interval $PI(c(\tau) + 2, C(\tau))$, there is some $\mathfrak c \in \text{Char}(-E_8\oplus -\mathbb Z^{n-7})$ with norm $-(n+1)$ such that $\langle \mathfrak c, \tau\rangle = j$. 

In summary, we arrive at the definition of an $E_8$-changemaker as in Definition \ref{def:e8changemaker}, and obtain the following proposition.

\begin{prop}\label{prop:E8changemakerembedding}
For $K \subset \mathcal P$ an L-space knot and $p \geq 2g(K)$, if $K(p)$ bounds a sharp $4$-manifold $X$ with $H_1(X)$ torsion-free, then $Q_X$ embeds in the orthogonal complement to an $E_8$-changemaker $\tau$.
\end{prop}

\begin{proof}
    Let $\tau = (s, \sigma) \in -E_8 \oplus -\mathbb Z^{n+1}$, and let $\{d_0, \ldots, d_n\}$ be an orthonormal basis for $-\mathbb Z^{n+1}$.
    
    The discussion following Lemma \ref{lem:Shortsmallpairings} proves the proposition in the case that $\langle s, s\rangle \leq -4$. 
    
    Note that if $s = 0$, then 
    $$\{\langle \mathfrak c, \tau \rangle \colon \text{Short}(-E_8 \oplus -\mathbb Z^{n-7})\} = \{\langle \mathfrak c, \sigma\rangle \colon \mathfrak c \in \text{Short}(-\mathbb Z^{n+1})\}.$$
    Observe that 
    $$\max\{\langle \mathfrak c, \sigma\rangle \colon \mathfrak c \in \text{Short}(-\mathbb Z^{n-7})\} = |\sigma|_1 + 2 \sigma_n.$$
    We will now show that $PI(|\sigma|_1 + 2, |\sigma|_1 + 2\sigma_{}) \subset \{\langle \mathfrak c, \sigma\rangle \colon \mathfrak c \in \text{Short}(-\mathbb Z^{n+1})\}$. Let $\sigma' = (\sigma_0, \ldots, \sigma_{n-1}) \in - \mathbb Z^{n+1}$, and note that $\sigma'$ is a changemaker, thus, for every $0\leq k \leq |\sigma'|_1$, there is some $A \subset \{0, \ldots, n-1\}$ such that $\sum_{i \in A} \sigma_i = k$. Therefore, $PI(3\sigma_n - |\sigma'|_1, 3\sigma_n + |\sigma'|_1) \subset \{\langle \mathfrak c, \sigma\rangle \colon \mathfrak c \in \text{Short}(-\mathbb Z^{n-7})\}$. Observe now that
    \begin{equation*}
    \begin{split}
        3\sigma_n - |\sigma'|_1 & \leq \sigma_n + |\sigma|_1 + 1 - |\sigma'|_1\\
        & \leq \sigma_n- (\sigma_0 + \ldots + \sigma_{n-1}) + |\sigma|_1 + 1\\
        &\leq 1 + |\sigma|_1 + 1,
    \end{split}
    \end{equation*}
    and the proof in the case that $s = 0$ is complete.

    We complete the proof of the proposition in Proposition \ref{prop:s2e8changemaker} once we have a handle on working in $E_8$.
\end{proof}

In the front matter, we have chosen to adhere to the convention that the lens space $L(p,q)$ is $-p/q$-surgery on the unknot in $S^3$. This convention precipitates a preference for working with negative-definite $4$-manifolds. However, as the arguments henceforth are primarily lattice-theoretic and combinatorial in nature, we ask the reader to keep in mind that the analysis on the negative-definite $4$-manifolds arising from our topological considerations may be carried out on the corresponding positive-definite lattices by means of reversing the orientation of every $4$-manifold we have constructed so far. In what follows, every lattice will be positive definite, and in particular we will abuse notation and write $\Lambda(p,q)$ to refer to the intersection form of $-X(p,q)$. 

\section{Working in the $E_8$ lattice.}\label{sec:workingin}
The $E_8$ integer lattice is the unique even, unimodular lattice of rank 8. It is realized by restricting the dot product on $\mathbb R^8$ to the subset \begin{equation} \label{eq:E8realization}
    \Big\{x \in \mathbb Z^8 \cup \Big(\mathbb Z + \frac 1 2\Big)^8|\sum_{i=0}^8 x_i \equiv 0\ \mod 2\Big\}.
\end{equation}It admits the following Gram matrix with respect to the basis specified in Figure 5, which we fix as the preferred basis throughout this section:
\[ A=
\begin{bmatrix}
2&-1&-1&0&-1&0&0&0\\
-1&2&0&0&0&0&0&0\\
-1&0&2&-1&0&0&0&0\\
0&0&-1&2&0&0&0&0\\
-1&0&0&0&2&-1&0&0\\
0&0&0&0&-1&2&-1&0\\
0&0&0&0&0&-1&2&-1\\
0&0&0&0&0&0&-1&2
\end{bmatrix}.
\]
\begin{figure}
    \centering
\begin{tikzpicture}[main_node/.style={circle,fill=blue!20,draw,minimum size=1em,inner sep=.8pt]}]

    \node[main_node] (4) at (0,0) {$e_4$};
    \node[main_node] (3) at (1, 0)  {$e_3$};
    \node[main_node] (1) at (2, 0) {$e_1$};
    \node[main_node] (5) at (3, 0) {$e_5$};
    \node[main_node] (6) at (4, 0) {$e_6$};
    \node[main_node] (7) at (5, 0) {$e_7$};
    \node[main_node] (8) at (6, 0) {$e_8$};
    \node[main_node] (2) at (2, -1.3) {$e_2$};

    \draw (4) -- (3) -- (1) -- (5) -- (6) -- (7) -- (8);
    \draw (1) -- (2);
\end{tikzpicture}
    \caption{A basis of simple roots for the $E_8$ lattice.}
    \label{fig:E_8Dynkin}
\end{figure}
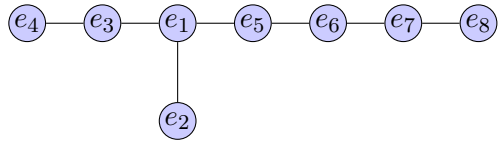

Many low-dimensional topologists are at least familiar with the $E_8$ lattice, though we suspect many are unfamiliar with performing explicit computations in in this lattice. In what follows, we outline a perspective on $E_8$ that significantly eases our computational load. We are grateful to Daniel Allcock and Richard Borcherds for sharing the following perspective.

\subsection{The fundamental Weyl chamber of $E_8$.}
A \textit{root} is any vector $v\in E_8$ with $\langle v, v \rangle = 2$, of which $E_8$ is known to have 240. We refer to the basis elements specified in Figure 5 as \textit{simple roots}. Let $R$ be one of the 240 roots in $E_8$ and write $R = \sum_{i=1}^8a_i e_i$. Then either $a_i \geq 0$ for all $i \in \{1, \ldots, 8\}$ or $a_i \leq 0$ for all $i \in \{1, \ldots, 8\}$, and $R$ is said to be a \emph{positive root} or \emph{negative root} accordingly. Denote by $\mathcal R_+$ the set of positive roots. The set of positive roots admits a partial order: $\sum_{i=1}^8a_ie_i \leq \sum_{i=1}^8b_ie_i$ if $a_i\leq b_i$ for all $i \in \{1,\ldots, 8\}$. 

For the standard embedding $E_8\subset \mathbb R^8$ specified in (\ref{eq:E8realization}), the \textit{Weyl group} of $E_8$ is the group of linear transformations of $\mathbb R^8$ generated by the reflections through each of the 120 hyperplanes in $\mathbb R^8$ orthogonal to one of the 120 positive roots in $E_8$. The \textit{fundamental Weyl chamber}---a fundamental domain for the action of the Weyl group on $E_8$---is the subset

\[
\mathcal C = \{v\in E_8|\langle v, e_i\rangle \geq 0 \text{ for all } 1\leq i \leq 8\}.
\]

As $E_8$ is unimodular, there is a natural isometry $E_8\overset \sim \to E_8^*: = \{v^* \in \text{Hom}(\mathbb R^8,\mathbb R)\colon v^*(x) \in \mathbb Z \text{ for all }x \in E_8\}$ given by $v \mapsto \langle v, -\rangle$. Let $\{e_1^*,\ldots,e_8^*\}$ denote the basis of $E_8^*$ such that $e_i^*(e_j) = \delta_{ij}$, and for a vector $v\in E_8$ let $v^*_i= e_i^*(v)$ for $i \in \{1,\ldots, 8\}$. Then, $v^* = \sum_{i = 1}^8v^*_ie_i^*$. A vector $s \in E_8$ is in $\mathcal C$ if and only if $s^*_i\geq 0$ for all $i \in \{1,\ldots, 8\}$.

Recall that if $p\geq 2g(K)$ and $K(p)$ is an L-space bounding a sharp 4-manifold $X$ with $b_2(X) = n \geq 7$ and $H_1(X)$ torsion-free, then there exists a full rank embedding $Q_X \subset (\tau)^\perp$ for some $\tau \in E_8 \oplus \mathbb Z^{n-7}$, and $\tau$ is an $E_8$-changemaker if $\langle s, s\rangle \geq 4$. Let us now fix a possibly empty orthonormal basis $\{d_0, \ldots, d_{n-8}\}$ for $\mathbb Z^{n-7}$. By applying an appropriate isometry of $E_8 \oplus \mathbb Z^{n-7}$, we may arrange that $\tau = (s, \sigma) \in E_8 \oplus \mathbb Z^{n-7}$ has $s \in \mathcal C$ and $0\leq \sigma_0\leq \ldots \leq \sigma_{n-8}$. 

By restricting our attention to the fundamental Weyl chamber, we may readily deduce some combinatorial constraints on $s$ imposed by $(\ref{eq:hard_eight_inclusion})$. The set $\text{Short}(E_8\oplus \mathbb Z^{n -7})$ partitions naturally as $\big(\text{Short}(E_8) \oplus \text{short}(\mathbb Z^{n-7})\big) \coprod \big( \text{short}(E_8) \oplus \text{Short}(\mathbb Z^{n-7})\big)$. Explicitly, we have a decomposition of $\text{Short}(E_8 \oplus \mathbb Z^{n-7})$ into: $$\mathfrak C_1:=\text{Short}(E_8) \oplus \text{short}(\mathbb Z^{n-7}) =  \{(2R, \chi) \ \colon \ |R| = 2, \chi \in \{\pm 1\}^{n-7}\},\ \text{and}$$  $$\mathfrak C_2:= \text{short}(E_8) \oplus \text{Short}(\mathbb Z^{n-7}) = \{(0,\chi + 2\langle\chi, d_i\rangle d_i): \chi \in \{\pm 1\}^{n-7}\},$$
or, prosaically, $\mathfrak C_2$ is the set of vectors where the $E_8$ coordinates are all 0, and all but one of the $\mathbb Z^{n-7}$ coordinates are $\pm 1$, and the remaining coordinate is $\pm 3$. Notice that if $|\sigma|_1 = 0$, then $2 \in \{\langle \mathfrak c, \tau \rangle \colon \mathfrak c \in \text{Short}(E_8)\}$, thus there is some root $R \in E_8$ with $\langle R, \tau\rangle = 1$.

\begin{proof}[Proof of Theorem \ref{thm:hardeightembedding}.]
Suppose that $p\geq 2g(K)$, and that $K(p)$ bounds a sharp $4$-manifold $X$ with no torsion in $H_1(X)$. Then, by Proposition \ref{prop:E8changemakerembedding}, $Q_X$ embeds in the orthogonal complement to some $E_8$-changemaker $\tau = (s, \sigma) \in E_8 \oplus \mathbb Z^{n+1}$ $(n \geq -1)$. If $|\sigma|_1 \geq 1$, then $\sigma_i = 1$ for some $0 \leq i \leq n$; let $w = d_i$. If $|\sigma|_1 = 0$, then there is some root $R \in E_8\oplus \{(0)\}$ such that $\langle R, \tau \rangle = 1$; let $w = R$. Following \cite[Lemma 3.6]{Gre13}, consider the map $\varphi : E_8 \oplus \mathbb Z^{n+1} \to \mathbb Z/|\tau| \mathbb Z$ given by $\varphi(v) = \langle v, \tau\rangle \ \mod |\tau|$, which is onto since $\langle w, \tau\rangle = 1$. Therefore, the lattice $\mathfrak K = \ker(\varphi)$ has discriminant $[E_8 \oplus \mathbb Z^{n+1} \colon \mathfrak K]^2 = |\tau|^2$. On the other hand, since $\mathfrak K = (\tau)^\perp \oplus (\tau)$, we have that $\text{disc}(\tau)^\perp = \text{disc}(\mathfrak K)/|\tau| = |\tau|$. It follows $Q_X \cong (\tau)^\perp$ since $Q_X$ is a full rank sublattice of $(\tau)^\perp$ and $\text{disc}(Q_X) = \text{disc}(\tau)^\perp = p$.
\end{proof}

In order to prove our main theorems, and in keeping with the conventions of the study of changemaker lattices established in \cite{Gre13}, we seek to understand exactly when the orthogonal complement to an $E_8$-changemaker $\tau = (s, \sigma)\in E_8 \oplus \mathbb Z^{n+1}$ for $n \geq 2$ is a linear lattice. In the next section, we will see how the partial order on $\mathcal R_+$ readily leads to desirable constraints on $s$ for $\tau = (s, \sigma)\in E_8 \oplus \mathbb Z^{n+1}$ an $E_8$-changemaker and informs an algorithm for producing the \emph{standard basis} of $(\tau)^\perp$.

\subsection{$E_8$-changemaker lattices.}

Ultimately, we are interested in constructing a basis $\mathcal S$ for the orthogonal complement of an $E_8$-changemaker $\tau = (s, \sigma) \in E_8\oplus \mathbb Z^{n+1}$. We have already established that $\sigma$ is a changemaker with $0 \leq \sigma_0 \leq \ldots \leq \sigma_n$, and in fact $\sigma_0 = 1$ if $(\tau)^\perp \cong \oplus_{i=1}^k \Lambda(p_i, q_i)$ in light of Theorem \ref{thm:hardeightembedding} since no $\Lambda(p_i,q_i)$ contains a $\mathbb Z$ summand. The upshot of the fact that $\sigma$ is a changemaker is that we may construct a basis for $(\tau)^\perp$ that includes the $n$ \emph{changemaker basis} elements of $(\sigma)^\perp \subset \mathbb Z^{n+1}$, which we denote by $\mathcal V$. The task at hand is now to extend this set of vectors to a basis for all of $(\tau)^\perp$, which will consist of the vectors in $\mathcal V$ together with $8$ additional vectors, each one corresponding a unique simple root of $E_8$; we denote this set of $8$ vectors by $\mathcal W$. The following lemma establishes constraints on $s$ imposed by (\ref{eq:hard_eight_inclusion}) that allow us to construct $\mathcal W$ favorably. 

\begin{figure}
    \centering
\includegraphics[height=5.5in]{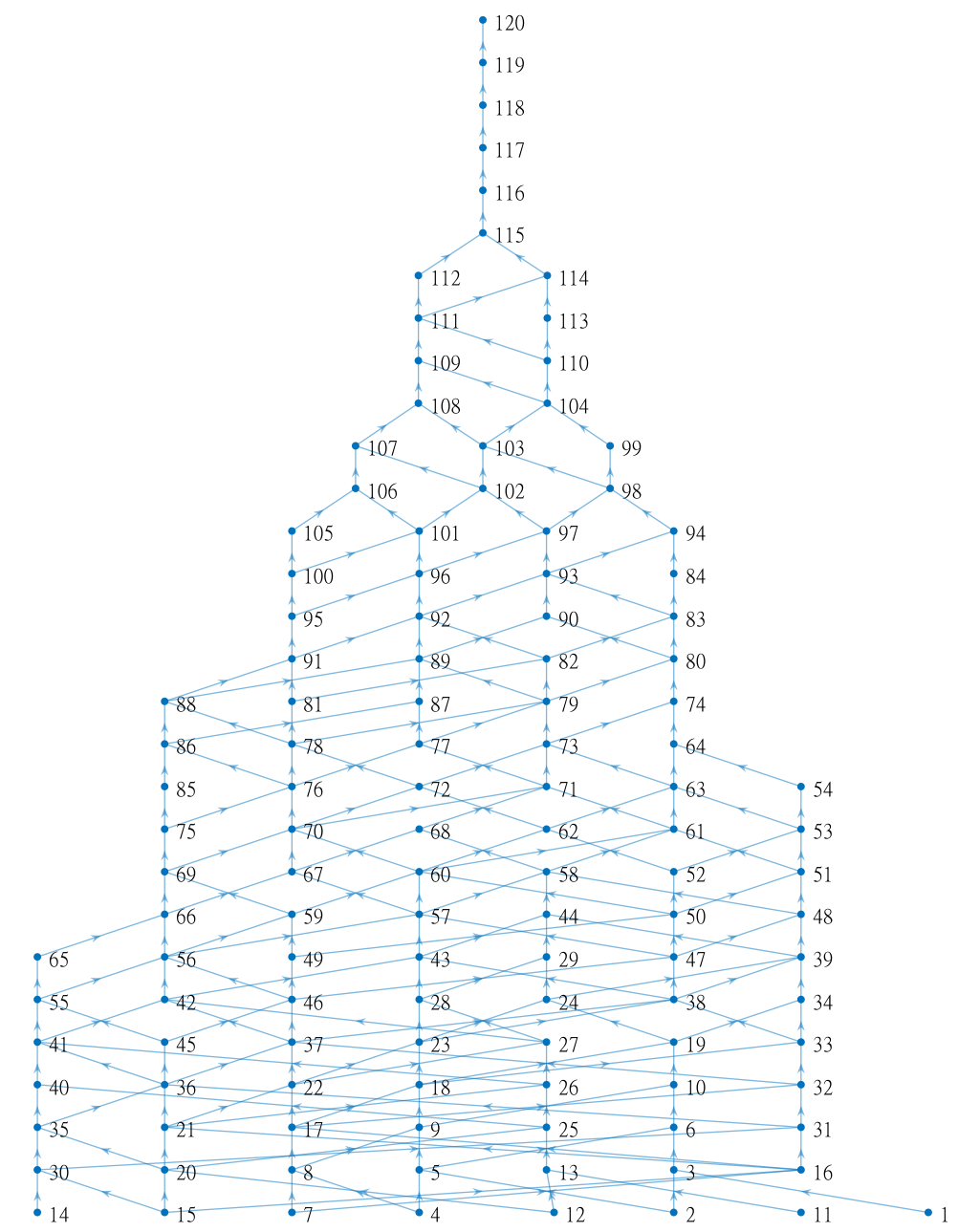}
    \caption{A Hasse diagram giving the partial order on the set of positive roots in $E_8$. Nodes are labeled by positive roots according to the table in Figure 7.}
    \label{fig:E_8Hasse}
\end{figure}

\begin{figure}
    \centering
\begin{tabular}{ | c | c| c | }
 \hline 
$R_{1}=(0, 0, 0, 0, 0, 0, 0, 1)$ & $R_{41}=(1, 1, 1, 1, 1, 0, 0, 0)$ & $R_{81}=(3, 1, 2, 1, 3, 2, 1, 0)$\\ 
\hline
$R_{2}=(0, 0, 0, 0, 0, 0, 1, 0)$ & $R_{42}=(1, 1, 1, 1, 1, 1, 0, 0)$ & $R_{82}=(3, 1, 2, 1, 3, 2, 1, 1)$\\ 
\hline
$R_{3}=(0, 0, 0, 0, 0, 0, 1, 1)$ & $R_{43}=(1, 1, 1, 1, 1, 1, 1, 0)$ & $R_{83}=(3, 1, 2, 1, 3, 2, 2, 1)$\\ 
\hline
$R_{4}=(0, 0, 0, 0, 0, 1, 0, 0)$ & $R_{44}=(1, 1, 1, 1, 1, 1, 1, 1)$ & $R_{84}=(3, 1, 2, 1, 3, 3, 2, 1)$\\ 
\hline
$R_{5}=(0, 0, 0, 0, 0, 1, 1, 0)$ & $R_{45}=(2, 1, 1, 0, 1, 0, 0, 0)$ & $R_{85}=(3, 2, 2, 1, 2, 1, 0, 0)$\\ 
\hline
$R_{6}=(0, 0, 0, 0, 0, 1, 1, 1)$ & $R_{46}=(2, 1, 1, 0, 1, 1, 0, 0)$ & $R_{86}=(3, 2, 2, 1, 2, 1, 1, 0)$\\ 
\hline
$R_{7}=(0, 0, 0, 0, 1, 0, 0, 0)$ & $R_{47}=(2, 1, 1, 0, 1, 1, 1, 0)$ & $R_{87}=(3, 2, 2, 1, 2, 1, 1, 1)$\\ 
\hline
$R_{8}=(0, 0, 0, 0, 1, 1, 0, 0)$ & $R_{48}=(2, 1, 1, 0, 1, 1, 1, 1)$ & $R_{88}=(3, 2, 2, 1, 2, 2, 1, 0)$\\ 
\hline
$R_{9}=(0, 0, 0, 0, 1, 1, 1, 0)$ & $R_{49}=(2, 1, 1, 0, 2, 1, 0, 0)$ & $R_{89}=(3, 2, 2, 1, 2, 2, 1, 1)$\\ 
\hline
$R_{10}=(0, 0, 0, 0, 1, 1, 1, 1)$ & $R_{50}=(2, 1, 1, 0, 2, 1, 1, 0)$ & $R_{90}=(3, 2, 2, 1, 2, 2, 2, 1)$\\ 
\hline
$R_{11}=(0, 0, 0, 1, 0, 0, 0, 0)$ & $R_{51}=(2, 1, 1, 0, 2, 1, 1, 1)$ & $R_{91}=(3, 2, 2, 1, 3, 2, 1, 0)$\\ 
\hline
$R_{12}=(0, 0, 1, 0, 0, 0, 0, 0)$ & $R_{52}=(2, 1, 1, 0, 2, 2, 1, 0)$ & $R_{92}=(3, 2, 2, 1, 3, 2, 1, 1)$\\ 
\hline
$R_{13}=(0, 0, 1, 1, 0, 0, 0, 0)$ & $R_{53}=(2, 1, 1, 0, 2, 2, 1, 1)$ & $R_{93}=(3, 2, 2, 1, 3, 2, 2, 1)$\\ 
\hline
$R_{14}=(0, 1, 0, 0, 0, 0, 0, 0)$ & $R_{54}=(2, 1, 1, 0, 2, 2, 2, 1)$ & $R_{94}=(3, 2, 2, 1, 3, 3, 2, 1)$\\ 
\hline
$R_{15}=(1, 0, 0, 0, 0, 0, 0, 0)$ & $R_{55}=(2, 1, 1, 1, 1, 0, 0, 0)$ & $R_{95}=(4, 2, 2, 1, 3, 2, 1, 0)$\\ 
\hline
$R_{16}=(1, 0, 0, 0, 1, 0, 0, 0)$ & $R_{56}=(2, 1, 1, 1, 1, 1, 0, 0)$ & $R_{96}=(4, 2, 2, 1, 3, 2, 1, 1)$\\ 
\hline
$R_{17}=(1, 0, 0, 0, 1, 1, 0, 0)$ & $R_{57}=(2, 1, 1, 1, 1, 1, 1, 0)$ & $R_{97}=(4, 2, 2, 1, 3, 2, 2, 1)$\\ 
\hline
$R_{18}=(1, 0, 0, 0, 1, 1, 1, 0)$ & $R_{58}=(2, 1, 1, 1, 1, 1, 1, 1)$ & $R_{98}=(4, 2, 2, 1, 3, 3, 2, 1)$\\ 
\hline
$R_{19}=(1, 0, 0, 0, 1, 1, 1, 1)$ & $R_{59}=(2, 1, 1, 1, 2, 1, 0, 0)$ & $R_{99}=(4, 2, 2, 1, 4, 3, 2, 1)$\\ 
\hline
$R_{20}=(1, 0, 1, 0, 0, 0, 0, 0)$ & $R_{60}=(2, 1, 1, 1, 2, 1, 1, 0)$ & $R_{100}=(4, 2, 3, 1, 3, 2, 1, 0)$\\ 
\hline
$R_{21}=(1, 0, 1, 0, 1, 0, 0, 0)$ & $R_{61}=(2, 1, 1, 1, 2, 1, 1, 1)$ & $R_{101}=(4, 2, 3, 1, 3, 2, 1, 1)$\\ 
\hline
$R_{22}=(1, 0, 1, 0, 1, 1, 0, 0)$ & $R_{62}=(2, 1, 1, 1, 2, 2, 1, 0)$ & $R_{102}=(4, 2, 3, 1, 3, 2, 2, 1)$\\ 
\hline
$R_{23}=(1, 0, 1, 0, 1, 1, 1, 0)$ & $R_{63}=(2, 1, 1, 1, 2, 2, 1, 1)$ & $R_{103}=(4, 2, 3, 1, 3, 3, 2, 1)$\\ 
\hline
$R_{24}=(1, 0, 1, 0, 1, 1, 1, 1)$ & $R_{64}=(2, 1, 1, 1, 2, 2, 2, 1)$ & $R_{104}=(4, 2, 3, 1, 4, 3, 2, 1)$\\ 
\hline
$R_{25}=(1, 0, 1, 1, 0, 0, 0, 0)$ & $R_{65}=(2, 1, 2, 1, 1, 0, 0, 0)$ & $R_{105}=(4, 2, 3, 2, 3, 2, 1, 0)$\\ 
\hline
$R_{26}=(1, 0, 1, 1, 1, 0, 0, 0)$ & $R_{66}=(2, 1, 2, 1, 1, 1, 0, 0)$ & $R_{106}=(4, 2, 3, 2, 3, 2, 1, 1)$\\ 
\hline
$R_{27}=(1, 0, 1, 1, 1, 1, 0, 0)$ & $R_{67}=(2, 1, 2, 1, 1, 1, 1, 0)$ & $R_{107}=(4, 2, 3, 2, 3, 2, 2, 1)$\\ 
\hline
$R_{28}=(1, 0, 1, 1, 1, 1, 1, 0)$ & $R_{68}=(2, 1, 2, 1, 1, 1, 1, 1)$ & $R_{108}=(4, 2, 3, 2, 3, 3, 2, 1)$\\ 
\hline
$R_{29}=(1, 0, 1, 1, 1, 1, 1, 1)$ & $R_{69}=(2, 1, 2, 1, 2, 1, 0, 0)$ & $R_{109}=(4, 2, 3, 2, 4, 3, 2, 1)$\\ 
\hline
$R_{30}=(1, 1, 0, 0, 0, 0, 0, 0)$ & $R_{70}=(2, 1, 2, 1, 2, 1, 1, 0)$ & $R_{110}=(5, 2, 3, 1, 4, 3, 2, 1)$\\ 
\hline
$R_{31}=(1, 1, 0, 0, 1, 0, 0, 0)$ & $R_{71}=(2, 1, 2, 1, 2, 1, 1, 1)$ & $R_{111}=(5, 2, 3, 2, 4, 3, 2, 1)$\\ 
\hline
$R_{32}=(1, 1, 0, 0, 1, 1, 0, 0)$ & $R_{72}=(2, 1, 2, 1, 2, 2, 1, 0)$ & $R_{112}=(5, 2, 4, 2, 4, 3, 2, 1)$\\ 
\hline
$R_{33}=(1, 1, 0, 0, 1, 1, 1, 0)$ & $R_{73}=(2, 1, 2, 1, 2, 2, 1, 1)$ & $R_{113}=(5, 3, 3, 1, 4, 3, 2, 1)$\\ 
\hline
$R_{34}=(1, 1, 0, 0, 1, 1, 1, 1)$ & $R_{74}=(2, 1, 2, 1, 2, 2, 2, 1)$ & $R_{114}=(5, 3, 3, 2, 4, 3, 2, 1)$\\ 
\hline
$R_{35}=(1, 1, 1, 0, 0, 0, 0, 0)$ & $R_{75}=(3, 1, 2, 1, 2, 1, 0, 0)$ & $R_{115}=(5, 3, 4, 2, 4, 3, 2, 1)$\\ 
\hline
$R_{36}=(1, 1, 1, 0, 1, 0, 0, 0)$ & $R_{76}=(3, 1, 2, 1, 2, 1, 1, 0)$ & $R_{116}=(6, 3, 4, 2, 4, 3, 2, 1)$\\ 
\hline
$R_{37}=(1, 1, 1, 0, 1, 1, 0, 0)$ & $R_{77}=(3, 1, 2, 1, 2, 1, 1, 1)$ & $R_{117}=(6, 3, 4, 2, 5, 3, 2, 1)$\\ 
\hline
$R_{38}=(1, 1, 1, 0, 1, 1, 1, 0)$ & $R_{78}=(3, 1, 2, 1, 2, 2, 1, 0)$ & $R_{118}=(6, 3, 4, 2, 5, 4, 2, 1)$\\ 
\hline
$R_{39}=(1, 1, 1, 0, 1, 1, 1, 1)$ & $R_{79}=(3, 1, 2, 1, 2, 2, 1, 1)$ & $R_{119}=(6, 3, 4, 2, 5, 4, 3, 1)$\\ 
\hline
$R_{40}=(1, 1, 1, 1, 0, 0, 0, 0)$ & $R_{80}=(3, 1, 2, 1, 2, 2, 2, 1)$ & $R_{120}=(6, 3, 4, 2, 5, 4, 3, 2)$\\ 
\hline
\end{tabular} 
    \caption{A list of the 120 positive roots in $E_8$ expressed with respect to the basis of simple roots given in Figure 5.}
    \label{fig:positiverootslist}
\end{figure}

\begin{lem}\label{lem:hardeightcombinatorial}
Let $\tau = (s, \sigma)$ be an $E_8$-changemaker. Then,
\begin{enumerate}
    \item $s^*_i \leq |\sigma|_1 + 1$ for $i \in \{1, 5, 6, 7, 8\}$;
    \item $s^*_2\leq |\sigma|_1 + 1$ or $s^*_3 \leq |\sigma|_1 + 1$;
    \item if $s^*_2 > |\sigma|_1 + 1$, then\begin{enumerate}
        \item $s^*_2\leq s^*_3 + s^*_4 + |\sigma|_1+1$, and
        \item $s^*_2 \leq s^*_5 + 2s^*_6 + 2s^*_7 + s^*_8 + |\sigma|_1 + 1$;
    \end{enumerate}
    \item if $s^*_3 > |\sigma|_1 + 1$, then
    \begin{enumerate}
        \item $s^*_3 \leq s^*_2 + |\sigma|_1 + 1$, and
        \item $s^*_3 \leq s^*_5 + s^*_6 + s^*_7 + s^*_8 + |\sigma|_1 + 1$
    \end{enumerate}
    \item if $s^*_4 > |\sigma|_1 + 1$, then $s^*_4 \leq s^*_2 + s^*_1 + s^*_5 + s^*_6 + s^*_7 + s^*_8 + |\sigma|_1 +1$;
    \item and if $s^*_2 > |\sigma|_1 + 1$ and $s^*_4> |\sigma|_1 + 1$, then
    \begin{enumerate}
        \item $s^*_2\leq s^*_3 + |\sigma|_1 + 1$, and
        \item $s^*_4 \leq s^*_1 + s^*_5 + s^*_6 + s^*_7 + s^*_8 + |\sigma|_1 + 1$.
    \end{enumerate}
\end{enumerate}
\end{lem}

\begin{proof}
For a list of the 120 positive roots in $E_8$ expressed in coordinates with respect to the basis of simple roots in Figure 5, see Figure 7. For a visualization of the partial order on $\mathcal R_+$, see Figure 6. 

Let $\tau = (s, \sigma)\in E_8 \oplus \mathbb Z^{n+1}$ be an $E_8$ changemaker. Note that for any root $R = \sum_{i=1}^8 a_i e_i \in E_8$, $\langle R, s\rangle = \sum_{i = 1}^8 a_i s^*_i$. Then for any $\mathfrak c = (2R, \chi) \in \mathfrak C_1$, we have $\langle \mathfrak c, \tau\rangle = \sum_{i = 1}^8a_i s^*_i + \chi \cdot \sigma$, where here $\cdot$ denotes the standard dot product. In particular, we have

\begin{equation*} \label{eq1}
\begin{split}
    \max\{\langle \mathfrak c, \tau\rangle \ \colon \ \mathfrak c \in \mathfrak C_1\} &= \langle 2R_{120}, s\rangle + |\sigma|_1\\ 
    &=2(6s^*_1 + 3s^*_2 + 4s^*_3 + 2 s^*_4 + 5s^*_5 + 4s^*_6 + 3s^*_7 + 2s^*_8) + |\sigma|_1.
\end{split}
\end{equation*}

On the other hand,
\begin{equation*} \label{eq2}
\begin{split}
    \max\{\langle \mathfrak c, \tau\rangle \ \colon \ \mathfrak c \in \mathfrak C_2\} &= |\sigma|_1+2\sigma_n\\
    &\leq 2|\sigma|_1 +1.
\end{split}
\end{equation*}

Notice that if $\max\{\langle \mathfrak c, \tau \rangle \ \colon \ \mathfrak c \in \mathfrak C_2\} \geq \langle (2R_{44}, (1, \ldots, 1)), (s, \sigma) \rangle$, then $2(s^*_1 + \ldots + s^*_8)\leq 2|\sigma|_1 +1$, in which case $s^*_i \leq |\sigma|_1 + 1$ for $i = 1, \ldots 8$ and the lemma follows. We now assume that $2\sigma_n < s^*_1 + \ldots + s^*_8$. 

Since $\sigma$ is a changemaker, it is clear that for any $1\leq i \leq 120$, for all $j \in [\langle 2R_i, s\rangle -|\sigma|_1,\langle 2R_i, s\rangle + |\sigma|_1]$ with $j \equiv \langle \tau, \tau\rangle \ \mod 2$ there is some $\mathfrak c \in \mathfrak C_1$ such that $\langle c, \tau\rangle = j$. Therefore, we have \begin{equation}
    \{j \in [|\sigma|_1 + 2, \langle 2R_{120}, s\rangle + |\sigma|_1]\ \colon \ j \equiv \langle \tau , \tau \rangle\  \mod 2\}\subset \bigcup_{i\geq 44}[\langle 2R_i, s\rangle - |\sigma|_1, \langle 2R_i, s\rangle + |\sigma|_1].
\end{equation}
Upon consulting the top 6 vertices of Figure 6, we see that\begin{equation}
    \langle 2R_i, s\rangle + |\sigma|_1 + 2 \geq \langle 2R_{i+1}, s\rangle - |\sigma|_1 \text{ for } 115 \leq i \leq 119,
\end{equation}
which gives us (1) of the lemma.

Part (2) of the lemma follows from observing vertices 112, 114, and 115 in Figure 6, i.e.\begin{equation}
    \langle 2R_{112}, s\rangle + |\sigma|_1 + 2 \geq \langle 2R_{115}, s\rangle -|\sigma|_1 \text{ or } \langle 2R_{114}, s\rangle + |\sigma|_1 + 2 \geq \langle 2R_{115}, s\rangle -|\sigma|_1
\end{equation}

Part (3)(a) comes from noting that if $R< R_{113}$, then $R<R_{112}$, so if $s^*_2>|\sigma|_1 + 1$, then \begin{equation}
    \langle 2R_{112}, s\rangle + |\sigma|_1 + 2 \geq \langle 2R_{113}, s\rangle -|\sigma|_1.
\end{equation}
To establish (3)(b), note that $R_{85}$ is the unique minimum of $\{R\in \mathcal R_+\ \colon \ R\not \leq R_{84}\}$, so that if $s^*_2 > |\sigma|_1 + 1$, then \begin{equation}
    \langle 2R_{84}, s\rangle + |\sigma|_1 + 2 \geq \langle 2R_{85}, s\rangle -|\sigma|_1.
\end{equation}

To see (4)(a), note that if $s^*_3> |\sigma|_1 + 1$, then 
\begin{equation}
    \langle 2R_{114}, s\rangle + |\sigma|_1 + 2 \geq \langle 2R_{112}, s\rangle - |\sigma|_1.
\end{equation}
To see (4)(b), note that $R_{100}$ is the unique minimum of $\{R \in \mathcal R_+\ \colon \ R \not \leq R_{99}\}$, so that if $s^*_3 > |\sigma|_1 + 1$, then\begin{equation}
    \langle 2R_{99}, s\rangle + |\sigma|_1 + 2 \geq \langle 2R_{100}, s\rangle -|\sigma|_1.
\end{equation}

To see (5), observe that $R_{105}$ is the unique minimum of $\{R\in \mathcal R_+ \ \colon \ R\not \leq R_{113}\}$, so in the even $s^*_4 > |\sigma|_1 + 1$ it must be that\begin{equation}
    \langle 2R_{113}, s\rangle + |\sigma|_1 + 2 \geq \langle 2R_{105}, s\rangle  -|\sigma|_1.
\end{equation}

To see (6)(a), note that if both $s^*_2 > s^*_3 + |\sigma|_1 + 1$ and $s^*_4 > |\sigma|_1 + 1$, then $\langle 2R_{114}, s\rangle -|\sigma|_1 -2 \not \in \{\langle \mathfrak c, \tau \rangle \ \colon \ \mathfrak c \in \mathfrak C_1\}$. Therefore if $s^*_2$ and $s^*_4$ are both greater than $|\sigma|_1 + 1$, it must be that $s^*_2 \leq s^*_3 + |\sigma|_1+1$. 
To see (6)(b), note that $R_{105}$ is the unique minimum of $\{R \in \mathcal R_+ \ \colon \ R \not \leq R_{110}\}$, so that if $s^*_2 > |\sigma|_1 + 1$ and $s^*_4 > |\sigma|_1 + 1$, then \begin{equation}
    \langle 2R_{110}, s\rangle + |\sigma|_1 + 2 \geq \langle 2R_{105}, s\rangle -|\sigma|_1.
\end{equation}
\end{proof}

\begin{proof}[Proof of Lemma \ref{lem:Shortsmallpairings}.]
Let $\tau = (s, \sigma) \in E_8 \oplus \mathbb Z^{n+1}$, with $\sigma$ a changemaker and $|\tau| = p \geq |\sigma|+ 4$. We may arrange so that $s^*_i \geq 0$ for all $i \in \{1, \ldots, 8\}$ by applying an isometry of $E_8$ that places $s$ in the fundamendtal Weyl chamber of $E_8$. If $C(\tau) = |\sigma|_1 + 2\sigma_n$ and $C(\tau)> p$, then $|\sigma|_1 + 2\sigma_n > |\sigma|$, so either $\sigma = (0, \ldots, 0, 1, \ldots, 1)$ and $s= 0$, or $\sigma = (0, \ldots, 0, 1, \ldots, 1, 2)$ and $s^* = (0, 0, 0, 0, 0, 0, 0, 1)$, in which case $|s| = 2$. We may therefore assume that $$C(\tau) = 12s^*_1 + 6s^*_2 + 8s^*_3 + 4s^*_4 + 10s^*_5 + 8s^*_6 + 6s^*_7 + 4s^*_8 + |\sigma|_1.$$ Upon consulting the diagonal entries of the matrix

\[A^{-1} = 
\begin{bmatrix}
30&15&20&10&24&18&12&6\\
15&8&10&5&12&9&6&3\\
20&10&14&7&16&12&8&4\\
10&5&7&4&8&6&4&2\\
24&12&16&8&20&15&10&5\\
18&9&12&6&15&12&8&4\\
12&6&8&4&10&8&6&3\\
6&3&4&2&5&4&3&2
\end{bmatrix},
\]which records the pairings $(A^{-1})_{ij} = \langle e_i^*, e_j^*\rangle$ of the basis of $E_8$ dual to the choice of simple roots $\{e_1, \ldots, e_8\}$, we see that $$p \geq 30(s^*_1)^2 + 8(s^*_2)^2 + 14(s^*_3)^2 + 4(s^*_4)^2 + 20 (s^*_5)^2 + 12(s^*_6)^2 + 6(s^*_7)^2 + 2(s^*_8)^2 + |\sigma|.$$

Since $|\sigma|\geq |\sigma|_1$, if $C(\tau)>p$, we must have $s^*_i = 0$ for all $1\leq i \leq 7$ and $s^*_8 \leq 1$, in which case $|s| \leq 2$. Conclude that if $|s|\geq 4$, then $C(\tau) \leq p$. 
\end{proof}

\begin{prop}\label{prop:s2e8changemaker}
If $|s| = 2$, then $\tau$ is an $E_8$-changemaker.
\end{prop}

\begin{proof}
    If $|s| = 2$, then we may take $s = R_{120}$ (cf. Figure 7), which is the unique root in the fundamental Weyl chamber of $E_8$. It follows that either $C(\tau) = |\sigma|_1 + 4$, in which case $\sigma_n \leq 1$, or $C(\tau) = |\sigma|_1 + 2\sigma_n$, and in either case we have $PI(|\sigma|_1 + 2,C(\tau))\subset \{\langle \mathfrak c, \tau\rangle \colon \mathfrak c \in \text{Short}(-E_8 \oplus -\mathbb Z^{n+1})\}$.
\end{proof}

\subsection{The standard basis of an $E_8$-changemaker lattice.}
We are now ready to describe the standard basis $\mathcal S$ of irreducible vectors of an $E_8$-changemaker lattice $(\tau)^\perp$. We first note that, by a computer search, there are 1003 non-zero $E_8$-changemakers in $E_8$. The author developed an algorithm for producing a standard basis of irreducible vectors for an $E_8$-changemaker lattice in $E_8$, but feels that it does not bear mentioning further in light of the following proposition. 

\begin{prop}\label{prop:normboundnlessthan2}
If $\tau = (s, \sigma) \in E_8 \oplus \mathbb Z^{n+1}$ is an $E_8$-changemaker and $n \in \{-1,0,1\}$, then $\langle \tau, \tau \rangle \leq 100,000$.
\end{prop}

\begin{proof}
By Lemma \ref{lem:hardeightcombinatorial}, we have that $s^*_j \leq |\sigma|_1 + 1$ for $j \in \{1,5,6,7,8\}$, $s^*_2\leq 3|\sigma|_1 + 3$, $s^*_3\leq 2|\sigma|_1 + 2$, $s^*_4 \leq 7|\sigma|_1 + 7$, $s^*_2 \leq |\sigma|_1 + 1$ or $s^*_3\leq |\sigma|_1 + 1$, and $s^*_2 \leq 2|\sigma|_1 + 2$ or $s^*_4 \leq 6|\sigma|_1 +6$, where we understand $|\sigma|_1$ to be $0$ if $n = -1$. Recall the Gram matrix for $E_8$, 
\[A =
\begin{bmatrix}
2&-1&-1&0&-1&0&0&0\\
-1&2&0&0&0&0&0&0\\
-1&0&2&-1&0&0&0&0\\
0&0&-1&2&0&0&0&0\\
-1&0&0&0&2&-1&0&0\\
0&0&0&0&-1&2&-1&0\\
0&0&0&0&0&-1&2&-1\\
0&0&0&0&0&0&-1&2
\end{bmatrix},
\]

and note that 

\[A^{-1} = 
\begin{bmatrix}
30&15&20&10&24&18&12&6\\
15&8&10&5&12&9&6&3\\
20&10&14&7&16&12&8&4\\
10&5&7&4&8&6&4&2\\
24&12&16&8&20&15&10&5\\
18&9&12&6&15&12&8&4\\
12&6&8&4&10&8&6&3\\
6&3&4&2&5&4&3&2
\end{bmatrix}
\]

facilitates the computation, for $s = \sum_{i = 1}^8s_i e_i$,

\[
\langle s, s\rangle = \begin{bmatrix}
    s_1 & \cdots & s_8
\end{bmatrix}\cdot A \cdot \begin{bmatrix}
    s_1 \\
    \vdots\\
    s_8
\end{bmatrix} = \begin{bmatrix}
    s_1 & \cdots & s_8
\end{bmatrix}\cdot A(A^{-1}A) \cdot \begin{bmatrix}
    s_1 \\
    \vdots\\
    s_8
\end{bmatrix} =\begin{bmatrix}
    s^*_1 & \cdots & s^*_8
\end{bmatrix}\cdot A^{-1} \cdot \begin{bmatrix}
    s^*_1 \\
    \vdots\\
    s^*_8
\end{bmatrix}.
\]
It is then straightforward to certify that $\sigma = (1,2)$ and $s^* = (4, 4, 8, 28, 4, 4, 4, 4)$ maximizes the quantity $\langle \tau, \tau\rangle$ for $\tau = (s, \sigma) \in E_8 \oplus \mathbb Z^{n+1}$ with $n \in \{-1, 0, 1\}$, and in particular
\[
\begin{bmatrix}
    4 & 4 & 8 & 28 & 4 & 4 & 4 & 4
\end{bmatrix} \cdot A^{-1} \cdot \begin{bmatrix}
    4\\
    4\\ 
    8\\
    28\\
    4\\
    4\\
    4\\
    4\\
\end{bmatrix} + 1^2 + 2^2 = 25,541 \leq 100,000.
\]
\end{proof}

As Rasmussen has computational proof \cite[Section 6.3]{Ras07} that every lens space $L(p,q)$ that is realized by surgery on a knot $K \subset \mathcal P$ with $2g(K)\leq p \leq 100,000$ is realized by surgery on a Tange knot, we will not in this work address when $\Lambda(p,q) \cong (\tau)^\perp$ for some $\tau \in E_8 \oplus \mathbb Z^{n+1}$ for $n \in \{-1,0,1\}$. However, we remain interested in when $(\tau)^\perp \cong \Lambda(p_1,q_1) \oplus \Lambda(p_2,q_2)$ for all $n \geq -1$, but postpone the discussion of this case until Section \ref{sec:main}.

Suppose now that $n\geq 2$ and $\tau \in E_8 \oplus \mathbb Z^{n+1}$ is an $E_8$-changemaker. We construct a basis for the $E_8$-changemaker lattice $ L = (\tau)^\perp \subset E_8\oplus\mathbb Z^{n+1}$ as follows. Fix an index $1 \leq j \leq n$, and suppose that $\sigma_j = 1 + \sum_{i=0}^{j-1}\sigma_i$. In this case, set $v_j = -d_j + 2d_0 + \sum_{i=1}^{j-1} d_i$. Otherwise, $\sigma_j \leq \sum_{i=0}^{j-1}\sigma_i.$ It follows that there exists a subset $A\subset \{0, \ldots, j-1\}$ such that $\sigma_j = \sum_{i\in A}\sigma_i$. Amongst all such subsets, choose the one maximal with respect to the total order $<$ on subsets of $\{0,1,\ldots, n\}$ defined by declaring $A'<A$ if the largest element in $(A \cup A')\setminus(A\cap A')$ lies in $A$; equivalently, $A' < A$ if $\sum_{i\in A'}2^i<\sum_{i\in A}2^i.$ Then set $v_j = -d_j + \sum_{i\in A}d_i \in L$. If $v = -d_j + \sum_{i\in A'}d_i$ for some $A'<A$, then write $v\ll v_j$. We call the set $\mathcal V:=\{v_1, \ldots, v_n\}$ the \emph{changemaker basis} for $\tau$.

Now fix an index $1\leq j \leq 8$. If $s^*_j = |\sigma|_1 + 1$, then set $w_j = -e_j + 2d_0 + \sum_{i = 1}^n d_i$. If $s^*_j \leq |\sigma|_1$, then set $w_j = -e_j + \sum_{i\in A}d_i$ for $A$ maximal such that $\sum_{i \in A}\sigma_i= s^*_j$. If $s^*_j > |\sigma|_1 + 1$, then $j \in \{2,3,4\}$, and producing $w_j$ requires more care. Lemma \ref{lem:hardeightcombinatorial} ensures that if $s^*_j > |\sigma|_1 + 1$, then there is an especially simple root $r$ such that $0\leq -e_j + r\leq |\sigma|_1 +1$ for which we may then make change with $\sigma$. In this case, we write $w_j|_{E_8} := -e_j + r$. 

The simplest case to deal with is when $s^*_3 > |\sigma|_1 + 1$. By Lemma \ref{lem:hardeightcombinatorial}, we observe that $s^*_2 \leq |\sigma|_1 + 1$ and $|\sigma|_1| + 1 < s^*_3 \leq s^*_2 + |\sigma|_1 + 1$. In particular, $0 < s^*_3 - s^*_2 \leq |\sigma|_1 + 1$. If $s^*_3 - s^*_2 = |\sigma|_1 + 1$, then set $w_3 = -e_3 + e_2 + 2d_0 + \sum_{i = 1}^n d_i$. If $s^*_3 - s^*_2 \leq |\sigma|_1$, then set $w_3 = -e_3 + d_2 + \sum_{i \in A}d_i$ with $A$ maximal such that $\sum_{i\in A} = s^*_3 - s^*_2$.

Consider now the case when $s^*_2 > |\sigma|_1 + 1$. By Lemma \ref{lem:hardeightcombinatorial}, we observe that $s^*_3 \leq |\sigma|_1 + 1$ and $|\sigma|_1| + 1 < s^*_2 \leq s^*_3 + s^*4 + |\sigma|_1 + 1$. If $s^*_3 = 0$, then set $w_2 = -e_2 + e_4 + 2d_0 + \sum_{i=1}^n d_i$ or $w_2 = -e_2 + e_4 + \sum_{i \in A} d_i$ as appropriate. If $s^*_3 > 0$, then in the event that $s^*_4 = 0$ or $s^*_2 -s^*_3 - s^*_4 < 0$, set $w_2 = -e_2 + e_3 + 2d_0 + \sum_{i=1}^nd_i$ or $w_2 = -e_2 + e_3 + \sum_{i\in A}d_i$ as appropriate, and in the event that $s^*_4 > 0$ and $0\leq s^*_2 -s^*_3 -s^*_4$, set $w_2 = -e_2 + e_3 + e_4 + 2d_0 + \sum_{i=1}^nd_i$ or $w_2 = -e_2 + e_3 + e_4 + \sum_{i\in A}d_i$ as appropriate. Note that it is possible that $\langle w_2, d_i\rangle = 0$ for all $0\leq i \leq n$, and in this case $w_2 = -e_2 + e_3 + e_4$.

Consider lastly the case when $s^*_4 > |\sigma|_1 + 1$. By Lemma \ref{lem:hardeightcombinatorial}, we observe that $s^*_4 \leq s^*_2 + s^*_1 + s^*_5 + s^*_6 + s^*_7 + s^*_8 + |\sigma|_1 + 1$. We will describe an iterative algorithm to produce $w_4$ where we first produce $w_4|_{E_8}$. First, set $w_4|_{E_8} = -e_4$. Proceed in steps, as follows.

\begin{itemize}
    \item Step 1: If $0 < s^*_2 \leq |\sigma|_1 + 1$, set $w_4|_{E_8} = w_4|_{E_8} + e_2$.\\
    \item Step 2: If $\langle w_4|_{E_8},\tau \rangle < \langle w_4|_{E_8} + e_1, \tau \rangle \leq 0$ or $|w_4|_{E_8}| = 4$ and $\langle w_4|_{E_8}, \tau\rangle < \langle w_4|_{E_8} + e_j, \tau\rangle \leq 0$ for some $j \in \{1,5,6,7,8\}$, then set $w_4|_{E_8} = w_4|_{E_8} + e_1$, else proceed to Step 7.\\
    \item Step 3: If $\langle w_4|_{E_8},\tau \rangle < \langle w_4|_{E_8} + e_5, \tau \rangle \leq 0$ or $|w_4|_{E_8}| = 4$ and $\langle w_4|_{E_8}, \tau\rangle < \langle w_4|_{E_8} + e_j, \tau\rangle \leq 0$ for some $j \in \{5,6,7,8\}$, then set $w_4|_{E_8} = w_4|_{E_8} + e_5$, else proceed to Step 7.\\
    \item Step 4: If $\langle w_4|_{E_8},\tau \rangle < \langle w_4|_{E_8} + e_6, \tau \rangle \leq 0$ or $|w_4|_{E_8}| = 4$ and $\langle w_4|_{E_8}, \tau\rangle < \langle w_4|_{E_8} + e_j, \tau\rangle \leq 0$ for some $j \in \{6,7,8\}$, then set $w_4|_{E_8} = w_4|_{E_8} + e_6$, else proceed to Step 7.\\
    \item Step 5: If $\langle w_4|_{E_8},\tau \rangle < \langle w_4|_{E_8} + e_7, \tau \rangle \leq 0$ or $|w_4|_{E_8}| = 4$ and $\langle w_4|_{E_8}, \tau\rangle < \langle w_4|_{E_8} + e_j, \tau\rangle \leq 0$ for some $j \in \{7,8\}$, then set $w_4|_{E_8} = w_4|_{E_8} + e_7$, else proceed to Step 7.\\
    \item Step 6: If $\langle w_4|_{E_8},\tau \rangle < \langle w_4|_{E_8} + e_8, \tau \rangle \leq 0$ or $|w_4|_{E_8}| = 4$ and $\langle w_4|_{E_8}, \tau \rangle < \langle w_4|_{E_8} + e_8, \tau\rangle \leq 0$, then set $w_4|_{E_8} = w_4|_{E_8} + e_8$.\\
    \item Step 7: If $\langle w_4|_{E_8}, \tau \rangle = |\sigma|_1 + 1$, then set $w_4 = w_4|_{E_8} + 2d_0 + \sum_{i = 1}^nd_i$, else set $w_4 = w_4|_{E_8} + \sum_{i \in A} d_i$ for $A$ maximal such that $\langle w_4|_{E_8}, \tau \rangle = - \sum_{i \in A} \sigma_i$.
    
\end{itemize}       

\begin{rmk}
    Note that it is possible that $\langle w_4, d_i\rangle = 0$ for all $0 \leq i \leq n$.
\end{rmk}

We call the set $\mathcal W:=\{w_1, \ldots, w_8\}$ an $E_8$-\emph{extension set} for $\tau$.

The vectors $w_1,\ldots, w_8,v_1,\ldots, v_n$ are clearly linearly independent. The fact that they span $L$ is straightforward to verify, too: conditioning on whether $s^*_2> |\sigma|_1 + 1,s^*_3 > |\sigma|_1 + 1$, and $s^*_4> |\sigma|_1 +1$, add suitable multiples of $w_2,w_3,w_4$ in turn, followed by suitable multiples of $w_1,w_5,\ldots w_8, v_n,\ldots, v_1$ in turn to a given $x \in L$ to produce a sequence of vectors that converges to $0$.

\begin{defin}
The set $S= \mathcal V \cup \mathcal W$ constitutes the \emph{standard basis} for $L$. In keeping with \cite[Definition 3.11]{Gre13}, a vector $v \in \mathcal S$ is
\begin{itemize}
    \item \emph{tight} if the projection of $v$ onto $\mathbb Z^{n+1}$ is $2d_0 + \sum_{i = 1}^n d_i$;
    \item \emph{gappy} if the projection of $v$ onto $\mathbb Z^{n+1}$ is $\sum_{i\in A}d_i$, $A\neq \emptyset$, and $A$ does not consist of consecutive integers; 
    \item \emph{just right} if the projection of $v$ onto $\mathbb Z^{n+1}$ is $\sum_{i \in A}d_i$ and $A = \emptyset$ or $A$ consists of consecutive integers.
    \end{itemize}
    
    A \emph{gappy index} for a gappy vector $v_j$ is an index $k \in A$ such that $k + 1\not \in A\cup \{j\}$. A \emph{gappy index} for a gappy vector $w_j$ is an index $k \in A$, $k < n$, such that $k + 1 \not\in A$. A vector $w_j\in \mathcal W$ is \emph{loaded} if $s^*_j > |\sigma|_1 + 1$, and otherwise it is \emph{unloaded}.
\end{defin} 
Every vector in $\mathcal S$ is either tight, gappy, or just right (exclusively). At most one vector in each of the sets $\{w_2, w_3\}$, $\{w_4\}$ is loaded, and no other basis elements are loaded.

\begin{lem}\label{lem:irreducible}
The standard basis elements of an $E_8$-changemaker lattice are irreducible.
\end{lem}

In order to establish Lemma \ref{lem:irreducible}, we will make use of the following structural proposition regarding vectors of norm at most 4 in $E_8$ and their pairings against loaded elements in the $E_8$ extension set $\mathcal W$. 

\begin{prop}\label{prop:technicalpairings}
Let $w_j$ be a loaded vector in the standard basis of the $E_8$-changemaker lattice $(\tau)^\perp \subset E_8 \oplus \mathbb Z^n$ with $w_j|_{E_8} = -e_j + r$ for some positive root $r$, and let $z\in E_8 \oplus \{0\}^n$. 
\begin{enumerate}
    \item If $z$ is a positive root and $\langle -e_j + r + z, -z\rangle = 0$, then $\langle z, \tau\rangle > |\sigma|_1 + 1$. 
    \item If $z$ is a positive root and $\langle -e_j + r + z, -z\rangle = -1$, then $\langle z, \tau \rangle = 0$, $\langle z, \tau \rangle > |\sigma|_1 + 1$, or $\langle -e_j + r + z, \tau\rangle > 0$.
    \item If $|z| = 4$ and $\langle -e_j + r + z, -z \rangle = -1$, then either $\langle z, \tau\rangle \leq 0$, or $\langle z, \tau\rangle > |\sigma|_1 + 1$, or $\langle -e_j + r + z, \tau \rangle = 0$ and $|-e_j + r + z| = 2$, or $\langle -e_j + r + z, \tau\rangle > 0$.
\end{enumerate}
\end{prop}
\begin{proof}[Sketch of proof]
The proposition is established by explicit computation in a Jupyter notebook, available at \url{https://github.com/caudellj/LoadedPairings}.
\end{proof}

\begin{proof}[Proof of Lemma \ref{lem:irreducible}]

Choose a standard basis element $v_j \in S$ and suppose that $v_j = x+y$ for $x, y \in L$ with $\langle x, y\rangle \geq 0$. In order to prove that $v_j$ is irreducible, it stands to show that one of $x$ and $y$ equals $0$. Since $v_j|_{E_8} = 0$, write $x = z + \sum_{i = 0}^nx_i d_i$ and $y = -z + \sum_{i=0}^n y_id_i$ for some $z \in E_8$. The proof of \cite[Lemma 3.13]{Gre13} establishes the inequality $\sum_{i=0}^nx_i y_i \leq 1$, so we may conclude that $z=0$ since $-\langle z, z\rangle \leq -2$ for all $z \in E_8\setminus \{0\}$. That $v_j$ is irreducible then follows from the same argument as in the proof of \cite[Lemma 3.13]{Gre13}.

Now choose a standard basis element $w_j \in S$.

\emph{Case 1}. Suppose that $w_j$ is not loaded and not tight. Write $x = -e_j + z + \sum_{i=1}^n x_id_i$ and $y = -z + \sum_{i = 1}^n y_id_i$ for some $z \in E_8$. Since $x_i + y_i \in \{0, 1\}$ for all $i$, we have $x_iy_i \leq 0$ for all $i$. Observe that $\langle -e_j + z, -z\rangle =\langle e_j, z\rangle -\langle z,z\rangle \leq 0$, with equality if and only if $z \in \{0, e_j\}$. If $w_j$ is reducible, it must be that $z \in \{0,e_j\}$ and $\langle x, y\rangle = 0$, and so $0\leq x_i,y_i\leq 1$ for all $i$. Therefore $x = 0$ or $y=0$, as desired.

\emph{Case 2}. Suppose that $w_j$ is loaded and not tight. According to our algorithm for producing standard basis elements, the projection of of $w_j$ onto $E_8$ is $-e_j + r$, where $r$ is some positive root and $\langle -e_j+ r, -e_j + r\rangle = 4$. In this case, write $ x = -e_j + r + z+ \sum_{i=1}^{n}x_id_i$ and $y = -z + \sum_{i = 1}^n y_id_i$ for some $z \in E_8 \oplus \{0\}^{n+1}$. Then $$\langle x, y\rangle = \langle -e_j + r,- z\rangle - \langle z, z\rangle + \sum_{i=0}^nx_iy_i.$$

Since $\langle w_j, d_i\rangle \in \{0,1\}$ for all $0\leq i\leq n$, it follows that $\sum_{i = 1}^nx_iy_i\leq 0$, as in Case 1. Observe that
\begin{equation}
    \langle -e_j+ r,-z \rangle - \langle z, z\rangle \leq \sqrt{\langle -e_j + r, -e_j + r\rangle \langle z, z\rangle}-\langle z, z\rangle = 2\sqrt{|z|}-|z|,
\end{equation}
it follows that $|z|\leq 4$. 

If $|z| = 4$, then we must have $\langle -e_j+ r,- z\rangle = 4$, i.e. $z = e_j - r$. Then $x = 0$ or $y = 0$. 

If $|z| = 2$, then $\langle -e_j, -z\rangle +\langle r, -z\rangle - |z| \leq 0$ and $z \in \mathcal R_+$. If we have equality, then it follows that $\langle e_j, z\rangle = 1 = -\langle r, z\rangle$, or else $z \in \{ e_j, -r\}$, in which case $x = 0$ or $y = 0$. Since $z$ is a positive root and $\langle e_j, z\rangle  = 1$, it follows that $\langle z, \tau\rangle \geq s^*_j \geq |\sigma|_1 + 1$ so that $y_k \geq 2$ for some $k$ and that $\langle -e_j + r + z, \tau \rangle \geq 1$ so that $x_l\leq -1$ for some $l$. However, then $\langle x, y\rangle \leq x_ky_k + x_l y_l \leq -3$, a contradiction.

\emph{Case 3}. Suppose that $w_j$ is not loaded and tight. In this case, write $x = -e_j + z + \sum_{i=1}^n x_id_i$ and $y = -z + \sum_{i = 0}^n y_id_i$. In contrast with the case where $w_j$ is not loaded and not tight, $\sum_{i = 0}^nx_iy_i \leq 0$ unless $x_1 = y_1 = 1$ and $x_iy_i = 0$ for all $i \neq 0$. If $\sum_{i=0}^nx_iy_i \leq 0$, then the argument in the case $w_j$ is not loaded and not tight shows that $w_j$ is irreducible. If $\sum_{i=0}^nx_iy_i = 1$, then $\langle -e_j + z, -z\rangle \geq -1$, and therefore $|z| = 2$. Now since $y_i \geq 0$ for all $i$, $z$ must be a positive root. Since $\langle e_j, e_i\rangle \leq 0$ for all $i\neq j$, it must that $z_j \geq 1$, in which case $\langle z, \tau \rangle \geq |\sigma|_1 + 1$. Then $y_k \geq 2$ for some $k\geq 2$, in which case $\langle x, y\rangle \leq -2$.

\emph{Case 4}. Suppose that $w_j$ is loaded and tight. Write $x = -e_j + r + z + \sum_{i=0}^nx_id_i$ and $y = -z + \sum_{i=0}^ny_id_i$. As in Case 3, we may assume that $x_1 = y_1 = 1$, and $x_iy_i = 0$ for all $i \neq 1$, and so we seek to rule out the existence of a $z\in E_8$ such that $\langle -e_j + r + z, -z\rangle = -1$ and $\langle x, y \rangle = 0$. If $|z|\geq 6$, then $\langle -e_j + r + z, -z \rangle < -1$, so conclude that $|z| \leq 4$. 

Supposing that $|z| = 4$, then $\langle -e_j + r + z, -z\rangle = -1$, and supposing that $|z| = 2$, then $\langle -e_j + r + z, -z \rangle = -1$. In both cases, Proposition \ref{prop:technicalpairings} tells us that there are no such $z \in E_8\oplus \{0\}^n$ that yield $x, y \in (\tau)^\perp$ such that $w_j = x+y$ and $\langle x, y\rangle \geq 0$.
\end{proof}

\begin{lem}
If a standard basis element is not tight, it is unbreakable.
\end{lem}

\begin{proof}
Suppose that $v_i$ is not tight. If $v_i$ is breakable, then we may write $v_i = x+y$ with $x = z +\sum_{i=1}^nx_id_i$ and $y = -z + \sum_{i = 1}^n y_id_i$ for some $z \in E_8$ such that $|x|\geq 3$, $|y|\geq 3$, and $\langle x, y\rangle =-1$. Then $$\langle x, y \rangle = \langle z, -z\rangle + \sum_{i = 1}^n x_iy_i\leq \langle z, -z\rangle\leq -2 \text{ if } z \neq 0,$$ so we may conclude that $z = 0$. That $v_i$ is unbreakable then follows from the proof of \cite[Lemma 3.15]{Gre13}. 

Suppose that $w_j$ is unloaded and not tight. Write $w_j = x+y$ with $x = -e_j  + z + \sum_{i = 0}^nx_id_i$ and $y = -z + \sum_{i=0}^n y_i d_i$ for some $z \in E_8$ such that $|x|\geq 3$, $|y|\geq 3$, and $\langle x, y\rangle = -1$. But since $$\langle x, y \rangle = \langle -e_j + z , -z\rangle + \sum_{i = 0}^nx_iy_i,$$
it must be that either (1) $\langle -e_j  + r + z, -z\rangle = -1$ and $x_iy_i\geq 0$ for $i = 0,\ldots, n$, or (2) $\langle -e_j + r + z, -z\rangle = 0$ and $x_ky_k = -1$ for a unique index $k$ and $x_iy_i = 0$ for all other indices. In either scenario, we must have that $|z| \leq 2$ since $w_j|_{E_8} = -e_j$. In (1), we must have $x_i,y_i \in \{0,1\}$ for $i = 1, \ldots, n$ and $\langle z , \tau\rangle \geq 1$, and in particular $|z| = 2$. Then since $|z| = 2$ and $\langle -e_j , -z\rangle = 1$, we must have $z_j \geq 1$, in which case $\langle -e_j + z, \tau \rangle \geq 0$ and therefore $x_i = 0$ for $i = 1, \ldots, n$. But then $|x| = \langle -e_j + z, -e_j + z\rangle = 2$. In (2), without loss of generality we may assume $z = 0$, so that $y = \sum_{i = 1}^ny_id_i$, $y_k = -1$ (and so $x_k = 1$, and $y_i = 1$ if and only if $i \in \text{supp}(w_j)$. Then either $|y| = 2$ or $k = \max(\text{supp}(y))$, in which case $w_j << x$.

Suppose that $w_j$ is loaded and not tight. Write $w_j = x+y$ with $x = -e_j  + r + z + \sum_{i = 0}^nx_id_i$ and $y = -z + \sum_{i=0}^n y_i d_i$ for some $z \in E_8$. Then, as in the case $w_j$ is unloaded, either (I) $\langle -e_j + z, -z\rangle = -1$ and $x_iy_i\geq 0$ for $i = 0,\ldots, n$, or (II) $\langle -e_j + z, -z\rangle = 0$ and $x_ky_k = -1$ for a unique index $k$ and $x_iy_i = 0$ for all other indices. In (I), we must then have $x_i,y_i \in \{0,1\}$ for $i = 0,1, \ldots, n$ and (I.A) $|z| = 2$ and $\langle -e_j + r + z, -z\rangle = -1$ or (I.B) $|z| = 4$ and $\langle -e_j + r + z, -z\rangle = -1$. Items (2) and (3) of Proposition \ref{prop:technicalpairings} show that there are no such $z$ satisfying $w_j = x+y$, $x, y \in (\tau)^\perp$, $|x|,|y|\geq 3$, and $\langle x, y\rangle = -1$ for any $E_8$-changemaker $\tau$ and loaded, non-tight $w_j$. In (II), we are in the hypotheses of item (1) of Proposition \ref{prop:technicalpairings}, which shows that there are again no such $z$.
\end{proof}

\subsection{Standard basis elements and intervals.}

Taken together, (1) of Proposition \ref{prop:linearirreducibles} and Lemma \ref{lem:irreducible} then imply the following about the standard basis elements of $(\tau)^\perp$. 

\begin{prop}
For each standard basis element $v \in \mathcal S$, there is some interval $T(v)$ such that $v = \epsilon(v)[T(v)]$ for some $\epsilon(v)\in \{\pm 1\}$. \qed
\end{prop}

The interval $T(v)$ is either breakable, in which case $v$ is tight, or it is unbreakable. For an unbreakable interval $T$ with $|[T]| \geq 3$, then $T$ contains a single vertex with norm $\geq 3$. Let $z(T)$ denote the unique vertex of norm $\geq 3$ contained in the unbreakable interval $T$. Breakable intervals play an important role in the analysis of standard bases, and we record how breakable intervals may pair against unbreakable intervals in a linear lattice in the following lemma.

\begin{lem}
Suppose that $T$ is a breakable interval, and that $V$ is an unbreakable interval. If $|[V]|\geq 3$, then $[T]\cdot [V]$ equals

\begin{enumerate}
    \item $|[V]| -1$, iff $V \prec T$;
    \item $|[V]| -2$, iff $z(V) \in T$ and $V \pitchfork T$, or $|[V]| =3$, and $V\dag T$;
    \item $1$, iff $|[V]| = 3$, $z(V) \in T$, $V \pitchfork T$, and $\epsilon_V\epsilon_T=\epsilon$; 
    \item $-1$, iff $V \dag T$; or
    \item 0, iff $z(V) \not \in T$ and either $V$ and $T$ are distant or $V \pitchfork T$.
\end{enumerate}
If $|[V]| = 2$, then $|[V]\cdot [T]| \leq 1$, with equality iff $V$ and $T$ abut. 
\end{lem}
\begin{proof}[Proof sketch] The result follows by using the fact that $V$ is unbreakable and conditioning on how $V$ meets $T$ and whether or not $|[V]|\geq 3$.
\end{proof}

Greene uses this characterization of pairings together with the structure of changemaker basis elements to then prove the following lemmas.

\begin{lem}[Corollary 4.3 of \cite{Gre13}]
    A changemaker basis $\mathcal V$ contains at most one breakable vector, and it is tight. \qed
\end{lem}

\begin{lem}[Lemma 4.4 of \cite{Gre13}]
    Given a pair of unbreakable vectors $v_i, v_j \in \mathcal V$ with $|v_i|$, $|v_j|\geq 3$, we have $|v_i\cdot v_j|\leq 1$, with equality if and only if $T(v_i) \dag T(v_j)$ and $\epsilon(v_i)\epsilon(v_j) = -v_i\cdot v_j$. \qed
\end{lem}

\begin{cor}[Corollary 4.5 of \cite{Gre13}]
    If $T(v_i)$ and $T(v_j)$ are distinct unbreakable intervals with $|v_i|$, $|v_j| \geq 3$, then $z(v_i) \neq z(v_j)$. \qed
\end{cor}

We collect here one more observation about tight vectors in $\mathcal S$.

\begin{lem}\label{lem:vttightnowjtight}
    If $v_t \in \mathcal V$ is tight, then no $w_j \in \mathcal W$ is tight.
\end{lem}

\begin{proof}[Sketch of proof.]
If $v_t$ and $w_j$ are both tight, then $3 < |v_t|-2 = v_t \cdot w_j < |w_j|-2$, so $v_t \pitchfork w_j$ and $w_j -v_t = \pm([T(w_j)-T(v_t)] - [T(v_t)-T(w_j)])$ is reducible. However, $w_j -v_t = w_j|_{E_8} + 2d_t + \sum_{i = t+1}^n d_i$ is irreducible, the proof of which is essentially the same as the proof of Lemma \ref{lem:irreducible} (cf. Proposition \ref{prop:technicalpairings}). 
\end{proof}

\subsection{The intersection graph.}\label{sec:forbidden}
We introduce here the notion of the \emph{intersection graph} of a collection of intervals, upon which the basis of our analysis in Section \ref{sec:identifying} rests. The intersection graph is a variant of the notion of the pairing graph adapted to the study of a collection of intervals in a linear lattice.

Recall that every element of the standard basis $\mathcal S = \mathcal V \cup \mathcal W$ is represented by an interval in the pairing graph of the vertex basis of $L$. We write $\bar {\mathcal S}$, $\bar {\mathcal V}$, $\bar {\mathcal W}$ to denote the subset of unbreakable intervals in $\mathcal S$, $\mathcal V$, $\mathcal W$, respectively.

\begin{defin}
Given a collection of intervals $\mathcal S = \{T_1, \ldots, T_k\}$ whose classes are linearly independent, we define the \emph{intersection graph} as
\begin{equation*}
    G(\mathcal S) = (\mathcal S, \mathcal E), \ \ \mathcal E = \{(T_i, T_j)\ | \ T_i \text{ abuts } T_j\}.
\end{equation*}
\end{defin}

We now collect several key properties of the intersection graph.

\begin{defin}\label{def:claw}
The \emph{claw} $(i;j,k,l)$ is the graph $Y = (V,E)$ with 
\[
    V = \{i,j,k,l\} \ \ \text{ and } \ \ E = \{(i,j),(i,k),(i,l)\}.
\]
A graph $G$ is \emph{claw-free} if it does not contain an induced subgraph isomorphic to $Y$. 
\end{defin}

Equivalently, if three vertices in $G$ neighbor a fourth, then some two of them neighbor.

\begin{lem}\label{lem:clawfree}
$G(\mathcal S)$ is claw-free.
\end{lem}
\begin{proof}
If some interval $T_i$ abuts three intervals $T_j, T_k, T_l$, then it abuts some two at the same end, and then those two abut.
\end{proof}

\begin{defin}
A \emph{heavy triple} $(T_i, T_j, T_k)$ consists of distinct intervals of norm $\geq 3$ contained in the same component of $G(\mathcal S)$, none of which separates the other two in $G(\mathcal S)$. In particular, if $(x_i, x_j, x_k)$ spans a triangle, then it spans a \emph{heavy triangle}.
\end{defin}

\begin{lem}[Lemma 4.10 of \cite{Gre13}]
$G(\bar{\mathcal S})$ does not contain a heavy triple. \qed
\end{lem}

\begin{lem}
    If $\hat G(\mathcal S)$ is connected, then so is $G(\mathcal S)$. 
\end{lem}

\begin{proof}
If $\hat G(\mathcal S)$ is connected, then $(\tau)^\perp$ is indecomposable, so $(\tau)^\perp \cong \Lambda(p,q)$. If now $G(\mathcal S)$ is disconnected, then there is a pair $v, v' \in \mathcal S$ with $v \cdot v' \neq 0$ but no path from $v$ to $v'$ in $G(\mathcal S)$. It follows that $v \pitchfork v'$, $T(v) = \{x_i, \ldots, x_j\}$, $T(v') = \{x_{i'}, \ldots, x_{j'}\}$, and, without loss of generality, $i<i'$ and either $i'<j<j'$ or $j'<j$. If $i'<j<j'$, then either $[\{x_i, \ldots, x_{i'-1}\}]$, $[\{x_{i'}, \ldots, x_j\}]$, or $[\{x_{j+1}, \ldots, x_{j'}]$ is in the span of $\mathcal S \setminus \{v, v'\}$, or else $\mathcal S$ does not generate $(\tau)^\perp$, but then there is a path from $v$ to $v'$. If $j'<j$, then either $[\{x_i, \ldots, x_{i'-1}\}]$ or $[\{x_{j'+1}, \ldots, x_j\}]$ is in the span of $\mathcal S \setminus \{v, v'\}$, or else $\mathcal S$ does not generate $(\tau)^\perp$, but then there is a path from $v$ to $v'$.
\end{proof}

\begin{lem}\label{0pairingmediated}
    For $x$, $y$, $z$ irreducible in $\Lambda(p,q)$, if $|x|, |y|\geq 3$, $|z| = 2$, $x\cdot y = 0$, and $|x\cdot z| = |y\cdot z| = 1$, then $\epsilon(x) \epsilon(y) = (x\cdot z)(y\cdot z)$.
\end{lem}

\begin{proof}
    In this scenario either $x\dag z$ and $y \dag z$, in which case $\epsilon(x)\epsilon(z) = -x\cdot z$ and $\epsilon(y)\epsilon(z) = -y\cdot z$, or $z \prec x$ and $z \prec y$, in which case $\epsilon(x)\epsilon(z) = x \cdot z$ and $\epsilon(y)\epsilon(z) = y \cdot z$, and in either case the lemma follows.
\end{proof}

\begin{lem}\label{pitchforkmediated}
    For $x$, $y$, $z$ irreducible in $\Lambda(p,q)$, if $|x|, |y|\geq 3$, $x\pitchfork y$, $|z| = 2$, and $|x\cdot z| = |y\cdot z| = 1$, then $\epsilon(x) \epsilon(y) = -(x \cdot z)(y \cdot z)$
\end{lem}

\begin{proof}
    In this scenario, without loss of generality, $z\prec x$ and $z \dag y$, so $\epsilon(x)\epsilon(z) = x\cdot z$ and $\epsilon(y) \epsilon(z) = -y\cdot z$. 
\end{proof}

\begin{defin}
    We say that \emph{there is a sign error between $x$ and $y$ mediated by} $z$, if there is a triple $x$, $y$, $z \in \mathcal S$ standing in contradiction to lemmas (\ref{0pairingmediated}) and (\ref{pitchforkmediated}). We occasionally drop reference to the vectors $x$, $y$, or $z$ if it is clear from context.
\end{defin}

\begin{prop}
    If $(\tau)^\perp \cong \Lambda(p,q)$, then there are no sign errors. \qed
\end{prop}

\begin{lem}
    If $|v_i|, |w_j|\geq 3$, $v_i$ and $w_j$ are unbreakable, and $v_i\cdot w_j \neq 0$, then $v_i \not \pitchfork w_j$.
\end{lem}

\begin{proof}
    If $v_i \pitchfork w_j$, then $|v_i| = |w_j|$ and $v_i \cdot w_j = |v_i|-2$. Then $v_i-w_j$ is reducible, but $|v_i - w_j| = 4$, and $v_i -w_j = e_j -d_i + d_m$ for some $m$, and $e_j \not \in (\tau)^\perp$, a contradiction.
\end{proof}

\begin{lem}
    If $|w_i|, |w_j|\geq 3$, $w_i$ and $w_j$ are unbreakable, and $w_i \cdot w_j \neq 0$, then $w_i\not\pitchfork w_j$.
\end{lem}

\begin{proof}
    If $w_i \pitchfork w_j$, then $|w_i|= |w_j|$ and $w_i \cdot w_j = |w_i|-2$, and so $|w_i-w_j| = 4$ and $w_i -w_j$ is reducible, but $e_j -e_i$ is irreducible in $(\tau)^\perp$ since $e_j\cdot \tau = e_i\cdot \tau \neq 0$.
\end{proof}

\begin{lem}[Lemma 3.8 of \cite{Gre13}]
Given a cycle $C \subset G(\mathcal S)$, the intervals in $V(C)$ abut pairwise at a common end. That is, there exists an index $j$ such that each $T_i \in V(C)$ has left endpoint $x_{j+1}$ or right endpoint $x_j$. In particular, $V(C)$ induces a complete subgraph of $G(\mathcal S)$. \qed
\end{lem} 

In the sequel, we will abuse notation and use $v_i$ and $w_j$ to refer both to the elements of $\mathcal S$ and to the intervals $T(v_i)$ and $T(w_j)$ that they represent in $L$.

\section{Identifying $E_8$-changemaker embeddings.}\label{sec:identifying}
In this section, we identify when the orthogonal complement of an $E_8$-changemaker $\tau = (s, \sigma) \in E_8 \oplus \mathbb Z^{n+1}$ ($n \geq 2$) is a linear lattice, or a lattice that decomposes as the orthogonal sum of linear lattices, and record each such $\tau$ in terms of $s^*$ and $\sigma$. One might assume, given the structural diversity of linear lattices admitting changemaker embeddings, that the full breadth of linear lattices admitting $E_8$-changemaker embeddings is difficult to capture. While there are forty-four distinct countably infinite families of linear lattices and orthogonal sums of pairs of linear lattices that admit $E_8$-changemaker embeddings, \emph{a posteriori} there are only $4$ distinct families of changemaker tails $\sigma$ represented in this census: $(1, \ldots, 1)$, $(1, \ldots, 1, n+1)$, $(1, 2, \ldots, 2)$, and $(1, 1, 2, \ldots, 2)$ $\in \mathbb Z^{n+1}$.

One notices that the claw $(e_1;e_2, e_3, e_5)$ in the $E_8$ Dynkin diagram \ref{fig:E_8Dynkin} persists in $G(\mathcal S)$ if $s^*_1=s^*_2 = s^*_3 = s^*_5 = 0$. By Lemma \ref{lem:clawfree}, we therefore cannot have $s^*_1=s^*_2 = s^*_3 = s^*_5 = 0$. Of most interest is $w_1$, which corresponds to the trivalent vertex in the $E_8$ Dynkin diagram. Clearly, there are only so many ways that the claw in the Dynkin diagram may be resolved given a fixed form for $w_1$. Of next importance to us is the vector $w_5$---unlike $w_2$ and $w_3$, $w_5$ is always unloaded, and thus $w_1 \cdot w_5$ is easier to control since $w_5\cdot -e_1 = -1$ no matter what $w_5$ is. Our general strategy for identifying when $L$ is a linear lattice is: first we fix a changemaker tail $\sigma$, then we fix $w_1$, then, generically but not exclusively, we fix $w_5$, and then check whether $\{v_1, \ldots, v_n, w_1, w_5\}$ can be completed to a full standard basis $\mathcal S$ whose intersection graph contains none of the forbidden features detailed in Section \ref{sec:forbidden}.

This section, devoted entirely to the analysis necessary necessary to characterize which linear lattices admit $E_8$-changemaker embeddings, consists of four general subsections, corresponding to the three general types of indecomposable changemaker bases (cf. Sections 6, 7, and 8 of \cite{Gre13}) and to the case when the sublattice of $L$ generated by $\mathcal V$ is decomposable (cf. Section 4 of \cite{Gre13}): in the first subsection we treat the case when every element of $\mathcal V$ is just right; in the second subsection we treat the case when $\mathcal V$ contains a gappy vector but no tight vector; in the third subsection we treat the case when there is a tight $v_t \in \mathcal V$. Throughout, we rely extensively on the classification of changemaker bases whose intersection graphs contain no claws, heavy triples, or incomplete cycles carried out in \cite{Gre13} and make frequent reference to structural lemmas there without supplying proof here.

In what follows, we write $A_j$ to mean $\{0\leq i \leq n\colon w_j \cdot d_i \neq 0\}$. We write $v \sim w$ if $v$ and $w$ are adjacent in $G(\mathcal S)$, i.e. $v\dag w$, $v \prec w$, or $w \prec v$.

\subsection{When all vectors are just right.}
First, we treat the case when $\sigma = (1, \ldots, 1) \in \mathbb Z^{n+1}$.
\begin{prop}\label{prop:alloness10}
Suppose $\sigma = (1, \ldots, 1) \in \mathbb Z^{n+1}$ and $n \geq 2$. If $|w_1| = 2$ then one of the following holds:
\begin{enumerate}
    \item $s^* = (0, 1, n+1, 0, 0, 0, 0, 0)$,
    \item $s^* = (0, 1, n+1, 1, 0, 0, 0, 0)$,
    \item $s^* = (0, 1, 0, 0, n+1, 0, 0, 0)$,
    \item $s^* = (0, 1, 0, 0, n+1, 1, 0, 0)$,
    \item $s^* = (0, n+2, n+1, 0, 0, 0, 0, 0)$,
    \item $s^* = (0, n+2, n+1, 1, 0, 0, 0, 0)$,
    \item $s^* = (0, n+2, 0, 0, n+1, 0, 0, 0)$, or
    \item $s^* = (0, n+2, 0, 0, n+1, 1, 0, 0)$.
\end{enumerate}
\end{prop}

\begin{proof}
Observe that if $w_j \in \mathcal W$ is unloaded, then either $|w_j| = 2$, $w_j = -e_5 + d_n$, $w_j = -e_5 + d_1+ \ldots + d_n$, $w_j = -e_5 + d_0 + \ldots d_n$, or $w_j$ is tight. We begin to break our analysis into cases by conditioning on $w_5$.

Case I: $|w_5| = 2$.
First suppose that $|w_5| = 2$. Then both $w_2$ and $w_3$ are unloaded, for if $w_2$ or $w_3$ is loaded, then at least one of $w_2$ and $w_3$ is not tight, and so there is an incomplete cycle since $|w_j|\geq 3$ for some $j \in \{6,7,8\}$ if $w_2$ or $w_3$ is loaded and $|w_5| = 2$. It follows then that $w_2 \sim w_3$ and either $w_2 = -e_2 + d_0 + \ldots + d_n$ or $w_3 = -e_3 + d_0 + \ldots + d_n$ or else there is an incomplete cycle, and furthermore that $|w_6| = |w_7| = |w_8| = 2$. 

Suppose, by way of contradiction, that $w_3$ is tight. Then $w_2 = -e_2 + d_0 + \ldots + d_n$, and $|w_4|\geq 3$ or else $(w_3;v_1, w_1, w_4)$ is a claw. If $|w_4|\geq 3$ and $w_4$ is unloaded, then $w_4 = -e_4 + d_n$ or else $2 \leq w_4 \cdot w_2 \leq |w_4|-2$, but then $(v_1, \ldots, v_n, w_4, w_2, w_3, v_1)$ is an incomplete cycle. If $w_4$ is loaded, then $w_4|_{E_8} = -e_4 + e_2$ and $|A_4|\geq 2$. It follows that either $w_4 = -e_4 + e_2 + d_{n-1} + d_n$, in which case $(v_1, \ldots,w_4, w_1, w_3, v_1)$ is an incomplete cycle, or $ w_4 = -e_4 + e_2 + d_{n-2} + d_{n-1} + d_n$, in which case $(w_3, w_2, w_4, w_1, w_3)$ is an incomplete cycle, or else $2 \leq w_4 \cdot w_2 \leq |w_2|-3$. Conclude that $w_3$ is not tight.

Case I.1: $w_3 = -e_3 + d_0 + \ldots + d_n$.

Suppose that $w_3 = -e_3 + d_0 + \ldots + d_n$. Then either $w_2 = -e_2 + d_n$, in which case $\epsilon_2 = -\epsilon_3$, or $w_2$ is tight. If $w_2 = -e_2 + d_n$, then $w_4$ is unloaded, or else $w_4 = -e_4 + e_2 + 2d_0 + d_1 + \ldots + d_n$, and so $(v_1, \ldots v_n, w_2, w_1, w_4, v_1)$ is an incomplete cycle, and so either $|w_4| = 2$ or $w_4 = -e_4 + d_n$, or else $(w_1, w_2, w_4, w_3, w_1)$ is an incomplete cycle or $2 \leq w_4 \cdot w_3 \leq |w_3| -3$ or $w_4$ is tight, in which case either $(v_1, \ldots, v_n, w_2, w_4, v_1)$ is an incomplete cycle or $\epsilon_2 = \epsilon_4 = \epsilon_3$, a contradiction since $w_2 \cdot w_3 = 1$. So, if $w_2 = -e_2 + d_n$, then either $s^* =  (0, 1, n+1, 0, 0, 0, 0, 0)$ or $s^* = (0, 1, n+1, 1, 0, 0, 0, 0)$. Suppose now that $w_2$ is tight. If $w_4$ is loaded, then $w_4 \sim w_1$, and so $w_4 = -e_4 + e_2 + d_0 + \ldots + d_n$ or else there is an incomplete cycle, but then $w_4 \cdot w_3 = n = |w_3|-3$, which is absurd. Conclude that $w_4$ is unloaded, and so either $|w_4| = 2$ or $w_4 = -e_4 + d_n$, or else there is an incomplete cycle. So, if $w_2$ is tight, then $s^* = (0, n+2, n+1, 0, 0, 0, 0, 0)$ or $s^* = (0, n+2, n+1, 1, 0, 0, 0, 0)$. 

Case I.2: $w_2 = -e_2 + d_0 + \ldots + d_n$.

Suppose that $w_2 = -e_2 + d_0 + \ldots + d_n$. Then either $w_3 = -e_3 + d_n$ or $w_3$ is tight, and in either case there is a claw or an incomplete cycle if $w_4$ is unloaded. If $w_4$ is loaded, then $w_4|_{E_8} = -e_4 + e_2$, and so $w_4 \sim w_1$. But then $w_4 \sim w_3$ or else $(w_1;w_3, w_4, w_5)$ is a claw, and similarly $w_4 \sim w_2$ or else $(w_1;w_2, w_4, w_5)$ is a claw. But then $A_4 = \emptyset$ or else $(w_1, w_3, w_4)$ is a negative triangle, in which case $s^*_4 = |\sigma|_1$, which is absurd since $w_4$ is loaded. 

Case II: $w_5 = -e_5 + d_n$.

Now suppose that $w_5 = -e_5 + d_n$. Then $|w_6|\geq 3$ or else $(w_5;v_n, w_1, w_6)$ is a claw, and so either $w_6 = -e_6 + d_0 + \ldots + d_n$ or $w_6$ is tight, or else $(v_n;v_{n-1}, w_5, w_6)$ is a claw. Suppose that $w_6 = -e_6 + d_0 + \ldots + d_n$. Let $j \in \{2,3\}$. If $A_j = \{n\}$, then $|w_2|$, $|w_3|\geq 3$ and $w_j$ is loaded, or else $(v_n, w_5, w_1, w_j, v_n)$ is an incomplete cycle, but then $(w_2, w_3, w_5)$ is a heavy triple or either $w_2$ or $w_3$ is unloaded and tight, so either $(w_2;v_1, w_1, w_6)$ or $(w_3;v_1, w_1, w_6)$ is a claw. If $A_j = \{i, \ldots, n\}$ for $0 \leq i \leq n-1$, then $2 \leq w_j \cdot w_6 \leq |w_6|-2$. If $w_j$ is tight, then $w_j$ is loaded and $|w_k|\geq 3$ for $k \in \{2,3\}\setminus \{j\}$, or else $(v_1, \ldots, v_n, w_5, w_1, w_j, v_1)$ is an incomplete cycle, but then $w_k$ is unloaded, and $A_k = \{i, \ldots, n\}$ for some $0 \leq i \leq n$, so $(w_k, w_5, w_6)$ is a heavy triple. Suppose instead that $w_6$ is tight. Again, let $j \in \{2, 3\}$. If $A_j = \{n\}$, then $|w_2|$, $|w_3|\geq 3$ and $w_j$ is loaded, or else $(v_n, w_5, w_1, w_j, v_n)$ is an incomplete cycle, but then $(w_5, w_j, w_6, v_1, \ldots, v_n, w_j)$ is an incomplete cycle. If $A_j = \{i, \ldots, n\}$ for some $1\leq i \leq n-1$, then $w_j$ is unloaded, and thus $(w_1, w_j, v_i, \ldots, v_n, w_5, w_1)$ is an incomplete cycle. If $A_j = \{0, \ldots, n\}$, then $w_j = -e_j + d_0 + \ldots + d_n$, and so $(w_1, w_j, w_6, v_1, \ldots, v_n, w_5, w_1)$ is an incomplete cycle. 

Case III: $w_5 = -e_5 + d_1 + \ldots + d_n$.

Now suppose that $w_5 = -e_5 + d_1 + \ldots + d_n$. Then $w_6 = -e_6 + d_n$, or $n =2$ and $w_6 = -e_6 + d_1 +d_2$, or else there is a claw. Let $j \in \{2,3\}$ if $A_j =\{n\}$, then $(w_j, w_5, w_6)$ is a heavy triple. If $A_j = \{i, \ldots, n\}$ for $0 \leq i \leq n$, then $2\leq w_j \cdot w_5 \leq |w_5|-2$, which is absurd since $w_j$ and $w_5$ are unbreakable. If $w_j$ is tight, then $w_j$ is loaded or else $(w_1, w_j, v_1, w_5, w_1)$ is an incomplete cycle, but then $w_k$ is unloaded and $w_k = -e_k + d_i + \ldots + d_n$ for $k \in \{2,3\}\setminus \{j\}$. If $n = 2$ and $w_6 = -e_6 + d_1 + \ldots + d_n$, then $|w_j|\geq 3$ for some $j \in \{2,3\}$, and either $(w_j, w_5, w_6)$ is a heavy triple, or $(v_1;v_2, w_j, w_5)$ is a claw.

Case IV: $w_5 = -e_5 + d_0 + \ldots + d_n$.

Now suppose that $w_5 = -e_5 + d_0 + \ldots + d_n$. Then $|w_j| \geq 3$ for some unloaded $w_j$ with $j \in \{2, 3\}$, and moreover either $A_j= \{n\}$ or $w_j$ is tight. It follows that $|w_k| = 2$ for $k \in \{2, 3\} \setminus \{j\}$, or else, without loss of generality, either $(w_j, v_n, w_k, w_5, w_j)$ is an incomplete cycle or $(w_j,v_1, \ldots, v_n, w_k, w_5, w_j)$ is an incomplete cycle if $w_j$ is tight. Suppose that $w_j = -e_j + d_n$. Furthermore, $|w_7| = |w_8| = 2$ or else there is an incomplete cycle, and either $|w_6| = 2$, or $w_6 = -e_6 + d_n$, or else $w_6 = -e_6 + d_n + d_{n-1}$ and either $A_j = \{n\}$ and $(w_j, w_5, w_6)$ is a heavy triple, or $w_j$ is tight and $(w_5;w_1,w_j,w_6)$ is a claw. In any event, if $j = 3$, then $w_4$ is loaded or else either $|w_4| =2$, and either $w_3 = -e_3 + d_n$ and $(w_3;v_n, w_4, w_5)$ is a claw or $w_3$ is tight and $(w_3;v_1, w_4, w_5)$ is a claw, $w_4 = -e_4 + d_n$ and either $w_3 = -e_3 + d_n$ and $(w_3, w_4, w_5)$ is a heavy triple or $w_3$ is tight and $(w_3, v_1, \ldots, v_n, w_4, w_5, w_3)$ is an incomplete cycle, or $2 \leq w_4 \cdot w_5 \leq |w_5|-2$, which is absurd. If $j = 3$ and $w_4$ is loaded then, since $|w_1| = |w_2| = 2$, either $|w_6| = 2$ and $w_4|_{E_8} = -e_4 + e_5$, in which case $(w_1, w_4, w_6, w_5, w_1)$ is an incomplete cycle, or $w_6= -e_6 + d_n$ and $w_4|{E_8} = -e_4 + e_5 + e_6$, in which case either $w_3 = -e_3 + d_n$ and $(w_1, w_3, w_6, w_7, w_4, w_1)$ is an incomplete cycle or $w_3$ is tight and $(w_3, v_1, \ldots, v_n, w_6, w_7, w_4, w_1, w_3)$ is an incomplete cycle. Conclude that $j \neq 3$. Suppose instead that $j = 2$. Since then $|w_3| = 2$, we must have that $w_2$ is unloaded, for if not, then $(w_2, w_3, w_1, w_5, w_2)$ is an incomplete cycle. It follows that if $|w_4|\geq 3$, then $w_4$ is loaded, for if not, then either $w_4 = -e_4 + d_n$ and either $w_2 = -e_2 + d_n$ and $(w_2, w_4, w_5)$ is a heavy triple or $w_2$ is tight and $(w_2, v_1, \ldots, v_n, w_4, w_5, w_2)$ is an incomplete cycle, or $w_4$ is tight, in which case $w_2 = -e_2 + d_n$, so $(w_4, v_1, \ldots, v_n, w_2, w_5, w_4)$ is an incomplete cycle. Suppose now that $w_4$ is loaded. If $w_4|_{E_8} = -e_4 + e_2$, then $w_2$ is tight and $w_4 = -e_4 + e_2 + d_n$, or else $2 \leq w_4 \cdot w_5 \leq |w_5| -2 \neq |w_4| -2$. But then $(w_2, v_1, \ldots, v_n, w_4, w_5, w_2)$ is an incomplete cycle. If $w_4|_{E_8} = -e_4 + e_2 + e_1 + e_5$, then either $w_6 = -e_6 + d_n$ and $w_4= -e_4 + e_2 + e_1 + e_5$, in which case $w_2$ is tight and so $(w_2, v_1, \ldots, v_n, w_6, w_4, w_2)$ is an incomplete cycle, or $|w_6| = 2$ and $w_4 \sim w_6$, and so $w_4 = -e_4 + e_2 + e_1 + e_5$ or else $(w_6;w_4, w_5, w_7)$ is a claw or $(w_4, w_5, w_6)$ is a negative triangle, but then $w_4 \cdot w_2 = -1$, so $(w_2, w_4, w_6, w_5, w_2)$ is an incomplete cycle. If $w_4|_{E_8} = -e_4 + e_2 + e_1 + e_5 + e_6$, then $w_4 = -e_4 + e_2 + e_1 + e_5 + e_6$ or else $(w_7;w_4,w_6, w_8)$ is a claw, but then either $w_2 = -e_2 + d_n$ and $(w_2, w_6, w_7, w_4, w_2)$ is an incomplete cycle, or $w_2$ is tight and $(w_2, v_1, \ldots, v_n, w_6, w_7, w_4, w_2)$ is an incomplete cycle. Conclude that $w_4$ is not loaded, hence $|w_4| = 2$ or else $(w_4, w_3, w_1, w_5, w_4)$ is an incomplete cycle, hence $s^* = (0, 1, 0, 0, n+1, 0, 0, 0)$, $s^* = (0, 1, 0, 0, n+1, 1, 0, 0)$, $s^* = (0, n+2, 0, 0, n+1, 0, 0, 0)$, or $s^* = (0, n+2, 0, 0, n+1, 1, 0, 0)$.

Case V: $w_5$ is tight.

Suppose lastly, by way of contradiction, that $w_5$ is tight. Then, if $j \in \{2,3\}$, $|w_j|\geq 3$, and $w_j \cdot w_1 = -1$, then $w_j = -e_j + d_0 + \ldots + d_n$ or else there is an incomplete cycle or $w_j$ is loaded and $2 \leq w_5 \cdot w_j \leq |w_j| -3$. It follows that $|w_6|\geq 3$ or else $(w_5;v_1, w_1, w_6)$ is a claw, but then $w_6 = -e_6 + d_n$ or else $2\leq w_6 \cdot w_j = |w_6| -2$ or else $w_6$ is tight and $(v_1;v_2, w_5, w_6)$ is a claw, but then $(v_1, \ldots, v_n, w_6, w_j, w_1, w_5, v_1)$ is an incomplete cycle.
\end{proof}

\begin{prop}\label{prop:alloness11}
If $\sigma = (1, \ldots, 1) \in \mathbb Z^{n+1}$, $n \geq 2$, and $w_1 = -e_1 + d_n$, then one of the following holds:
\begin{enumerate}
    \item $s^* = (1, n+2, n+1, 0, 0, 0, 0, 0)$,
    \item $s^* = (1, n+2, n+1, 1, 0, 0, 0, 0)$,
    \item $s^* = (1, n+2, 0, 0, n+1, 0, 0, 0)$, or
    \item $s^* = (1, n+2, 0, 0, n+1, 1, 0, 0)$.
\end{enumerate}
\end{prop}

\begin{proof}
As in the previous lemma, if $w_j \in \mathcal W$ is unloaded, then either $|w_j| = 2$, $w_j = -e_j + d_n$, $w_j = -e_j + d_1+ \ldots + d_n$, $w_j = -e_j + d_0 + \ldots + d_n$, or $w_j$ is tight. We break our analysis into cases by conditioning on $w_5$.

Case I: $|w_5| = 2$.

Suppose first that $|w_5| = 2$. Let $j \in \{2,3\}$. Then both $w_2$ and $w_3$ are unloaded and $|w_2|$, $|w_3|\geq3$, or else there is a claw at $w_1$. It follows that $w_j = -e_j + d_0 + \ldots + d_n$ and $w_k$ is tight for $\{j, k\} = \{2,3\}$, and so $|w_6| = |w_7| = |w_8| = 2$ or else there is an incomplete cycle. If $w_4$ is loaded then $w_4|_{E_8} = -e_4 + e_2 + e_1$, so $(w_5;w_1,w_4, w_6)$ is a claw. It follows that $|w_4| = 2$, in which case $w_3 = -e_3 + d_0 + \ldots + d_n$ or $(w_3;v_1, w_2, w_4)$ is a claw, or $w_4 = -e_4 + d_n$, in which case $w_3 = -e_3 + d_0 + \ldots + d_n$, or else $w_3$ is tight and $(w_3, v_1, \ldots, v_n, w_4, w_2, w_3)$ is an incomplete cycle. Conclude that $s^* = (1, n+2, n+1, 0, 0, 0, 0, 0)$ or $s^* = (1, n+2, n+1, 1, 0, 0, 0, 0)$. 

Case II: $w_5 = -e_5 + d_n$.

Suppose now that $w_5 = -e_5 + d_n$. Then $(v_n; v_{n-1}, w_1, w_5)$ is a claw.

Case III: $w_5 = -e_5 + d_1 + \ldots + d_n$.

Suppose now that $w_5 = -e_5 + d_1 + \ldots + d_n$. Then, for $j \in \{2, 3\}$, $w_j \cdot v_i = 0$ for all $1 \leq i \leq n$ or else either there is a claw or $w_j$ is unbreakable and $2\leq w_j \cdot w_5 \leq |w_5|-2$, and $A_j \neq \{0, 1, \ldots, n\}$ or else $w_j \cdot w_5 = |w_5|-2 < |w_j|-2$. It follows that $|w_2| = |w_3| = 2$, but then $(w_1;v_n, w_2, w_3)$ is a claw.

Case IV: $w_5 = -e_5 + d_0 + \ldots + d_n$.

Suppose now that $w_5 = -e_5 + d_0 + \ldots + d_n$. Then either $|w_6| = 2$ or $w_6 = -e_6 + d_n$, and so either $|w_7|= 2$ or $w_7$ is tight, in which case either $(w_7;v_1, w_5, w_8)$ is a claw or $(w_1, w_6, w_8)$ is a heavy triple, and therefore $|w_8| = 2$ or else either $w_8$ is tight, in which case $(w_5, w_6, w_7, w_8, w_5)$ is an incomplete cycle if $|w_6| = 2$ and $w_8$ mediates a sign error between $w_1$ and $w_6$ if $w_6 = -e_6 + d_n$, or $w_8 = -e_8 +d_n$ and $(w_8;v_n, w_5, w_7)$ is a claw. Furthermore, $w_j$ is tight and $|w_k| = 2$ for $\{j, k\} = \{2,3\}$, since if $|w_2| = |w_3| = 2$, then $(w_1;v_n, w_2, w_3)$ is a claw, if $A_j = \{n\}$ and $w_j \cdot e_1 = -1$ then $(v_n;v_{n-1}, w_1, w_j)$ is a claw, and if $A_j = \{n\}$ and $w_j \cdot e_1 = 0$ then $w_k \sim w_5$ and either $w_k$ is unbreakable, in which case $(w_2,w_3, w_5)$ is a heavy triple, or $w_k$ is tight, in which case $(w_k, v_1, \ldots, v_n, w_j, w_5, w_k)$ is an incomplete cycle. If $|w_2| = 2$, then $(w_3;v_1, w_4, w_5)$ is a claw if $|w_4| =2$, and if $|w_4|\geq 3$ and $w_4$ is unloaded then $(v_1, \ldots, v_n, w_1, w_4, w_5, w_3, v_1)$ is an incomplete cycle. If $w_4$ is loaded then either $w_4|_{E_8} = -e_4 + e_1 + e_5$, in which case either $|w_6| = 2$ and $w_4 = -e_4 + e_1 + e_5 + d_n$, and therefore $(w_2, w_1, v_n, w_4, w_2)$ is an incomplete cycle, or $w_6=-e_6 + d_n$ and $w_4 = -e_4 + e_1 + e_5$, so $s^*_4 = |\sigma|_1 + 1$, or $w_6 = -e_6 + d_n$ and $w_4|_{E_8} = -e_4 + e_1 + e_5 + e_6$, in which case $(w_2, w_1, w_6, w_7, w_4, w_2)$ is an incomplete cycle. Conclude that $w_2$ is tight and $|w_3| = 2$. If $w_4|_{E_8} = -e_4 + e_2 + e_1$, then $2\leq w_4 \cdot w_5 < |w_5| -2$, if $w_4|_{E_8} = -e_4 + e_2 + e_1 + e_5$ then either $|w_6| = 2$, in which case $w_4 = -e_4 + e_2 + e_1 + e_5$ or else either $(w_6;w_4, w_5, w_7)$ is a claw or $(w_4, w_5, w_6)$ is a negative triangle, but then there is a sign error between $w_4$ and $w_5$ mediated by $w_2$, or $w_6 = -e_6 + d_n$, in which case $w_4 = -e_4 + e_2 + e_1 + e_5$, in which case $(w_2, v_1, \ldots, v_n, w_6, w_4, w_5, w_2)$ is an incomplete cycle, and if $w_4|_{E_8} = -e_4 + e_2 + e_1 + e_5 + e_6$, then $w_4 = -e_4 + e_2 + e_1 + e_5 + e_6$ or else $(w_7;w_4,w_6,w_8)$ is a claw, but then either $(w_2, v_1, \ldots, v_n, w_6, w_2)$ is an incomplete cycle, or there is a sign error between $w_4$ and $w_6$ mediated by $w_2$. If $w_4$ is unloaded, then $|w_4| = 2$ or else $(w_4, v_n, w_1, w_3, w_4)$ is an incomplete cycle. Conclude that $s^* = (1, n+2, 0, 0, n+1, 0, 0, 0)$ or $s^* = (1, n+2, 0, 0, n+1, 1, 0, 0)$.

Case V: $w_5$ is tight.

 Suppose lastly that $w_5$ is tight. Then for $j\in \{2,3\}$ with $w_j$ unloaded, $w_j = -e_j + d_0 + \ldots + d_n$ or else $(v_n;v_{n-1}, w_1, w_j)$ is a claw, and so $|w_6|\geq 3$ or else $(w_5;v_1, w_j, w_6)$ is a claw. It follows that $w_6 = -e_6 + d_n$ since $w_6$ is not tight and $w_6 \cdot w_2 \leq 1$, but then $(v_1, \ldots, v_n, w_6, w_j, w_5, v_1)$ is an incomplete cycle.
\end{proof}

\begin{lem}
If $\sigma = (1, \ldots, 1) \in \mathbb Z^{n+1}$, $n \geq 2$, and $w_1 = -e_1 + d_1 + \ldots + d_n$, then $(\tau)^\perp$ is not a linear lattice. 
\end{lem}

\begin{proof}
If $w_1 = -e_1  +d_1 + \ldots + d_n$, then for $j \in \{2,3,5\}$ with $w_j$ unloaded, either $w_j = -e_j + d_n$ or $n = 2$ and $w_j = -e_j + d_1 + d_2$, or else $w_j$ is tight and $(v_1;v_2, w_1, w_j)$ is a claw. It follows that $|w_j|\geq3$ for at most one $j \in \{2,3,5\}$ with $w_j$ unloaded or else there is a heavy triple or a claw at either $v_1$ or $v_n$, so there is a claw at $w_1$ if $w_2$ and $w_3$ are unloaded. If $w_3$ is loaded, then $w_3$ is tight or else $w_3 = -e_3 + e_2 + d_n$ and $(w_1, v_1, \ldots, v_n, w_3, w_1)$ is an incomplete cycle or $2 \leq w_3 \cdot w_1 = |w_3|-4$, but then $(v_1;v_2, w_1, w_3)$ is an incomplete cycle. If $w_2$ is loaded and $|w_3| = 2$, then $|w_5|\geq 3$ or else $(w_1;v_1, w_3, w_5)$ is an incomplete cycle, so unless $w_2$ is tight, then $w_2 \sim w_5$ and either $w_2 \sim w_1$ or $w_2 \sim v_n$, and in either case $(w_1, w_2, w_5)$ is a heavy triple. If $w_2$ is loaded and $|w_3|\geq 2$, then either $w_2$ is tight, in which case $(v_1;v_2,w_1,w_2)$ is a claw, or $w_2 \cdot w_1 = 0$, in which case $w_2 = -e_2 + e_3 + e_4$, or else $w_2 \cdot w_1 = 1$ and $(v_1, \ldots, v_n, w_2, w_1, v_1)$ is an incomplete cycle. If $w_2 = -e_2 + e_3 + e_4$, then $|w_3|\geq 3$ and $|w_4|\geq 3$ and $w_4$ is not tight, or else $(v_1;v_2, w_1, w_4)$ is a claw, so $w_4 = -e_4 + d_n$ or else $2\leq w_4 \cdot w_1 \leq |w_1|-2$, and therefore $(v_1, \ldots, v_n, w_4, w_1, v_1)$ is an incomplete cycle. 
\end{proof}

\begin{prop}\label{prop:alloness1full}
If $\sigma = (1, \ldots, 1) \in \mathbb Z^{n+1}$, $n \geq 2$, and $w_1 = -e_1 + d_0 + \ldots + d_n$, then one of the following holds:
\begin{enumerate}
    \item $s^* = (n+1, 1, 0, 0, 0, 0, 0, 0)$,
    \item $s^* = (n+1, 0, 1, 0, 0, 0, 0, 0)$,
    \item $s^* = (n+1, 0, 0, 0, 1, 0, 0, 0)$,
    \item $s^* = (n+1, 1, 1, 0, 0, 0, 0, 0)$,
    \item $s^* = (n+1, 1, 0, 0, 1, 0, 0, 0)$,
    \item $s^* = (n+1, 0, 1, 0, 1, 0, 0, 0)$,
    \item $s^* = (n+1, 1, 0, 1, 0, 0, 0, 0)$,
    \item $s^* = (n+1, 0, 0, 1, 1, 0, 0, 0)$,
    \item $s^* = (n+1, 0, 0, n+2, 1, 0, 0, 0)$,
    \item $s^* = (n+1, 1, 1, n+2, 0, 0, 0, 0)$,
    \item $s^* = (n+1, 1, 0, n+2, 0, 0, 0, 0)$, or
    \item $s^* = (n+1, 0, 1, n+2, 1, 0, 0, 0)$.
\end{enumerate}
\end{prop}

\begin{proof}
Since $w_1\cdot v_i = 0$ for all $1\leq i \leq n$, $|w_j| = 2$ for at most two $j \in \{2,3,5\}$. We will break the analysis at hand into three cases based on the size of $\{j \in \{2,3,5\} \colon |w_j|\geq 3\}$.

Case I: $|\{j \in \{2,3,5\} \colon |w_j|\geq 3\}| = 1$.

Suppose now that $|w_j|\geq 3$ for a single $j \in \{2,3,5\}$. It follows that $|w_7|=|w_8| = 2$ or else there is a claw at $w_1$. Now, if $|w_6|\geq 3$, then $w_6 = -e_6 + d_n$ and $(w_6;v_n, w_1, w_7)$ is a claw or $w_6$ is tight and $(w_6;v_1, w_1, w_7)$ is a claw. It follows that neither $w_2$ nor $w_3$ is loaded in this case, since if $|w_2|$ or $|w_3|\geq 3$ then $|w_5| = |w_6| = |w_7| = |w_8| = 2$. Furthermore, we must have $w_j = -e_j + d_n$ or $w_j$ is tight, or else $(w_1;w_2,w_3,w_5)$ is a claw. 

Case I.1: $|w_2|\geq 3$.

Suppose that $|w_2|\geq 3$, $|w_3| = 2$, and $|w_5| = 2$. If $w_4|_{E_8} = -e_4 + e_2$, then $w_2$ is tight and $w_4 = -e_4 + e_2 + d_i + \ldots + d_n$ for some $1\leq i\leq n$, but then $2 \leq w_4 \cdot w_1 \leq |w_1|-2$, which is absurd since $w_4$ and $w_1$ are unbreakable. If $w_4|_{E_8} = -e_4 + e_2 + e_1$, then $w_4 = -e_4 + e_2 + e_1$, or else either $(w_5;w_1, w_4, w_6)$ is a claw or $(w_1, w_4, w_5)$ is a negative triangle, but then $w_4 \dag w_2$, $w_4 \dag w_1$, and $w_1 \pitchfork w_2$, which is absurd. Conclude that $w_4$ is not loaded. If $w_2$ is tight, then either $|w_4| = 2$ and $\hat G(\mathcal S)$ is connected but $G(\mathcal S)$ is not, or $w_4 = -e_4 + e_n$, in which case there is either a sign error between $w_4$ and $w_1$ mediated by $w_2$ or $(v_1, \ldots, v_n, w_4, w_2, v_1)$ is an incomplete cycle. Conclude that $w_2 = -e_2 + d_n$ and either $|w_4| = 2$, $w_4 = -e_4 + d_n$, or $w_4$ is unloaded and tight, and so either $s^* = (n+1, 1, 0, 0, 0, 0, 0, 0)$, $s^* = (n+1, 1, 0, 1, 0, 0, 0, 0)$, or $s^*(n+1, 1, 0, n+2, 0, 0, 0, 0)$.

Case I.2: $|w_3|\geq 3$.

Suppose now that $|w_3|\geq 3$, $|w_2|= 2$, and $|w_5| = 2$. As before, we must have $|w_6| = |w_7| = |w_8| = 2$. It follows that $w_4$ is unloaded, or else $w_4|_{E_8} = -e_4 + e_1$, in which case $(w_2, w_1, w_5, w_4, w_2)$ is an incomplete cycle. If $w_3$ is tight, then $\hat G(\mathcal S)$ is connected but $G(\mathcal S)$ is not if $|w_4| = 2$, and $(w_1;w_2,w_4,w_5)$ is a claw if $w_4 = -e_4 + d_n$. It follows that $w_3 = -e_3 + d_n$, and thus that $|w_4| = 2$ since $(v_n; v_{n-1}, w_3, w_4)$ is a claw if $w_4 = -e_4 + d_n$ and $(w_1;w_2,w_4,w_5)$ is a claw if $w_4$ is tight, so $s^* = (n+1, 0, 1, 0, 0, 0, 0, 0)$.

Case I.3: $|w_5|\geq 3$.

Suppose now that $|w_5| \geq 3$, $|w_2| = 2$, and $|w_3| = 2$. It follows that $|w_6| = |w_7| = |w_8| = 2$ or else there is a claw at $w_1$. Suppose that $w_4|_{E_8} = -e_4 + e_1$. Then $w_5$ is tight and $w_4 = -e_4 + e_1 + d_i + \ldots + d_n$ for some $0\leq i \leq n-1$, but then $2 \leq w_4 \cdot w_5 = |w_4|-3$, which is absurd since $w_4$ is unbreakable. If $w_4|_{E_8} = -e_4 + e_1 + e_5$, then $w_4 = -e_4 + e_1 + e_5$ or else either $(w_6;w_4,w_5,w_7)$ is a claw or $(w_4,w_5,w_6)$ is a negative triangle, so $w_5$ is tight and $(w_5;v_1, w_4,w_6)$ is a claw. Conclude that $w_4$ is unloaded. If $w_5$ is tight, then $\hat G(\mathcal S)$ is connected but $G(\mathcal S)$ is not if $|w_4| = 2$, and either there is a sign error between $w_1$ and $w_4$ mediated by $w_5$ or $(v_1,\ldots, v_n, w_4, w_5, v_1)$ is an incomplete cycle if $w_4 = -e_4 + d_n$. Conclude that $w_5 = -e_5 + d_n$ and either $|w_4| = 2$, $w_4 = -e_4 + d_n$, or $w_4$ is unloaded and tight; hence $s^* = (n+1, 0, 0, 0, 1, 0, 0, 0)$, $s^* = (n+1, 0, 0, 1, 1, 0, 0, 0)$, or $s^* = (n+1, 0, 0, n+2, 1, 0, 0, 0)$.

Case II: $|\{j \in \{2,3,5\}\colon |w_j|\geq 3\}| = 2$

Case II.1: $|w_5| = 2$.

Suppose now that $|w_5| = 2$ and $|w_2|$, $|w_3| \geq 3$. If $w_j$ is loaded for $j \in \{2,3\}$ then $A_j = \{n\}$ and $w_j \dag w_1$ or $w_j$ is tight and $w_1 \prec w_j$, and $|w_k|\geq 3$ for some $k \in \{6, 7, 8\}$ minimal, and so the subgraph of $G(\mathcal S)$ induced by $\{w_j, w_1, w_5, \ldots, w_6, v_1, \ldots, v_n\}$ contains an incomplete cycle. Conclude that at least one $w_j = -e_j + d_n$ and at most one $w_j$ is tight and unloaded for $j \in \{2,3\}$. Note that $|w_7|=|w_8| = 2$, otherwise $w_k \sim w_1$ and either $w_k \sim v_1$ or $w_k \sim v_n$ for $k \in \{7,8\}$ minimal, and either $|w_6| = 2$, in which case $(w_1, w_5, w_6, \ldots, w_k, w_1)$ is an incomplete cycle or $|w_6|\geq 3$, in which case $w_6\sim w_1$ and either $w_6\sim v_1$ or $w_6 \sim v_n$, in which case the subgraph of $G(\mathcal S)$ induced by $(w_1, w_5, w_6, w_7, w_8, v_1, \ldots, v_n)$ contains an incomplete cycle. If $|w_6|\geq 3$, then either $w_6 = -e_6 + d_1$ or $w_6$ is tight. Suppose $w_6 = -e_6 + d_1$. Then $w_j$ is tight and $w_k = -e_k + d_n$ for $\{j, k\} = \{2,3\}$, or else $(w_j, w_k, w_6)$ is a heavy triple. But then $w_k \dag w_6$ and $\epsilon_k = -\epsilon_6$, but $w_j \pitchfork w_k$ and $w_j \pitchfork w_6$ or else there is an incomplete cycle, so $\epsilon_k = \epsilon_j = \epsilon_6$, a contradiction. Suppose that $w_6$ is tight. Then $w_j = -e_j + d_n$ and $w_k = -e_k + d_n$ for $j, k \in \{2,3\}$, but then $\epsilon_j = -\epsilon_k$ since $w_j \dag w_k$ and $\epsilon_j = \epsilon_6 = \epsilon_k$ since $w_j \pitchfork w_6$ and $w_k \pitchfork w_6$. Conclude that $|w_6| = 2$. If $w_4|_{E_8} = -e_4 + e_2$, then $w_2$ is tight, and $2 \leq w_4 \cdot w_1 \leq |w_1|-2$, which is absurd since $w_4$ and $w_1$ are unbreakable. If $w_4|_{E_8} = -e_4 + e_2 + e_1$, then $w_4 = -e_4 + e_2 + e_1$ or else either $(w_5;w_1,w_4,w_6)$ is a claw or $(w_4, w_5, w_6)$ is a negative triangle, but then $w_2$ is tight and $w_4 \dag w_2$, which is absurd since $w_4 \dag w_2$ and $w_4 \dag w_1$ but $w_1 \pitchfork w_2$. Conclude that $w_4$ is unloaded. If $w_2$ or $w_3$ is tight, then $w_4 = -e_4 + d_n$ or else $\hat G(\mathcal S)$ is connected but $G(\mathcal S)$ is not, but then there is a sign error between $w_4$ and $w_1$ mediated by $w_2$ or $(v_1, \ldots, v_n, w_4, w_2, v_1)$ is an incomplete cycle if $w_2$ is tight, and $w_4$ separates $w_3$ from $w_1$ in $G(\mathcal S)$ but $w_1 \pitchfork w_3$ and $w_4 \cdot w_3 = 0$ if $w_3$ is tight. Conclude that $w_2 = -e_2 + d_n$, $w_3 = -e_3 + d_n$, and either $|w_4| = 2$ or $w_4$ is unloaded and tight; hence $s^* = (n+1, 1, 1, 0, 0, 0, 0, 0)$ or $s^* = (n+1, 1, 1, n+2, 0, 0, 0, 0)$.

Case II.2: $|w_2| = 2$. 

Suppose now that $|w_2| = 2$ and $|w_3|$, $|w_5|\geq 3$. It follows that $w_3$ is unloaded, hence at most one $w_j$ is tight and at least one $w_j = -e_j + d_n$ for $j \in \{3, 5\}$. We must furthermore have that $|w_7| = |w_8| = 2$ or else for some $j \in \{7,8\}$ either $(w_3, w_5, w_j)$ is a heavy triple, or one of $w_3, w_5, w_j$ is tight and so either there is a sign error between the other two or the subgraph induced by $\{v_1, \ldots, v_n, w_3,w_5,w_j)\}$ contains an incomplete cycle. It follows that $|w_6| = 2$, or else $w_6 = -e_6 + d_n$ and $(w_6;v_n, w_1, w_7)$ is a claw, or $w_6$ is tight and $(w_6;v_1, w_1, w_7)$ is a claw. If $w_4|_{E_8} = -e_4 + e_1$, then $w_5$ is tight and $|A(w_4)|\geq 1$, but then $2\leq w_4 \cdot w_5 = |w_4|-3$, which is absurd since $w_4$ is unbreakable. If $w_4|_{E_8} = -e_4 + e_1 + e_5$, then $w_4 = -e_4 + e_1 + e_5$ or else either $(w_6;w_4,w_5, w_7)$ is a claw or $(w_4,w_5,w_6)$ is a negative triangle, but then $w_5$ is tight and $w_1 \pitchfork w_5$, which is absurd since $w_4 \dag w_1$ and $w_4 \dag w_5$. Conclude that $w_4$ is unloaded. If $w_3$ or $w_5$ is tight, then $w_4 = -e_4 + d_n$ or else $\hat G(\mathcal S)$ is connected but $G(\mathcal S)$ is not, but then $(v_n;v_{n-1}, w_3, w_4)$ is a claw if $w_5$ is tight and $w_4$ separates $w_1$ from $w_3$, $w_1 \pitchfork w_3$, and $w_4 \cdot w_3 = 0$ if $w_3$ is tight. Conclude that $w_3 = -e_3 + d_n$, $w_5 = -e_5 + d_n$, and either $|w_4| = 2$ or $w_4$ is unloaded and tight, hence $s^* = (n+1, 0, 1, 0, 1, 0, 0, 0)$ or $s^* = (n+1, 0, 1, n+2, 1, 0, 0, 0)$.

Case II.3: $|w_3| = 2$.

Suppose now that $|w_3| = 2$ and $|w_2|$, $|w_5|\geq 3$. If $w_2$ is loaded, then $w_2|_{E_8} = -e_2 + e_4$, so $|w_4|\geq 3$ and $w_4$ is unloaded, and therefore $w_4 \sim w_1$, and furthermore $w_2 \sim w_3$, $w_4 \sim w_3$, and $w_3 \sim w_1$. It follows that either $w_2$ or $w_4$ is tight, or else $(w_1,w_2,w_4)$ is a heavy triple. If $w_2$ is tight, then $w_4 = -e_4 + d_n$, in which case $(v_1, \ldots, v_n, w_4, w_1, w_3, w_2, v_1)$ is an incomplete cycle. If $w_4$ is tight, then $w_2 \cdot v_i = 0$ for all $1\leq i \leq n$ or else the subgraph induced by $(v_1, \ldots, v_n, w_1, w_2, w_3, w_4)$ contains an incomplete cycle, but then $2\leq w_2 \cdot w_1 = |w_1| -3$, which is absurd since both $w_1$ and $w_2$ are unbreakable. Conclude that $w_2$ is not loaded, and therefore at most one $w_j$ is tight and at least one $w_j = -e_j + d_n$ for $j \in \{2,5\}$.

We must furthermore have $|w_7| = |w_8| = 2$, or else for some $j \in \{7,8\}$ either $(w_2,w_5,w_j)$ is a heavy triple, or one of $w_2$, $w_5$, $w_j$ is tight and there is either a sign error between the other two or the subgraph induced by $\{v_1, \ldots, v_n, w_2, w_5, w_j\}$ contains an incomplete cycle. It follows that $|w_6| = 2$, or else $w_6 = -e_6 + d_n$ and $(w_6;v_n, w_1, w_7)$ is a claw, or $w_6$ is tight and $(w_6;v_1, w_1, w_7)$ is a claw. If $w_4|_{E_8} = -e_4 + e_2$, then $w_2$ is tight, so $w_4$ is unbreakable and $w_4 \dag w_5$, in which case $w_4 = -e_4 + e_2 + d_n$ or else $(w_5;v_n, w_4, w_6)$ is a claw, but then $2 = w_4 \cdot w_1 = |w_1| - n-1$, which is absurd. If $w_4|_{E_8} = -e_4 + e_2 + e_1$, then either $w_2$ is tight and $w_5 = -e_5 + d_n$ and $w_4 = -e_4 + e_2 + e_1$, in which case $(v_1, \ldots, v_n, w_5, w_4, w_2, v_1)$ is an incomplete cycle, or $w_2 = -e_2 + d_n$ and $w_5$ is tight and either $2 \leq w_4 \cdot w_5 \leq |w_4|-3$ or $w_4 = -e_4 + e_2 + e_1 + d_0 + \ldots + d_n$ and $(w_5;v_1, w_4, w_6)$ is a claw. If $w_4|_{E_8} = -e_4 + e_2 + e_1 + e_5$, then $w_4 = -e_4 + e_2 + e_1 + e_5$ or else either $(w_6;w_4, w_5, w_7)$ is a claw or $(w_4, w_5, w_6)$ is a negative triangle, but then the subgraph induced by $\{v_1, \ldots, v_n, w_2, w_4, w_5\}$ contains an incomplete cycle. Conclude that $w_4$ is unloaded. If $w_j$ is tight for $\{j, k\} = \{2, 5\}$, then either $(v_1, \ldots, v_n, w_k, w_j, v_1)$ is an incomplete cycle, or $w_4 = -e_4 + d_n$ and either $(v_1, \ldots, v_n, w_4, w_j, v_1)$ is an incomplete cycle or there is a sign error between $w_4$ and $w_k$ mediated by $w_j$, or $|w_4| = 2$ and $\hat G(\mathcal S)$ is connected but $G(\mathcal S)$ is not. Conclude that $w_2 = -e_2 + d_n$ and $w_5 = -e_5 + d_n$, and therefore that $|w_4| = 2$; hence $s^* = (n+1, 1, 0, 0, 1, 0, 0, 0)$.

Case III: $|\{j \in \{2,3,5\} \colon |w_j|\geq 3\}| = 3$.

Suppose lastly, by way of contradiction, that $|w_2|$, $|w_3|$, and $|w_5|$ $\geq 3$. 

Case III.1: $w_2$ or $w_3$ is loaded.

If $w_j$ is loaded for some $j \in \{2, 3\}$, then $w_j \sim w_1$ and either $A_j = \{n\}$ and $w_j \sim v_n$ or $w_j$ is tight and $w_j \sim v_1$. It follows that $w_5 = -e_5 + d_n$, since if $w_5$ is tight, then $A_j = \{n\}$ and $w_j \cdot w_1 = 1$, so either $w_j \dag w_5$ and $(v_1, \ldots, v_n, w_j, w_5, v_1)$ is an incomplete cycle, or $w_j \pitchfork w_5$ and there is a sign error between $w_1$ and $w_j$ mediated by $w_5$. It follows that for $\{j, k\} = \{2,3\}$, either $w_j$ is tight or $w_k$ is tight, otherwise $(w_2,w_3,w_5)$ is a heavy triple. 

Case III.1.a: $w_j$ is tight.

If $w_j$ is tight, then $|w_6| = |w_7| = |w_8| = 2$ or else $(v_1, \ldots, v_n, w_l, w_1, w_j, v_1)$ is an incomplete cycle for some $l \in \{6,7,8\}$, so $w_k = -e_k + d_n$ since $s^*_j \leq |\sigma|_1 + 1 + s^*_5 = |\sigma|_1 + 2$, and therefore $w_j = -e_j + e_k + 2d_0 + \ldots + d_n$. It follows that $w_4 \cdot v_i = 0$ for all $1 \leq i \leq n$ or else $(w_k, w_4, w_5)$ is a heavy triple. Furthermore, since $w_j$ is tight and hence $w_4$ is unbreakable, we must have that $w_4 \cdot w_1 = 0$, and therefore $w_4|_{E_8} = w_4$ or else $A_4 = \{0, \ldots, n\}$ and $2 \leq w_4 \cdot w_1$. If $w_4 = -e_4 + e_2 + e_1 + e_5$, then the subgraph induced by $\{v_1, \ldots, v_n, w_2, w_4, w_5\}$ contains an incomplete cycle. If $w_4 = -e_4 + e_2$, $-e_4 + e_2 + e_1$, $-e_4 + e_1$, or $-e_4 + e_1 + e_5$, then $|w_4 \cdot w_1| = 1$. Conclude that $w_4$ is unloaded, hence $|w_4| = 2$, in which case either $j = 2$ and $(v_1, \ldots, v_n, w_3, w_4, w_2, v_1)$ is an incomplete cycle, or $j = 3$ and $(w_3;v_1, w_1, w_4)$ is a claw.

Case III.1.b: $w_k$ is tight.

If $w_k$ is tight, then $|w_6| = |w_7| = |w_8| = 2$ or else for some $l \in \{6,7,8\}$ $w_l = -e_l + d_n$ and either $w_k \dag w_l$, in which case $(v_1, \ldots, v_n, w_l, w_k, v_1)$ is an incomplete cycle, or $w_k \pitchfork w_l$ and there is a sign error between $w_1$ and $w_l$ mediated by $w_k$. It follows that $w_j = -e_j + e_k + d_n$ since $s^*_j \leq |\sigma|_1 + 1 + s^*5= |\sigma|_1 + 2$, hence $w_j \cdot w_k = -1 \neq \pm (|w_j| - 2)$, so $w_j \dag w_k$ and $(v_1, \ldots, v_n, w_j, w_k, v_1)$ is an incomplete cycle. 

Case III.2: $w_2$, $w_3$, and $w_5$ are unloaded.

If $w_2$, $w_3$, and $w_5$ are unloaded, then, without loss of generality, for $\{j, k, l\} = \{2,3,5\}$, $w_j = -e_j + d_n$, $w_k = -e_k + d_n$, and $w_l$ is tight, or else the subgraph induced by $\{v_1, \ldots, v_n, w_2, w_3, w_5\}$ contains a claw or a heavy triple. Then, either $w_l \dag w_j$ and $(v_1, \ldots, v_n, w_j, w_l, v_1)$ is an incomplete cycle or $w_l \dag w_k$ and $(v_1, \ldots, v_n, w_k, w_l, v_1)$ is an incomplete cycle, or $w_j \pitchfork w_l$ and $w_k \pitchfork w_l$ and there is a sign error between $w_j$ and $w_k$ mediated by $w_l$.
\end{proof}

\begin{lem}
Suppose $w_1$ is tight. Then $(\tau)^\perp$ is not a linear lattice. 
\end{lem}

\begin{proof}
Observe that if $w_1$ is tight, then, since $w_5$ is unloaded, either $|w_5| = 2$, $w_5 = -e_5 + d_n$, or $w_5 = -e_5 + d_0 + \ldots + d_n$. We break our analysis into three respective cases.

Case I: $|w_5| = 2$.

Suppose that $|w_5|= 2$. Then either $|w_6| = 2$, in which case $|w_7| = |w_8| = 2$ or else there is an incomplete cycle, or $w_6 = -e_6 + d_0 + \ldots + d_n$ and $w_6 \prec w_1$, or else $(v_1, \ldots v_i, w_6, w_5, w_1, v_1)$ is an incomplete cycle for some $1\leq i \leq n$. If $|w_6| = 2$, then $w_2$ and $w_3$ are both unloaded and $w_j = -e_j + d_n$ or $w_j = -e_j + d_0 + \ldots + d_n$ for both $j \in \{2,3\}$, or else there is a claw at $w_1$. If $A_2 = A_3 = \{n\}$, then $(w_1,w_2,w_3)$ induces a heavy triple since $w_1 \cdot w_2 = w_1 \cdot w_3 = 0$, and if $A_j = \{n\}$ and $A_k = \{0, \ldots, n\}$ for $\{j,k\} = \{2,3\}$, then $w_j$ separates $w_k$ from $w_1$ in $G(\mathcal S)$, but $w_1 \cdot w_j = 0$ while $w_1 \cdot w_k \neq 0$. If $w_6 = -e_6 +d_0 + \ldots + d_n$, then $|w_2| = |w_3| = 2$ or else $A_j = \{n\}$ for some $j \in \{2,3\}$ and $(v_1, \ldots, v_n, w_j, w_6, w_1, v_1)$ is an incomplete cycle, and so $(w_1;v_1,w_2,w_3)$ is a claw.

Case II: $w_5 = -e_5 + d_n$.

Suppose now that $w_5 = -e_5 + d_n$.
If $A_k = \{0, \ldots, n\}$ for some $k \in \{2,3\}$, then $w_k$ is unloaded, or else $2\leq w_1 \cdot w_k = |w_k|-3$, and $w_5$ then separates $w_1$ from $w_k$ in $G(\mathcal S)$, but $w_5 \cdot w_1 = 0$ and $w_k\cdot w_1 \neq 0$, which is absurd. If $A_j = \{n\}$ for some $j \in \{2,3\}$, then $w_j$ is loaded and $|w_k|\geq 3$ for $\{j,k\} = \{2,3\}$, or else $(w_1,w_j,w_5)$ induces a heavy triple since $w_1 \cdot w_j = w_1 \cdot w_5 = 0$, and furthermore then $w_k = -e_k + d_0 + \ldots + d_n$, which by the argument above is absurd, or else $(w_j, w_k, w_5)$ is a heavy triple. It follows that $|A_2| = |A_3| = \emptyset$, but then $(w_1;v_1, w_2, w_3)$ is a claw.

Case III: $w_5 = -e_5 + d_0 + \ldots + d_n$. 

Suppose lastly that $w_5 = -e_5 + d_0 + \ldots + d_n$. Then for both $j \in \{2,3\}$, either $A_j = \emptyset$ or $A_j = \{n\}$, or else $2 \leq w_j \cdot w_5 \leq |w_5| -2$. If $A_j = \{n\}$, then $w_j$ is loaded and $|w_k|\geq 3$ for $\{j,k\} = \{2,3\}$, or else $w_j$ separates $w_5$ from $w_1$ in $G(\mathcal S)$ but $w_j \cdot w_1 = 0$ and $w_5 \cdot w_1 \neq 0$, but then $w_k = -e_k + d_n$. If $A_2 = A_3 = \emptyset$, then $w_2 \sim w_1$ and $w_3 \sim w_1$, and $(w_1;v_1, w_2,w_3)$ is a claw. 
\end{proof}

Having identified every linear lattice orthogonal to an $E_8$-changemaker with $\sigma = (1, \ldots, 1)$, we now consider the case when every vector in $\mathcal V$ is just right, $|v_i|\geq 3$ for some $1\leq i \leq n$, and $G(\mathcal V)$ is connected. 

We start with a basic observation.

\begin{prop}
    If every element of $\mathcal V$ is just right and $\sigma \neq (1, \ldots, 1)$, then for $r = \min\{i \in \{2, \ldots, n\} \colon |v_i|\geq 3\}$, either $v_r = -d_r + d_{r-1} + \ldots + d_0$ or $v_r = -d_r + d_{r-1} + \ldots + d_1$. \qed
\end{prop}

\begin{lem}
If every element of $\mathcal V$ is just right, $G(\mathcal V)$ is connected, and either $|v_r| = r+1$ for some $2\leq r \leq n$ or $|v_r| = r$ for some $3 \leq r \leq n$, then $(\tau)^\perp$ is not a linear lattice.
\end{lem}

\begin{proof}
Observe that no $w_j \in \mathcal W$ is tight: if $w_j$ is tight then $w_j \sim v_1$ and either $|v_r| = r+1$ for some $2 \leq r \leq n$, in which case $v_r \prec w_j$ and there is an incomplete cycle since $G(\mathcal V)$ is connected, or $|v_r| = r$ for some $3 \leq r \leq n$, in which case $(v_1;v_2, v_r, w_j)$ is a claw. 

Suppose, by way of contradiction, that $G(\mathcal S)$ contains a triangle of the form $(v_i, w_j, v_k)$ with $i < k$ and $w_j$ unbreakable. Then $|v_l| = 2$ for all $1 \leq l \leq i$ and $|v_r| = r+1$, or else there is a heavy triangle. It furthermore follows that $i = r-1$, or else $2 \leq i < r-1$ and $(v_i;v_{i-1}, v_{i+1}, w_j)$ is a claw, or $i = 1$ and $3 \leq r$, in which case $w_j \cdot v_r \geq 1$, so $w_j \cdot v_r = 1$ and $w_j \sim v_r$, and therefore $(v_1, w_j, v_r, \ldots, v_1)$ is an incomplete cycle. It follows that $r \in A_j$ or else there is an incomplete cycle, since $G(\mathcal V)$ is connected, and so $k = r+1$, and $v_k = -d_k + d_{k-1} + d_{k-2}$, or else $w_j \cdot v_k \geq 2$. If $|v_l| = 2$ for all $l \geq r+2$, then $G(\mathcal V)$ is disconnected, in contradiction to one of our assumptions. Let $l\geq r+2$ be minimal such that $|v_l|\geq 3$. Then $w_j \cdot v_l \geq 1$, so in fact $l = r+2$, $v_{r+2} = -d_{r+2} + d_{r+1} + d_r$, and $w_j \dag v_{r+2}$. Since $G(\mathcal V)$ is connected by assumption, we then have a path in $G(\mathcal V)$ from $v_{r-1}$ to $v_{r+2}$, so there is an incomplete cycle $(v_{r-1}, w_j, v_{r+2}, \ldots, v_{r-1})$. Conclude that $G(\mathcal S)$ does not contain a triangle of the form $(v_i,w_j,v_k)$, and therefore $w_j$ has at most one neighbor in $G(\mathcal V)$ for all $w_j$ in $\mathcal W$.

According to the classification of changemaker lattices whose intersection graphs do not contain claws, heavy triples, or incomplete cycles, and whose standard basis elements are all just right, $G(\mathcal V)$ is either a path or a union of a triangle $(v_i, v_{i+2}, v_k)$ and three vertex disjoint paths $P_1$, $P_2$, and $P_3$ emanating from $v_i$, $v_{i+2}$, and $v_k$, respectively \cite[Section 6]{Gre13}. Furthermore, $|v_l| = 2$ for all $v_l \in P_1$, or else there is a heavy triple. It follows that $|w_j|\geq 3$ for at most two values $j \in \{1,2,3,5\}$ and that the neighbor of $w_j$ in $G(\mathcal V)$ is one of two vertices with degree 1 in $G(\mathcal V)$ or else there is a claw, an incomplete cycle, or a heavy triple, but then either there is claw at $w_1$ or there is an incomplete cycle including $w_1$.
\end{proof}

\subsection{When there is a gappy vector but no tight vector.}

\begin{lem}
If $\mathcal V$ contains a gappy element but no tight element, then $(\tau)^\perp$ is not a linear lattice.
\end{lem}

\begin{proof}
If $\mathcal V$ contains a gappy element $v_g = -d_g + d_{g-1} + \ldots + d_j + d_k$ with $k\leq j -2$ and no tight element, then $v_g$ is the unique gappy element of $\mathcal V$, and $v_k$ and $v_{k+1}$ belong to different components of $G(\mathcal V_{g-1})$ \cite[Section 7]{Gre13}. It follows that $|v_r| = r+1$ for $r = \min\{i \in \{2, \ldots, n\}\colon |v_i|\geq 3\}$, and that either there are paths from $v_1$ to $v_k$ and from $v_r$ to $v_{k+1}$ or from $v_1$ to $v_{k+1}$ and from $v_r$ to $v_k$ in $G(\mathcal V_{g-1})$. If $w_j$ is tight, then the union of the pair of paths mentioned above with $v_g$ and $w_j$ produces an incomplete cycle. Conclude that $w_j$ is not tight for any $w_j \in \mathcal W$.

Suppose, by way of contradiction, that $G(\mathcal S)$ contains a triangle of the form $(v_i, w_j, v_k)$, where $i < k$, without loss of generality, and $w_j$ is unbreakable. Then $|v_l| = 2$ for all $1 \leq l \leq i$, so $i = r-1$. We must then have that $r \in A_j$, or else $w_j \dag v_r$ and thus there is an incomplete cycle since $G(\mathcal V)$ is connected. 

Let $l\geq r+1$ be minimal such that $|v_l|\geq 3$. 

If $l = g$, then $g \geq r+2$ and $v_{g} = -d_{g} + d_{g-1} + d_{r-1}$ or else $2 \leq w_j \cdot v_g$, so then $g = r+2$ or else $(v_r, v_{r+1}, \ldots, v_g, v_r)$ is an incomplete cycle. We must then have that $\{r-1,r, r+1\} \subset A_j$, and furthermore that $g = r+2 \in A_j$ or else $w_j \cdot v_g = 2$. Then $n = g$, for if $|v_{g+1}| = 2$, then $(v_g;v_r,w_j,v_{g+1})$ is a claw, and if $|v_{g+1}|\geq 3$, then either $v_{g+1} = -d_{g+1} + d_{g} + d_{g-1}$, in which case $(v_r, v_g, v_{g+1})$ is a heavy triple, or $|v_{g+1}|\geq 4$, in which case $2\leq w_j \cdot v_{g+1}$. It follows then that if $|w_{j'}|\geq 3$ for $j'\neq j$, then $|w_{j'}\cdot v_r| = 1$ and $w_{j'} \cdot v_l = 0$ for all $l \neq r$ or else there is a heavy triple, a claw, or an incomplete cycle. But then $|A_{j'}\cap A_j|\geq 3$, so $w_{j'}$ or $w_{j}$ is loaded, $w_j \cdot w_{j'} = 1$, and $w_j \dag w_j'$, but then $(w_j, w_{j'}, v_r, v_g, w_j)$ is an incomplete cycle.

Conclude that $l \neq g$, so $v_l$ is just right and $w_j \cdot v_l \geq 1$, so $w_j \dag v_l$. We must then have that $v_l \dag v_{r-1}$, otherwise there is a length $\geq 2$ path $v_{r-1}, \ldots, v_l$ in $G(\mathcal V)$ since $G(\mathcal V)$ is connected by assumption, and therefore $(w_j, v_{r-1}, \ldots, v_l, w_j)$ is an incomplete cycle. We must then have that $l = r+1$, and $v_{r+1} = -d_{r+1} + d_r + d_{r-1}$. If $|v_{r+2}| = 2$, then $|v_m| = 2$ for all $r+2 \leq m \leq g$ or else $2 \leq w_j \cdot v_m$ for $m$ minimal such that $m\geq r+3$, $|v_m|\geq 3$, and $v_l$ is just right. Then, $g = r+3$ and $v_g = -d_g + d_{g-1} + d_r$ or else there is a claw or $(v_g, v_{r+1}, w_j)$ is a heavy triple. Since $|v_{r+2}| = 2$, we have $\{r-1,r, r+1, r+2\}\subset A_j$, so either $w_j \cdot v_g = 2$ or $g \in A_j$, in which case $w_j \cdot v_g = 1$ and $(w_j, v_{r+1}, v_{r+2}, v_g, w_j)$ is an incomplete cycle. Conclude that $|v_{r+2}|\geq 3$, hence $v_{r+2}$ is just right, or else $(v_{r+2}, v_{r+1}, w_j)$ is a heavy triple. We must then have that $v_{r+2} = -d_{r+2} + d_{r+1} + d_r$ and $r+2 \in A_j$, or else $w_j \cdot v_{r+2} \geq 2$. Then, since $G(\mathcal V)$ is connected there is a length $\geq 2$ path from $v_{r+1}$ to $v_{r+2}$ in $G(\mathcal V)$, hence there is an incomplete cycle $(w_j, v_{r+1}, \ldots, v_{r+2}, w_j)$. Conclude that there is no triangle of the form $(v_i,w_j,v_k)$ in $G(\mathcal S)$, so each $w_j \in \mathcal W$ has at most one neighbor in $G(\mathcal V)$.

As in the case where $G(\mathcal V)$ is connected and every element of $\mathcal V$ is just right

By the classification of changemaker latices with no claws, heavy triples, or incomplete cycles that contain a gappy vector but no tight vector, there are exactly two $v_i$'s with a single neighbor in $\mathcal V$ such that no heavy triple is formed if $w_j \sim v_i$ for some $w_j$ \cite[Section 7]{Gre13}. But then $|w_j|\geq 3$ for at most two distinct values $j \in \{1,2,3,5\}$ or else there is a heavy triple, but then either there is a claw at $w_1$ or else there is an incomplete cycle containing $w_1$. 
\end{proof}

\subsection{When $t \geq 2$ and $v_t$ is tight.}

\begin{lem}
Let $t \geq 2$ and suppose that $v_t$ is tight. Then $\mathcal V = \mathcal V_t$, $|v_i| = 2$ for $1\leq i\leq t-1$, and $w_1 = -e_1 + d_0 + \ldots + d_{t-1}$.
\end{lem}

\begin{proof}
Let $m = \min(A_1)$. Since by Lemma \ref{lem:vttightnowjtight} $w_1$ is not tight, we have six cases to consider: \emph{Case 0}: $m = n$, \emph{Case I}: $t< m < n$, \emph{Case II}: $t = m$, \emph{Case III}: $m = t-1$, \emph{Case IV}: $0 < m < t-1$, and \emph{Case V}: $m = 0$. 

Case 0: $t< m = n$.

If $m = n$, then $w_1 = -e_1 + d_n$ and $w_1 \sim v_n$. We break this case into three subcases: Case 0.1: $|w_5| = 2$, Case 0.2: $|w_5|\geq 3$ and $n \not \in A_5$, and Case 0.3: $n \in A_5$.

Case 0.1: $|w_5| = 2$.

If $|w_5| = 2$, then $|w_6| = |w_7| = |w_8| = 2$, or else the subgraph induced by $\{v_1, \ldots, v_n, w_1, w_5,$ $\ldots, w_l\}$ for $l = \min\{i \in \{6,7,8\} \colon |w_i| \geq 3\}$ contains an incomplete cycle. It follows that both $w_2$ and $w_3$ are unloaded, and $|w_2|,|w_3|\geq 3$ or else there is a claw at $w_1$. For $j \in \{2,3\}$, if $n \not \in A_j$, then either $(w_1;v_n,w_j,w_5)$ is a claw or $w_j \sim v_n$, but then $|v_n|\geq 3$, and so $(v_n,w_1,w_j)$ is a heavy triple. It follows that $n \in A_2\cap A_3$, hence $n = A_2 \cap A_3$, or else $2 \leq w_2 \cdot w_3 \leq |w_2|-2$. Since $w_2 \dag w_3$ and both $w_2$ and $w_3$ have neighbors in $\mathcal V$, there is a unique $r \in \{1,\ldots, n\}$ such that $w_2 \sim v_r$ and $w_3 \sim v_r$ or else the subgraph induced by $\{v_1, \ldots, v_n, w_2, w_3\}$ contains an incomplete cycle. Furthermore, $|v_r| = 2$, or else $(v_r, w_2, w_3)$ is a heavy triple; thus $r \in A_2 \cap A_3$ and $r = n$, but then $(w_1, w_2, w_3)$ is a heavy triple.

Case 0.2: $|w_5|\geq 3$ and $n \not \in A_5$.

If $|w_5| \geq 3$ and $n \not \in A_5$, then $w_5 \sim w_1$, so $w_5\sim v_n$ and $w_5 \cdot v_i = 0$ for all $i \in \{1, \ldots, t-1, t+1, \ldots, n\}$ or else the subgraph induced by $\{v_1 \ldots, v_n, w_1, w_5\}$ contains an incomplete cycle. Then $w_5 \cdot v_n = 1$ and $|v_n|\geq 3$, but then $(v_n, w_1, w_5)$ is an incomplete cycle.

Case 0.3: $n \in A_5$. 

If $n \in A_5$, then $w_5\cdot w_1 = 0$. If $w_5 \sim v_n$, then $w_5 \sim v_r$ for some $1<r <n$ with $v_n \sim v_r$ or else $(v_n;v_r, w_1, w_5)$ is a claw, and, moreover, $w_5 \cdot v_i = 0$ for all $i \in \{1, \ldots, n-1\}\setminus \{r\}$, or else the subgraph induced by $\{v_1, \ldots, v_n, w_5\}$ contains an incomplete cycle. It follows that $|v_n|\geq 3$ and thus $|v_r| = 2$, or else $(v_r, v_n, w_5)$ is a heavy triple, so $w_5 \cdot v_n = 1$, $w_5 \cdot v_r = -1$, and $v_n \cdot v_r = -1$. It then follows that $v_n$ is just right, so $v_n = -d_n + d_{n-1} + d_{n-2}$ and $w_5 = -e_5 + d_n + d_{n-1} + d_{n-2}$, and so $(v_l, v_n, w_5)$ induces a heavy triple for $l = \max \{i < n-2\colon |v_i|\geq 3\}$.

Case I: $t< m < n$.

If $t < m$ then $v_m$ and $v_{m+1}$ are unbreakable, $v_m \sim v_r$ for some $r<m$, and $v_{m+1} \sim v_s$ for some $s\leq m$. It follows that $v_{m+1} \cdot v_m = 0$, or else $v_{m+1} \dag v_m$, in which case either $w_1 \cdot v_{m+1} = 0$ and $(v_{m};v_r, v_{m+1}, w_1)$ is a claw, or $w_1 \cdot v_{m+1} = 1$, in which case $v_{m+1} \cdot v_m = -1$ or else $(v_m, v_{m+1}, w_1)$ is a negative triangle, but then $(v_m, w_1, v_{m+1}, v_s, \ldots, v_r, v_m)$ is an incomplete cycle. Since $G(\mathcal V)$ is connected, it follows that $m+1 \in A_1$ or else $G(\mathcal V \cup \{w_1\})$ contains an incomplete cycle. Hence, $A_1 = \{m, \ldots, n\}$ or else there is some $m+2\leq k \leq n$ such that $v_k$ is gappy and $w_1 \cdot v_k = 1$, but then the subgraph induced by $\{v_1, \ldots, v_k, w_1)$ contains an incomplete cycle. Furthermore, we must have $|v_i| = 2$ for all $m+2\leq i \leq n$. We now break the analysis down into three subcases based on the assumption that $|w_5| = 2$, $|w_5|\geq 3$ and $A_5 \cap A_1 = \emptyset$, or $A_5 \cap A_1 \neq \emptyset$. 

Case I.1: $|w_5| = 2$.

If $|w_5| = 2$, then, as in Case 0, $|w_6| = |w_7| = |w_8| = 2$ or else there is an incomplete cycle, and therefore $w_2$ and $w_3$ are loaded and $|w_2|,|w_3|\geq 3$. For $j \in \{2,3\}$, if $A_j \cap A_1 = \emptyset$, then $w_j \sim w_1$, so $w_j \sim v_m$ or else there is an incomplete cycle, but then $|v_m|\geq 3$, so $(v_m, w_1, w_j)$ is a heavy triple. If $m \in A_j$, then $m+1 \in A_j$ or else $A_j \cap \text{supp}(v_{m+1}) = \{m\}$, in which case the subgraph induced by $\{v_1, \ldots, v_{m+1}, w_j\}$ contains an incomplete cycle. But then $\{m, m+1\} = A_j \cap \{m, \ldots, n\}$ or else $2 \leq w_j \cdot w_1 < |w_1|-2$, so we must have $m+1 = n$ since $|v_i|= 2$ for all $m+2\leq i \leq n$. If $m \not \in A_j$, then $n\in A_j$, again since $|v_i|= 2$ for all $m+2\leq i \leq n$. It follows that $n \in A_2 \cap A_3$, so $w_2 \cdot w_3 = 1$, and therefore there is a unique $i \in \{1, \ldots, n\}$ with $w_2\sim v_i$ and $w_3 \sim v_i$, so $w_2 = -e_2 + d_n$ and $w_3 = -e_3 + d_n$, but then $(v_{m+1}, w_2, w_3)$ is a heavy triple. 

Case I.2: $|w_5|\geq 3$ and $A_5 \cap A_1 = \emptyset$.

If $|w_5| \geq 3$ and $A_5 \cap A_1 = \emptyset$, then $w_5 \sim w_1$, so $w_5\sim v_m$ and $w_5 \cdot v_i = 0$ for all $i \in \{1, \ldots, t-1, t+1, \ldots, n\}$ or else the subgraph induced by $\{v_1 \ldots, v_n, w_1, w_5\}$ contains an incomplete cycle. Then $w_5 \cdot v_m = 1$ and $|v_m|\geq 3$, but then $(v_m, w_1, w_5)$ is an incomplete cycle.

Case I.3: $A_5 \cap A_1 \neq \emptyset$.

If $A_5 \cap A_1 \neq \emptyset$, then the argument from Case I.1 shows that $n\in A_5$. Moreover, for $j \in \{2,3\}$ if $w_j$ is unloaded and $|w_j|\geq 3$, or if $j =2$, $w_2$ is loaded and $|w_3|=2$, then $n \in A_j$, so $w_j \cdot w_5 = 1$, and therefore there is a unique $i \in \{1, \ldots, n\}$ with $w_j\sim v_i$ and $w_5 \sim v_i$, so $A_j = \{n\}$ and $w_5 = -e_5 + d_n$, but then $(v_{m+1}, w_j, w_5)$ is a heavy triple. 

Case II: $m = t$.

If $m = t$, then $w_1 \cdot v_t = -1$, so either $w_1 \dag 
v_t$ or $w_1 \pitchfork v_t$. We will first show that $w_1
\dag v_t$. Suppose, by way of contradiction, that $w_1 = -
e_1 + d_t$ and $w_1 \pitchfork v_t$. We must have 
$\mathcal V = \mathcal V_t$; if $|v_{t+1}|= 2$, then $t+1 
\in A_1$, so $|w_1\cdot v_t|\leq |w_1|-3$ and therefore 
$w_1 \dag v_1$, and if $v_{t+1} = -d_{t+1} + d_t + d_{t-
1}$, then $v_{t+1}$ separates $v_t$ from $w_1$ in 
$G(\mathcal S)$, but $v_{t+1} \cdot v_t = 0$. Furthermore, 
$w_1 + v_t$ is reducible. Then we may write $w_1 + v_t = -
e_1 + 2d_0 + d_1 + \ldots + d_{t-1}$ as $x + y$ with $x, y 
\in (\tau)^\perp$ and $x \cdot y = 0$, which we will show 
is absurd. If $x \cdot d_i$, $y\cdot d_i \geq 0$ for all 
$1\leq i \leq n$, then either $y|_{E_8} \neq 0$, or $x = -
e_1 + 2d_0 + d_1 + \ldots + d_{t-1}$ and $y = 0$, or $x 
\cdot \tau \neq 0$. If $y|_{E_8} \neq 0$, then $x|_{E_8} 
\cdot y|_{E_8} = -1$, $x\cdot d_0 = y\cdot d_0 = 1$, and 
$\{x \cdot d_i , y \cdot d_i \} = \{0, 1\}$ for all $1 \leq 
0 \leq t-1$. However, if $x|_{E_8} \cdot y|_{E_8} = -1$, 
then $y|_{E_8}$ is a negative root, so $y$ is irreducible 
in $(\tau)^\perp$. Furthermore, $y \cdot w_1 = -1$ and $y 
+ w_1$ is irreducible in $(\tau)^\perp$ since $|
(y+w_1)|_{E_8}|=2$ and $(y + w_1)\cdot d_i\in \{0,1\}$ for 
all $0\leq i \leq n$, so $y \dag w_1$ and $\epsilon(y) = 
\varepsilon(v_t)$; hence there is a sign error between $y$ 
and $w_1$ mediated by $v_t$. If now $x \cdot d_n = -1$, 
then $y \cdot d_n = 1$, $x\cdot d_0 = 1$, $y\cdot d_0 = 
1$, and $y|_{E_8} = 0$, thus $y \cdot \tau \neq 0$. 
Conclude that $w_1 \dag v_t$.

Since $w_1\dag v_t$, either $\mathcal V= \mathcal V_t$, or $|v_{t+1}| = 2$, in which case $(v_t;v_1, v_{t+1}, w_1)$ is a claw, or $v_{t+1} = -d_{t+1} + d_t + d_{t-1}$, in which case $t+1 \in A_1$ or else $(v_1, \ldots, v_{t-1}, v_{t+1}, w_1, v_t, v_1)$ is an incomplete cycle. It follows then that $|v_i| = 2$ for all $t+2 \leq i \leq n$ and $A_1 = \{t, \ldots, n\}$. We will break the remaining analysis into three subcases, as usual, depending on whether $|w_5| = 2$, $|w_5|\geq 3$ and $A_5 \cap A_1 = \emptyset$, or $A_5 \cap A_1 \neq \emptyset$.

Case II.1: $|w_5| = 2$.

If $|w_5| = 2$, then, as before, $|w_6| = |w_7| = |w_8| = 2$ or else there is an incomplete cycle, since the argument that shows that $w_1 \dag v_t$ shows that $w_j \dag v_t$ if $w_j \cdot v_t = -1$ for all $j \in \{1, \ldots, 8\}$. It follows that $w_2$ and $w_3$ are unloaded, and that $|w_2|, |w_3|\geq 3$, or else there is a claw at $w_1$. Furthermore, $A_2 \cap A_1 \neq \emptyset$ and $A_3 \cap A_1 \neq \emptyset$, or else there is an incomplete cycle. Then, for $j \in \{2,3\}$, it follows that either $n \in A_j$, or else $n>t$ and $A_j \cap A_1 = \{t\}$, in which case $w_j \sim v_{t+1}$ and either $w_j \cdot v_t \neq 0$, in which case $w_j \dag v_t$ and thus $(v_1, \ldots, v_{t-1}, v_{t+1}, w_j, v_t, v_1)$ is an incomplete cycle, or $w_j \cdot v_t = 0$, in which case $w_j = -e_j + d_t + d_r$ for some $r < t$, and $w_j \sim v_r$, and therefore either $r < t-1$ and the subgraph induced by $\{v_1, \ldots, v_{t+1}, w_j\}$ contains and incomplete cycle, or $r = t-1$, and $(v_t, v_{t+1}, w_j)$ induces a heavy triple since $v_{t+1} \cdot v_t = w_j \cdot v_t = 0$. Conclude that $\{n\} = A_2 \cap A_3$, so $w_2 \dag w_3$ and either there is an incomplete cycle or $A_2 = A_3 = \{n\}$, in which case $(v_l, w_2, w_3)$ induces a heavy triple for $l = \max\{i \leq n \colon |v_i|\geq 3\}$. 

Case II.2: $|w_5|\geq 3$ and $A_5 \cap A_1 = \emptyset$.

If $w_5 \geq 3$ and $A_5 \cap A_1 = \emptyset$, then $w_5 \dag w_1$ and $A_5 \subset \{0, \ldots, t-1\}$, and so either $|v_i| = 2$ for all $1\leq i \leq t-1$ and $w_5 = -e_5 + d_0 + \ldots + d_{t-1}$, or there is an incomplete cycle. It follows that $t = n$, or else $(v_1, \ldots, v_{t-1}, v_{t+1}, w_5, v_t, v_1)$ is an incomplete cycle. It follows that for $j \in \{2,3\}$ if $|w_j|\geq 3$, then $A_j = \{t\}$, or else $w_j \dag w_5$ and $w_j \sim v_{t-1}$, and therefore $(v_1, \ldots, v_{t-1}, w_j, w_5, v_t, v_1)$ is an incomplete cycle. It follows then that if $|w_j|\geq 3$ for $j \in \{2,3\}$, then $w_j\dag v_t$, using the same argument as for $w_1$ if $w_j$ is unloaded, and therefore $(v_t;v_1, w_5, w_j)$ is a claw.

Case II.3: $A_5 \cap A_1 \neq \emptyset$. 

If $A_5 \cap A_1 \neq \emptyset$, then either $n \in A_5$, or $A_5 \cap A_1 = \{t\}$ and $w_5 \cdot v_{t+1} = 1$, in which case either $t = n$, or $t< n$ and either $w_5 \cdot v_t \neq 0$, in which case $w_5 \dag v_t$ and thus $(v_1, \ldots, v_{t-1}, v_{t+1}, w_5, v_t, v_1)$ is an incomplete cycle, or $w_5 \cdot v_t = 0$, in which case $w_5 = -e_5 + d_t + d_r$ for some $r < t$, and $w_5 \sim v_r$, and therefore either $r < t-1$ and the subgraph induced by $\{v_1, \ldots, v_{t+1}, w_5\}$ contains and incomplete cycle, or $r = t-1$, and $(v_t, v_{t+1}, w_5)$ induces a heavy triple since $v_{t+1} \cdot v_t = w_5 \cdot v_t = 0$. Conclude that $n \in A_5$. If $|w_2| = |w_3| = 2$, then $(w_1;v_t, w_2, w_3)$ is a claw, so let $j \in \{2,3\}$ with $|w_j|\geq 3$ and $w_j \cdot -e_1 = -1$. As with $w_5$ in Case II.2, either $n \in A_j$, in which case $w_j \dag w_5$ and there is a unique $v_i$ such that $w_j \sim v_i$ and $w_5 \sim v_i$, so $i = n$ and $(v_l, w_j, w_5)$ induces a heavy triple for $ l = \max\{i \in \{1, \ldots, n\}\colon |v_i|\geq 3$, or $|v_i| = 2$ for all $1 \leq i \leq t-1$ and $w_j = -e_j + d_0 + \ldots + d_n$, in which case $(v_t;v_1,w_j, w_5)$ is a claw.

Case III: $m = t-1$. 

If $m = t-1$, then $w_1 \dag v_{t-1}$, so either $w_1 = -e_1 + d_t$ and $w_1 \pitchfork v_t$, or $t \in A_1$, in which case $|v_i| = 2$ for all $t+1 \leq i \leq n$ or else $v_{t+1} = -d_{t+1} + d_t + d_{t-1}$ and $(v_t, v_{t+1}, w_1)$ induces a heavy triple since $v_{t+1} \cdot v_t = w_1 \cdot v_t = 0$, and so furthermore $w_1 = -e_1 + d_{t-1} + \ldots + d_n$.

Case III.1 $|w_5| = 2$.

If $|w_5| = 2$, then $|w_6| = |w_7| = |w_8| = 2$ or else there is an incomplete cycle, so then $|w_2|, |w_3|\geq 3$ and $w_2$ and $w_3$ are both unloaded. Let $j \in \{2,3\}$. If $A_j \cap A_1 = \emptyset$, then $w_1 = -e_1 + d_{t-1}$, and $n \in A_j$ since $|v_i| = 2$ for all $t+1 \leq i \leq n$, but then the subgraph induced by $\{v_1, \ldots, v_n, w_1, w_j\}$ contains an incomplete cycle. Suppose now that $A_j \cap A_1 \neq \emptyset$. If $w_1 = -e_1 + d_{t-1}$, then either $2 < t$ and $(v_{t-1};v_{t-2}, w_1, w_j)$ is a claw or $(v_1, \ldots, v_{t-1}, w_j, v_t, v_1)$ is an incomplete cycle, or $t = 2$ and $(v_1;v_2, w_1, w_j)$ is a claw or $(v_1, v_2, w_j)$ is a negative triangle. If $w_1 = -e_1 + d_{t-1} + \ldots + d_n$, then either $n \in A_j$ or $w_j = -e_j + d_0 + \ldots + d_{t-1}$ and $w_j \prec v_t$, or else $w_j = -e_j + d_{t-1}$ and there is either a claw, an incomplete cycle, or a negative triangle as before. If $n \in A_2 \cap A_3$, then $w_2 \dag w_3$, so the subgraph induced by $\{v_1, \ldots, v_n\}$ contains an incomplete cycle, unless $w_2 = -e_2 + d_n$ and $w_3 = -e_3 + d_n$, in which case $(v_t, w_2, w_3)$ induces a heavy triple. If $w_j = -e_j + d_0 + \ldots + d_{t-1}$, then $t= n$ or else $(v_t;v_1, v_{t+1}, w_j)$ is a claw, and if then $w_k = -e_k + d_t$ for $\{j, k\} = \{2,3\}$, then $(v_t;v_1, w_j,w_k)$ is a claw, and if $w_k = -e_k + d_{t-1} + d_t$, then $(v_1, \ldots, v_{t-1}, w_1, w_k, w_j, v_n, v_1)$ is an incomplete cycle. 

Case III.2 $|w_5|\geq 3$ and $A_5 \cap A_1 = \emptyset$.

If $|w_5|\geq 3$ and $A_5 \cap A_1 = \emptyset$, then $w_1 = -e_1 + d_{t-1}$ and $A_5 \subset \{t, \ldots, n\}$, in which case the subgraph induced by $\{v_1, \ldots, v_n, w_1, w_5\}$ contains an incomplete cycle.

Case III.3 $A_5 \cap A_1 \neq \emptyset$. 

If $A_5 \cap A_1 \neq \emptyset$, then it follows from the argument in Case III.1 that $w_1 = -e_1 + d_{t-1}+ \ldots + d_n$ and either $w_5 = -e_5 + d_n$ or $w_5 = -e_5 + d_0 + \ldots + d_{t-1}$. Let $j \in \{2,3\}$ and suppose that $w_j \cdot -e_1 = -1$. Then, if $|w_j|\geq 3$, precisely one element $w_k$ of $\{w_j, w_5\}$ has $n \in A_k$ while the other element $w_l$ has $A_l = \{0, \ldots, t-1\}$, in which case either $(v_t;v_1, w_k, w_l)$ is a claw or $(v_1, \ldots, v_{t-1}, w_1, w_k, w_j, v_n, v_1)$ is an incomplete cycle, as in Case III.1.

Case IV: $0<m < t-1$.

If $0<m<t-1$, then $|v_{m+1}|\geq 3$ or else there is a claw at $v_m$, and $m+1 \in A_1$ or else the subgraph induced by $\{v_1, \ldots, v_{m+1}, v_t, w_1\}$ contains an incomplete cycle. It follows, since every element of $\{v_1, \ldots, v_{t-1}\}$ is just right \cite[Lemma 8.1]{Gre13} and $w_1 \not \sim v_t$, that $w_1 = -e_1 + d_m + \ldots + d_{t-1}$, and furthermore that $|v_i| = 2$ for all $m + 2\leq i \leq t-1$. Since $|v_{m+1}| \geq 3$, we have $n = t$ by \cite[Lemma 8.2]{Gre13}.

Case IV.1: $|w_5| = 2$. 

If $|w_5| = 2$, then $|w_6| = |w_7| = |w_8| = 2$ or else there is an incomplete cycle, as before. It follows that $w_2$ and $w_3$ are unloaded, and $|w_2|$, $|w_3|\geq 3$. Let $j \in \{2,3\}$. If $A_j \cap A_1 = \emptyset$, then $A_j \subset \{1, \ldots, m-1, t\}$. If $A_j = \{t\}$, then $(v_t; v_1, v_r, w_j)$ is a claw for the unique $2\leq r \leq m+1$ with $|v_r| = r+1$. If $A_j \cap \{1, \ldots, m-1\} \neq \emptyset$, then there is some $k \leq m$ with $|v_k|\geq 3$ such that $w_j \cdot v_k = 1$, and some $l < k$ such that $w_j \cdot v_l = -1$, thus the subgraph induced by $\{v_1, \ldots, v_k, v_t, w_j\}$ contains an incomplete cycle. It follows that $A_j \cap A_1 \neq \emptyset$, so either $t-1 \in A_j$ or $A_j = \{m\}$, in which case the subgraph induced by $\{v_1, \ldots, v_{m+1}, v_t, w_j\}$ contains an incomplete cycle. If $t-1 \in A_j$, then either $w_j = -e_j + d_{t-1}$ or $w_j = -e_j + d_{t-1} + d_t$; $|A_j \cap A_1| = 1$ and $A_j \cap \{1, \ldots, m-1\}\neq 0$, in which case there is a pair $1 \leq l < k$ with $v_l$ and $v_k$ in separate components of $G(\bar{\mathcal V})$, $w_j \cdot v_k = 1$, and $w_j \cdot v_l = -1$, and so there is an incomplete cycle in the subgraph induced by $\{v_1, \ldots, v_t, w_j\}$; or $w_j = -e_j + d_{t-2} + d_{t-1}$ and $m = t-2$, but then $w_1 \dag w_j$ and there is a sign error between $w_1$ and $w_j$ mediated by $v_t$. It follows that for $\{j, k\} = \{2,3\}$, $w_j = -e_j + d_{t-1}$ and $w_k = -e_k + d_{t-1} + d_t$, or or else either $w_j \cdot w_k = 2$ or $A_j = A_k = \{n\}$ and so there is a sign error between $w_j$ and $w_k$ mediated by $v_t$, but then $(v_{m+1}, w_j, w_k)$ is a heavy triangle. 

Case IV.2: $|w_5|\geq 3$ and $A_5 \cap A_1 = \emptyset$. 

If $|w_5|\geq 3$ and $A_5 \cap A_1 = \emptyset$, then, as in the argument for $w_j$ in Case Iv.1, either $w_j = \{t\}$ and $(v_t;v_1, v_r, w_5)$ is a claw, or $w_5 \cdot v_k = 1$ for some $r \leq k <t$ and $w_5 \cdot v_l$ for some $1 \leq l < k$, so the subgraph induced by $\{v_1, \ldots, v_k, v_t, w_5\}$ contains an incomplete cycle.

Case IV.3: $A_5 \cap A_1\neq \emptyset$. 

If $A_5 \cap A_1 \neq \emptyset$, then, as in the argument for $w_j$ in Case Iv.1, either $w_5 = -e_5 + d_{t-1}$ or $w_5 = -e_5 + d_{t-1} + d_t$. If $|w_2| = |w_3| = 2$, then $(w_1;v_m, w_2, w_3)$ is a claw, and, for $j \in \{2,3\}$, if $|w_j|\geq 3$ and $w_j \cdot -e_1 = -1$, then, letting $\{k,l\} = \{j, 5\}$, $A_k = \{t-1\}$ and $A_l = \{t-1, t\}$, but then $(v_{m+1}, w_j, w_5)$ is a heavy triple. 

Case V: $m = 0$. 

If $m = 0$, then $\{0, \ldots, t-1\}\subset A_1$, so $|v_i| = 2$ for all $1 \leq i \leq t-1$. If $t \in A_1$, then $2 \leq w_1 \cdot v_t \leq |w_1|-3$, which is absurd. Suppose that $w_1 = -e_1 + d_0 + \ldots + d_{t-1} + d_r$ for some $t < r$. Then $2\leq w_1 \cdot v_t = |w_1|-2$, so $w_1 -v_t$ is reducible. Suppose $w=-e_1 + d_r + d_t -d_0 = x + y$ for some $x, y \in (\tau)^\perp$ and $x\cdot y \geq 0$. Then, without loss of generality, either $x = w$ and $y = 0$, or $|x| = 3$, $|y| = 2$, and $x\cdot y = 0$. We must have $y|_{E_8} = 0$ if $y \in (\tau)^\perp$, so we must have $x = -e_1 + d_k$ for some $k \in \{0, t, r\}$, which is absurd since $e_1 \cdot \tau = t + d_r \cdot \sigma > \max\{d_0 \cdot \sigma, d_t \cdot \sigma, d_r \cdot \sigma\}$. Conclude that $w_1 = -e_1 + d_0 + \ldots + d_{t-1}$ and $w_1 \prec v_t$. It follows that $t = n$; if not, then either $|v_{t+1}| = 2$, in which case $(v_t;v_1, v_{t+1}, w_1)$ is a claw, or $v_{t+1} = -d_{t+1} + d_t + d_{t-1}$, in which case $(v_1, \ldots, v_{t+1}, w_1, v_t, v_1)$ is an incomplete cycle.
\end{proof}

\begin{prop}\label{prop:vttight}
If $t\geq 2$ and $v_t$ is tight, then $\sigma = (1, \ldots, 1, n+1) \in \mathbb Z^{n+1}$, and one of the following holds:
\begin{enumerate}
    \item $s^* = (n, n+1, 1, 0, 0, 0, 0, 0)$,
    \item $s^* = (n, 1, n+1, 0, 0, 0, 0, 0)$,
    \item $s^* = (n, n+2, 1, 0, 0, 0, 0, 0)$,
    \item $s^* = (n, 1, n+2, 0, 0, 0, 0, 0)$,
    \item $s^* = (n, n+1, 0, 0, 1, 0, 0, 0)$,
    \item $s^* = (n, 1, 0, 0, n+1, 0, 0, 0)$, 
    \item $s^* = (n, n+2, 0, 0, 1, 0, 0, 0)$,
    \item $s^* = (n, 1, 0, 0, n+2, 0, 0, 0)$,
    \item $s^* = (n, 0, n+1, 0, 1, 0, 0, 0)$,
    \item $s^* = (n, 0, 1, 0, n+1, 0, 0, 0)$,
    \item $s^* = (n, 0, n+2, 0, 1, 0, 0, 0)$, or
    \item $s^* = (n, 0, 1, 0, n+2, 0, 0, 0)$.
\end{enumerate}
\end{prop}

\begin{proof}
First note that for $w_j \in \mathcal W$ $j \neq 1$, and $w_j$ unloaded, if $A_j \neq \emptyset$, then $A_j \in \{\{t\}, \{t, t-1\}, \{t-1\}\}$. Recall that if $A_j 
= t$, then $w_j \dag v_t$ for all $j \in \{1, \ldots, 8\}$. Note that $|w_6| = |w_7| = |w_8| = 2$ or else either $(v_t;v_1, w_1, w_j)$ is a claw or $(v_1, \ldots, v_{t-1}, w_j, w_1, v_t, v_1)$ is an incomplete cycle for some $j \in \{6,7,8\}$

Case I: $|w_5| = 2$.

Suppose that $|w_5| = 2$. It follows that $w_2$ and $w_3$ are both unloaded, and furthermore that $w_j = -e_j + d_{t-1}$ and $w_k = -e_k + d_{t-1} + d_t$ or $w_k = -e_k + d_t$ for $\{j,k\} = \{2,3\}$. If $w_4$ is loaded, then $w_4|_{E_8} = -e_4 + e_2 + e_1$ or else $s^*_4 \leq |\sigma|_1$, and therefore $w_4 = -e_4 + e_2 + e_1 + d_t$ or else $w_4 \cdot w_1 \geq 0$ and either $(w_5;w_1, w_4, w_6)$ is a claw or $(w_1, w_4, w_5)$ is a negative triangle, but then $(v_t, w_1, w_5, w_4, v_t)$ is an incomplete cycle. Conclude that $w_4$ is unloaded. If $t\in A_4$, then $(v_1, \ldots, v_{t-1}, w_4, w_1, v_t, v_1)$ is an incomplete cycle. If $w_4 = -e_4 + d_t$, then $(v_t;v_1, w_1, w_4)$ is a claw. Conclude that $|w_4| = 2$, thus $s^* = (n, n+1, 1, 0,0, 0, 0, 0)$, $s^* = (n, 1, n+1, 0, 0, 0,0, 0)$, $s^* = (n, n+2, 1,0, 0, 0, 0, 0)$, or $s^* = (n, 1, n+2, 0, 0, 0, 0, 0)$. 

Case II: $w_5 = -e_5 + d_{t-1}$.

Suppose now that $w_5 = -e_5 + d_{t-1}$. If $w_j$ is loaded and $w_j \cdot e_1 = 0$ for $j \in \{2,3\}$, then $A_j = \{t\}$ or else $(v_1, \ldots, v_{t-1}, w_j, w_1, v_t, v_1)$ is an incomplete cycle. But then $(v_t;v_1, w_1, w_j)$ is a claw. If $w_2$ is loaded but $|w_3| = 2$, then $w_2 \sim w_3$ and $w_3 \sim w_1$, so either $w_2 = -e_2 + e_4 + d_{t-1} + d_t$ and $(v_1, \ldots, v_{t-1}, w_5, w_2, w_3, w_1, v_t, v_1)$ is an incomplete cycle, or $w_2 = -e_2 + e_4 + d_t$, and $(v_t, w_1, w_3, w_2, v_t)$ is an incomplete cycle. It follows that $w_j = -e_j + d_t$ or $w_j = -e_j + d_{t-1} + d_t$ and $|w_k| = 2$ for $\{j, k\} = \{2,3\}$. Since $|\sigma|_1 + 1 = 2n + 2$, $s^*_1 = n$, $s^*_2 \leq n + 2$, and $s^*_5 = 1$, we must have $w_4|_{E_8} = -e_4 + e_2 + e_1 + e_5$, in which case $w_4 = -e_4 + e_2 + e_1 + e_5 + d_t$, or else $(w_6;w_4, w_5, w_7)$ is a claw. But then $(v_1, \ldots, v_{t-1}, w_5, w_4, v_t, v_1)$ is an incomplete cycle. Conclude that $w_4$ is unloaded. If $t-1 \in A_4$, then $(v_1, \ldots, v_{t-1}, w_5, w_4, w_1, v_t, v_1)$ is an incomplete cycle, and if $w_4 = -e_4 + d_t$, then $(v_t;v_1, w_1, w_4)$ is a claw. Conclude that $|w_4| =2$ so $s^* = (n,0, n+1, 0, 1, 0, 0, 0)$, $s^* = (n, n+1, 0, 0, 1, 0, 0, 0)$, $s^* = (n, 0, n+2, 0, 1, 0, 0, 0)$, or $s^* = (n, n+2, 0, 0, 1, 0, 0, 0)$. 

Case III: $w_5 = -e_5 + d_t$.

Suppose now that $w_5 = -e_5 + d_t$. Citing the same argument as from the beginning of Case II, we see that $w_2$ and $w_3$ are both unloaded. It follows that $w_j = -e_j + d_{t-1}$ and $|w_k| = 2$ for $\{j,k\} = \{2,3\}$. If $w_4$ is loaded, then either $w_2 = -e_2 + d_{t-1}$ and $w_4|_{E_8} = -e_4 + e_2 + e_1 + e_5$ or $|w_2| = 2$ and $w_4|_{E_8} = -e_4 + e_1 + e_5$, and in either case we see that $w_4 \sim w_6$, so $w_4 \dag w_5$ and $w_4 \cdot w_5 = -1$ or else $(w_6;w_4,w_5, w_7)$ is a claw. But then $A_4 = \{t-1\}$, so $(v_1, \ldots, v_{t-1}, w_4, v_t, v_1)$ is an incomplete cycle or $(v_1, v_t, w_4)$ is a negative triangle. Conclude that $w_4$ is unloaded. If $t-1\in A_4$, then $(v_1, \ldots, v_{t-1}, w_4, w_1, v_t, v_1)$ is an incomplete cycle, and if $w_4= -e_4 + d_t$, then $(v_t,w_4, w_5)$ induces a heavy triple. Conclude that $|w_4| = 2$, so $s^* = (n, 0, 1, 0, n+1, 0, 0, 0)$ or $s^* = (n, 1, 0, 0, n+1, 0, 0, 0)$.

Case IV: $w_5 = -e_5 + d_{t-1} + d_t$.

Suppose lastly that $w_5 = -e_5 + d_{t-1} + d_t$. Citing the same argument as from the beginning of Case III, it follows that $w_j = -e_j + d_{t-1}$ and $|w_k| = 2$ for $\{j,k\} = \{2,3\}$. If $w_4$ is loaded then either $w_4|_{E_8} = -e_4 + e_2 + e_1 + e_5$ or $w_4|_{E_8} = -e_4 + e_1 + e_5$, and in both cases we see that $w_4 \sim w_6$, so $w_4 \dag w_5$ and $w_4 \cdot w_5 = -1$, but then $A_4 = \emptyset$, which is absurd since then $s^*_4 \leq n+2 = |\sigma|_1 + 1$. Conclude that $w_4$ is unloaded. If $w_4$ is unloaded and $t-1\in A_4$, then $(v_1, \ldots, v_{t-1}, w_4, w_1, v_t, v_1)$ is an incomplete cycle. If $w_4= -e_4 + d_t$, then $(v_t;v_1, w_1, w_4)$ is a claw. Conclude that $|w_4| = 2$, so $s^* = (n, 0, 1, 0, n+2, 0, 0, 0)$ or $s^* = (n, 1, 0, 0, n+2, 0, 0, 0)$.
\end{proof}

\subsection{When $v_1$ is tight.}

\begin{lem}
    If $v_1$ is tight, then $v_2 \sim v_1$.
\end{lem}

\begin{proof}
    Either $|v_2| = 2$, so $v_2 \dag v_1$ or $v_2 \prec v_1$ since $v_2 \cdot v_1 = -1$, or $v_2 = -d_2 + d_1 + d_0$, in which case $|v_2|= 3$, $v_2 \cdot v_1 = 1$, and so either $v_2 \dag v_1$ or $v_2 \pitchfork v_1$. Suppose, by way of contradiction, that $v_2 \pitchfork v_1$. Then $T(v_1) \cup T(v_2) \setminus T(v_1) \cap T(v_2)$ consists of two intervals $T_0$ and $T_1$, and without loss of generality, either $|T_0| = 4$ and $|T_1| = 2$ or $|T_0| =|T_1| =3$. Suppose first that $|T_0|=4$ and $|T_1| =2$. Since $\varepsilon_1 = \varepsilon_2$, it follows that there is some $x \in (\tau)^\perp \cap \mathbb Z^{n+1}$ with $|x| = 2$ such that $x \cdot v_1 = 1$ and $x \cdot v_2 = -1$. Write $x = \sum_{i = 0}^n x_i d_i$, and we must have $2x_0 - x_1 = 1$ and $x_0 + x_1 -x_2 = -1$, and therefore $x_0 = x_2 = 0$ and $x_1 = -1$, or else $|x|\geq 2$. But then $x_i = 0$ for all but one $2\leq i \leq n$, so $x\cdot \sigma\neq 0$, which is absurd. Suppose instead that $|T_0|=3$ and $|T_1| = 3$. Then there is some $x \in (\tau)^\perp$ with $|x| = 3$ such that $x \cdot v_1 = 2$ and $x \cdot v_2 = -1$. Then either $x = r + d_i$ for some $r \in E_8 \oplus (0)$, in which case $x = r + d_0$ or else $x \cdot v_1 \neq 2$, and so $x \cdot v_2 = 1$, or $x = \sum_{i = 0}^n x_id_i$, in which case $2x_0 -x_1 = 2$, so $x_0 = 1$ and $x_1 = 0$, and $1 -x_2 = -1$, so $|x|\geq 3$. Conclude that $v_2 \dag v_1$ if $v_2 = -d_2 + d_1 + d_0$.
\end{proof}

\begin{lem}
    If $v_1$ is tight and $v_3 = -d_3 + d_2 + d_1$, then $v_3 \dag v_1$.
\end{lem}

\begin{proof}
    Suppose to the contrary. Then $T(v_1) \cup T(v_3) \setminus T(v_1) \cap T(v_3)$ consists of two intervals $T_0$ and $T_1$, and without loss of generality, either $|T_0| = 4$ and $|T_1| = 2$ or $|T_0| =|T_1| =3$. Suppose there is some $x \in (\tau)^\perp$ with $|x| = 2$, $x\cdot v_1 = 1$, and $x \cdot v_3 = 1$. Then, writing $x = \sum_{i=0}^nx_id_i$, $2x_0 -x_1 = 1$ and $x_1 + x_2 -x_3 = 1$, so $x_0 = x_1 = 1$, but then $x \cdot \tau \neq 0$. Suppose there is some $x \in (\tau)^\perp$ with $|x| = 3$, $x\cdot v_1 = 2$, and $x \cdot v_3 = 1$. Then, $2x_0 -x_1 = 2$, so $x_0 = 1$ and $x_1 = 0$, and $x_2 -x_3 = 1$, so either $x_2 = 0$ and $x_3 = -1$, in which case $x \cdot \tau = \sigma_0 - \sigma_3 \pm \sigma_j\neq 0$ since $|\sigma_j| > |\sigma_3-\sigma_1|$ for $j \geq 4$, or $x_2 = 1$ and $x_3 = 0$, in which case $x\cdot \tau = \sigma_0 +\sigma_2 + \sigma_j \neq 0$ since $|\sigma_j|\geq \sigma_0 + \sigma_2$ for $j \geq 4$. Conclude that $v_3 \dag v_1$
\end{proof}

\begin{cor}
    If $v_1$ is tight, then for all $2 \leq j \leq n$ there is a path from $v_j$ to $v_1$ in $G(\mathcal V)$.
\end{cor}
\qed

\begin{lem}\label{lem:v1tightw1}
If $v_1$ is tight and $n \geq 2$, then $w_1 = -e_1 + d_1 + \ldots + d_n$ and $|v_i| = 2$ for all $2\leq i \leq n$. 
\end{lem}

\begin{proof}

Let $m = \min(A_1)$. We break our analysis into cases based on $m$.

Case I: $m = 0$.

If $m = 0$, then either $w_1 = -e_1 + d_0$ and $w_1 \prec v_1$, $w_1 = -e_1 + d_0 + d_g$ for some $2\leq g$ and $w_1 \pitchfork v_1$, or $1 \in A_1$ and $w_1 \dag v_1$.

Case I.1: $w_1 = -e_1 + d_0$.

Suppose $w_1 = -e_1 + d_0$. It follows that $v_i$ is just right for all $2 \leq i \leq n$, or else $v_g = -d_g + d_{g-1} + d_0$ for some $2<g$, in which case there is a sign error between $v_g$ and $w_1$ mediated by $v_1$. If $j \in \{2,3,5\}$, $|w_j|\geq 3$, and $w_j\cdot -e_1 = -1$, then either $0 \in A_j$ or $1 \in A_j$, or else $(v_1, \ldots, v_{m'}, w_j, w_1, v_1)$ is an incomplete cycle for $m' = \min(A_j)$. 

Suppose now that $0\in A_j$. It follows that $w_j \pitchfork v_1$ or else either $|v_2| = 2$ and $(v_1; v_2, w_1, w_j)$ is a claw, or $v_2 = -d_2 + d_1 + d_0$ and $(v_1, w_1, v_2, w_j, v_1)$ is a claw. Conclude that $w_j = -e_j + d_0 + d_g$ for some $2 \leq g$. 

Suppose instead that $0 \in A_j$ but $1 \in A_j$. It follows that $w_j\dag v_1$ or else there is a sign error between $w_1$ and $w_j$ mediated by $v_1$. If $v_2 = -d_2 + d_1 + d_0$, then $2 \in A_j$ or else $(v_1,v_2, w_j)$ is a negative triangle, but then $(v_1, v_2, w_1, w_j, v_1)$ is an incomplete cycle. Conclude that $|v_i| = 2$ for all $2 \leq i \leq n$ and $A_j = \{1,2, \ldots, n\}$. 

We further break this case down depending on whether $|w_5| = 2$, $w_5 = -e_5 + d_0 + d_g$ for some $2 \leq g$, or $|v_i|=2$ for all $2 \leq i \leq n$ and $w_5 = -e_5 + d_1 + \ldots + d_n$.

Case I.1.a: $|w_5| = 2$.

If $|w_5| = 2$, then $|w_6| = |w_7| = |w_8| = 2$, or else the subgraph induced by $\{v_1, \ldots, v_n, w_1, w_5,$ $w_6, w_7, w_8\}$ contains an incomplete cycle. It follows that $w_2$ and $w_3$ are both unloaded and $|w_2|$, $|w_3|\geq 3$, or else there is a claw at $w_1$. Let $\{j,k\} = \{2,3\}$. Then, without loss of generality, $w_j = -e_j + d_0 + d_g$ for some $2\leq g$ and $w_k = -e_k + d_1+ \ldots + d_n$ and $|v_i| = 2$ for all $2\leq i \leq n$, or else either $2\leq w_j \cdot w_k \leq |w_j|-2$ or $w_j \cdot w_k = 1$ and $w_j, w_k \pitchfork v_1$, in which case there is a sign error between $w_j$ and $w_k$ mediated by $v_1$. However, then $w_j = -e_j + d_0 + d_n$ and there is a sign error between $w_j$ and $v_2 + \ldots + v_n$ mediated by $v_1$. Conclude that $|w_5|\geq 3$.

Case I.1.b: $w_5 = -e_5 + d_0 + d_g$.

Suppose that $w_5 = -e_5 + d_0 + d_g$. We must then have $|w_j|\geq 3$ for some $j \in \{2,3\}$ with $w_j \cdot -e_1 = -1$, and so by the same argument as in Case I.1.a, we must have $|v_i| = 2$ for all $2\leq i \leq n$ and $A_j = \{1, \ldots, n\}$, but then $w_5 = -e_5 + d_0 + d_n$ and there is a sign error between $w_5$ and $v_2 + \ldots + v_n$ mediated by $v_1$. 

Case I.1.c: $|v_i| = 2$ for all $2 \leq i \leq n$ and $w_5 = -e_5 + d_1 + \ldots + d_n$.

If $|v_i| = 2$ for all $2 \leq i \leq n$ and $w_5 = -e_5 + d_1 + \ldots + d_n$, then there is some $j \in \{2,3\}$ with $|w_j|\geq 3$ and $w_j \cdot -e_1 = -1$, so $w_j = -e_j + d_0 + d_n$ by the arguments outlined in Cases I.1.a and I.1.b, but then $(v_1, \ldots, v_n, w_j, w_5, v_1)$ is an incomplete cycle.

Case I.2: $w_1 = -e_1 + d_0 + d_g$ for some $2\leq g$.

Suppose that $w_1 = -e_1 + d_0 + d_g$ for some $2 \leq g$. Then either $g = n$, or $|v_{g+1}|\geq 3$, which is absurd since then $w_1 \cdot v_{g+1} = 1$, and so $v_{g+1}$ separates $w_1$ from $v_1$ in $G(\mathcal V)$ but $v_{g+1} \not \prec v_1$. Conclude that $g= n$, and furthermore that $0 \in \text{supp}(v_n)$, and moreover that $|v_n|\geq 3$, or else $w_1 \dag v_n$ and so $v_n$ separates $w_1$ from $v_1$ in $G(\mathcal V)$ but $v_n \not \prec v_1$. Then, since $|w_1| = 4$, $|v_1| = 5$, and $w_1 \cdot v_1 = 2$, there exists some $x \in (\tau)^\perp$ with $|x| = 2$ and $x\cdot w_1 = x\cdot v_1 = -1$. Since $-d_0 + d_1 \not \in (\tau)^\perp$, we must have $x = d_1 -d_n$, in which case $|v_i|=2$ for all $2\leq i \leq n$, a contradiction.

Case I.3: $1 \in A_1$ and $w_1 \dag v_1$.

If $1 \in A_1$, then $w_1 = -e_1 + d_0 + \ldots + d_n$, and therefore $|v_i| = 2$ for all $2\leq i \leq n$ or else either $v_2 = -d_2 + d_1 + d_0$, which is absurd since then $v_2 \dag v_1$ and $v_2 \dag w_1$, or $v_3 = -d_3 + d_2 + d_1$, which is absurd since then $(v_1, v_3, w_1)$ is a negative triangle. It follows that if $|w_j|\geq 3$ and $w_j \cdot -e_1 = -1$, then $A_j = \{n\}$, or else either $w_j = -e_j + d_0$, in which case $(v_1;v_2, w_1, w_j)$ is a claw, or $w_j = -e_j + d_0 + d_n$, in which case there is a sign error between $w_j$ and $v_1 + \ldots + v_n$ mediated by $v_1$, or $A_j = \{n-1, n\}$, in which case either $n = 2$ and $(v_1, w_1, w_j)$ is a negative triangle, or $n >2$ and $(v_{n-1};v_{n-2}, v_n, w_j)$ is a claw.

Case I.3.a: $|w_5| = 2$.

If $|w_5| = 2$, then $|w_6| = |w_7| = |w_8| = 2$ or else there is an incomplete cycle, so $w_2$ and $w_3$ are unloaded and $|w_2|$, $|w_3|\geq 3$ or else there is a claw at $w_1$. We must then have that $w_2 = -e_2 + d_n$ and $w_3 = -e_3 + d_n$, but then $(v_1,w_2,w_3)$ induces a heavy triple.

Case I.3.b: $w_5 = -e_5 + d_n$.

If $w_5 = -e_5 + d_n$, then we must have $|w_j|\geq 3$ and $w_j \cdot -e_1 = -1$ for some $j \in \{2,3\}$ or else $|w_2|=|w_3| = 2$ and $(w_1;v_1, w_2, w_3)$ is a claw. But then $A_j = \{n\}$ and $(v_1, w_j, w_5)$ induces a heavy triple.

Case II: $m = 1$.

If $m =1$, then by work above we see that $w_1\dag v_1$. If $|v_2| = 2$, then $2 \in A_1$, in which case $w_1 = -e_1 + d_1 + \ldots + d_n$, in which case either $|v_i| = 2$ for all $2\leq i \leq n$, or $v_3 = -d_3 + d_2 + d_1$ or $v_3 = -d_3 + d_2 + d_1 + d_0$, but in either case $v_3 \dag v_1$ and $v_3 \dag w_1$, which is absurd. Suppose now that $w_2 = -d_2 + d_1 + d_0$. We must have $2 \in A_1$ or else $(v_1, v_2, w_1)$ is a negative triangle, and so either $|v_i| = 2$ for all $3\leq i \leq n$ and $A_1 = \{1, \ldots, n\}$, or $v_3 = -d_3 + d_2 + d_0$, in which case either $3 \in A_1$ and $(v_1;v_2, v_3, w_1)$ is a claw or $3 \not \in A_1$ and $(v_1, v_3, w_1)$ is a negative triangle. Conclude that $A_1 = \{1, \ldots, n\}$, $|v_i| = 2$ for all $3 \leq i \leq n$, and $|v_2|\in \{2,3\}$. 

Suppose that $v_2 = -d_2 + d_1 + d_0$ and let $|w_j|\geq 3$ with $w_j \cdot -e_1 = -1$. If $A_j \cap A_1 = \emptyset$, then $w_j = -e_j + d_0$ and $w_j \prec v_1$, but then $(v_1, v_2, w_j, w_1, v_1)$ is an incomplete cycle. If $A_j \cap A_1 = \{1\}$, then $(v_1, v_2, w_j)$ is a negative triangle. If $n \geq 3$ and $A_j \cap A_1 = \{1, n\}$, then $(v_1, w_j, v_n, \ldots, v_2, v_1)$ is an incomplete cycle. If $A_j \cap A_1 = \{n, n-1\}$, then either $n = 2$ and $(v_1, w_1, w_j)$ induces a heavy triple, or $n > 2$ and there is a claw at $v_{n-1}$. Conclude that $A_j = \{n\}$. We must have $|w_j|$, $|w_k|\geq 3$ with $w_j \cdot -e_1 = w_k \cdot -e_1 = -1$ for $j \neq k \in \{2,3,5\}$ or else there is a claw at $w_1$ or $|w_5|=2$, $w_2$ or $w_3$ is loaded, and the subgraph induced by $\{v_1, \ldots, v_n, w_1, w_5, w_6, w_7, w_8\}$ contains an incomplete cycle, but then $(v_2, w_j, w_k)$ is a heavy triple. Conclude that $v_2 = -d_2 + d_1$.

Case III: $2\leq m$.

If $2 \leq m$, then either $m  = n$ or $|v_{m+1}|\geq 3$ and $m+1 \in A_1$, or else there is an incomplete cycle or $|v_{m+1}| = 2$ and there is a claw at $v_m$. Suppose that $|v_{m+2}|\geq 3$. Then, either $v_{m+2} = -d_{m+2} + d_{m+1} + d_m$ and $(v_l,v_{m+2}, w_1)$ induces a heavy triple for $l = \max\{i\leq m \colon |v_i|\geq 3\}$, or $v_{m+2} = -d_{m+2}+ d_{m+1} + d_0$, in which case $v_{m+1}= -d_{m+1} + d_m + \ldots + d_1$ or else either $v_{m+1} = -d_{m+1} + d_m + \ldots + d_0$ and $(v_1;v_2,v_{m+1}, v_{m+2})$ is a claw if $|v_2| =2$ and $(v_1, v_{m+1}, v_2, v_{m+2}, v_1)$ is an incomplete cycle if $v_2 = -d_2 + d_1 + d_0$, or $|v_{m+1}|\leq m$ and $(v_1, \ldots, v_{m+1}, v_{m+2}, v_1)$ is an incomplete cycle, and therefore $|v_i| = 2$ for all $2 \leq i \leq m$. It follows that $m+2 \in A_1$ or else $(v_1, \ldots, v_m, w_1, v_{m+2}, v_1)$ is an incomplete cycle; thus, $w_1 = -e_1 + d_m + \ldots + d_n$ and $|v_i| = 2$ for all $m+3\leq i \leq n$. Note that $\epsilon(w_1) = \varepsilon(v_1)$ since $w_1 \cdot v_1 = 0$ and $w_1 \cdot (v_2 + \ldots + v_m) = v_1 \cdot (v_2 + \ldots + v_m) = -1$. Suppose instead that $|v_{m+2}| = 2$. Then $|v_i| = 2$ for all $m+2\leq i \leq n$ and $w_1 = -e_1 + d_m + \ldots + d_n$.

We now break this case into subcases based on whether $m = n$, $v_{m+2} = -d_{m+2} + d_{m+1} + d_0$, or $|v_i| = 2$ for all $m+2\leq i \leq n$. 

Case III.1: $m = n$. 

Suppose that $m = n$, and suppose that $|w_j|\geq 3$ for some $w_j \cdot -e_1 = -1$. Then, $n \in A_j$ or else the subgraph induced by $\{v_1, \ldots, v_n, w_1, w_j\}$ contains an incomplete cycle. We must then either have $w_j \cdot v_n = 0$, or else $w_j \sim v_k$ for some $k \leq n-1$ with $v_n \sim v_k$, or else $(v_n;v_k, w_1, w_j)$ is a claw.

Case III.1.a: $|w_5| = 2$.

If $|w_5| = 2$, then either $|w_6| = |w_7| = |w_8| = 2$ or $v_n = -d_n + d_{n-1} + \ldots + d_0$ and $w_6 = -e_6 + d_n + d_0$, or else the subgraph induced by $\{v_1, \ldots, v_n, w_1, w_5, w_6, w_7, w_8\}$ contains an incomplete cycle. Furthermore, we must have $|w_2|$, $|w_3|\geq 3$ or else there is a claw at $w_1$. It follows that there is some $j \in \{2,3\}$ with $|w_j|\geq 3$ and $w_j \cdot -e_1 = -1$, so $n \in A_j$, and therefore $|w_6| = |w_7| = |w_8| = 2$ or else either $0 \in A_j$ and $w_j \cdot w_6 = 2$, which is absurd, or $w_j \dag w_6$, in which case the subgraph induced by $\{v_1, \ldots, v_n, w_1, w_j, w_5, w_6\}$ contains an incomplete cycle. It follows that $w_2$ and $w_3$ are both unloaded and $w_2 \dag w_3$. We must therefore have $|v_n|\geq 3$ or else $v_n = -d_n + d_{n-1}$ and then either $w_2 \cdot w_3 = 2$, which is absurd, or $(v_n;v_{n-1}, w_1, w_j)$ is a claw for some $j \in \{2,3\}$. Let $\{j,k\} = \{2,3\}$. If $v_n = -d_n + d_{n-1} + \ldots + d_l$ for some $l \leq n-2$, then either $w_2 \cdot w_3 = 2$, or, without loss of generality, $w_j = -e_j + d_l + d_n$ and $w_k = -e_k + d_{n-1} + d_n$, but then either $n \geq 3$ and $(v_1, \ldots, v_l, w_j, w_k, v_{n-1}, \ldots v_1)$ is an incomplete cycle or $n = 2$, $w_j \cdot v_1 = 2$ and $w_k \cdot v_1 = -1$, so there is a sign error between $w_j$ and $w_k$ mediated by $v_1$.

Case III.1.b: $n \in A_5$. 

If $n \in A_5$, then $|w_j|\geq 3$ for some $j \in \{2,3\}$ with $w_j \cdot -e_1 = -1$ or else $|w_2| = |w_3| =2$ and $(w_1;v_n,w_2,w_3)$ is a claw. Then $w_j \dag w_5$, and the argument from Case III.1.a produces an incomplete cycle or a sign error between $w_j$ and $w_5$.

Case III.2: $m+1 = n$.

Suppose now that $m+1 = n$, and that $|w_j|\geq 3$ and $w_j \cdot -e_1 = -1$. If $A_j \cap A_1 =\emptyset$, then $w_j \dag v_1$ and $w_j \sim v_i$ for some $i \leq m$, so the subgraph induced by $\{v_1, \ldots, v_n, w_1, w_j\}$ contains an incomplete cycleIf $m \in A_j$, then $A_j \cap A_1 = \{m\}$ or else $A_j \cap A_1 = \{m, m+1\}$, and then $w_j \dag w_1$ and $w_j \dag v_l$ for some $l \leq m$, in which case there is either a heavy triple or an incomplete cycle. It follows that if $m \in A_j$, then either $w_j \cdot v_{m+1} = 1$, hence $w_j \dag v_{m+1}$, so $w_j \cdot v_m = 0$ or else there is a claw at $v_m$ or the subgraph induced by $\{v_1, \ldots, v_n, w_j\}$ contains an incomplete cycle or a heavy triple. But then, either $w_j \cdot v_1 = v_{m+1} \cdot v_1 = 0$, in which case $(v_1, v_{m+1}, w_j)$ induces a heavy triple; or $|\{x \in \{v_{m+1}, w_j\} \colon x\cdot v_1 \neq 0\}| = 1$, in which case either $w_j \pitchfork v_1$ and $v_{m+1}\cdot v_1 = 0$, so $w_j = -e_j + d_m + d_0$, $v_m = -d_m + d_{m-1} + \ldots + d_0$, $|v_i| = 2$ for all $2\leq i \leq m-1$, and $v_{m+1} = -d_{m+1} + d_m + d_{m-1}$, so there is a sign error between $w_j$ and $v_2 + \cdots + v_{m-1} + v_{m+1}$ mediated by $v_1$; $w_j \dag v_1$ and $v_{m+1} \cdot v_1 = 0$, in which case $(v_1, \ldots, v_{m+1}, w_j, v_1)$ is an incomplete cycle; or $w_j \cdot v_1 = 0$ and $v_{m+1} \sim v_1$, in which case $w_j \sim v_i$ for some $2\leq i \leq m-1$ and $(v_1, \ldots, v_i, w_j, v_{m+1}, v_1)$ is an incomplete cycle; or $w_j \pitchfork v_1$ and $v_{m+1} \dag v_1$, so there is a sign error between $w_j$ and $v_{m+1}$; or $w_j \dag v_1$ and $v_{m+1} \prec v_1$, in which case $m = 2$, $v_2 = -d_2 + d_1$, $v_3 = -d_3 + d_2 + d_0$, and $w_j = -e_j + d_2 + d_1$, but then $(v_1, v_2, v_3, w_j, v_1)$ is an incomplete cycle. Conclude that $A_j \cap A_1 = \{n\}$. 

Case III.2.a: $|w_5| = 2$.

If $|w_5| = 2$, then $|w_6| = |w_7| = |w_8| = 2$, or else for $l = \min \{j \in \{6,7,8\} \colon |w_j|\geq 3\}$ either $w_j \dag v_i$ for some $i \neq m$, in which case there is an incomplete cycle, or $\{m, m+1\} \subset A_l$, so $w_l \cdot v_1 = 2$, which is absurd. It follows that $w_2$ and $w_3$ are unloaded and $|w_2|$, $|w_3|\geq 3$ or else there is a claw at $w_1$. Then, $n \in A_2 \cap A_3$, so $w_2 \cdot w_3 = 1$ and $w_2 \dag w_3$, and therefore $v_{m+1}$ is just right or else $v_{m+1} = -d_{m+1} + d_m + d_0$ and either $0 \in A_2 \cap A_3$, in which case $2\leq w_2 \cdot w_3 \leq |w_2|-2$, or $w_j \dag v_{m+1}$ but either $w_j \cdot v_1 = 0$ or $w_j \dag v_1$ and $w_k \pitchfork v_1$ for $\{j,k\} = \{2,3\}$, which is absurd, or $w_2\dag v_{m+1}$ and $w_3 \dag v_{m+1}$ and $(w_2, w_3, v_{m+1})$ is a heavy triple. However, if $v_{m+1} = -d_{m+1} + d_m + \ldots + d_l$ for some $l \leq m-1$, then either $(v_{m+1}, w_2, w_3)$ is a heavy triple, or either $w_j \dag v_{m+1}$ and $w_k \cdot v_{m+1} = 0$, so $w_k \cdot v_o = 1$, or $w_j \cdot v_l = w_k \cdot v_o = 1$ for $\{j,k\} = \{2,3\}$ and some $l+1\leq l\leq  o \leq m$ with $|v_l|$,$|v_o|\geq 3$, so the subgraph induced by $\{v_1, \ldots, v_n, w_2, w_3\}$ contains either a heavy triple or an incomplete cycle.

Case III.2.b: $n \in A_5$.

If $n \in A_5$, then $|w_j|\geq 3$ and $w_j \cdot -e_1 = -1$ for some $j \in \{2,3\}$ or else $(w_1;v_m, w_2,w_3)$ is a claw, but then the argument in Case III.2.a produces a contradiction if $n \in A_j \cap A_5$.

Case III.3: $m+2 \leq n$.

If $m + 2 \leq n$, then either $v_{m+2} = -d_{m+2} + d_{m+1}+ d_0$, $v_{m+2} = -d_{m+2} + d_{m+1}$, or $v_{m+2} = -d_{m+2} + d_{m+1} + d_m$, in which case $(v_l, v_{m+2}, w_1)$ induces a heavy triple for $l = \max\{i\leq m\colon |v_i|\geq 3\}$.

Case III.3.a: $v_{m+2} = -d_{m+2} + d_{m+1} + d_0$.

Suppose, by way of contradiction, that $v_{m+2} = -d_{m+2} + d_{m+1} + d_0$. Then, either $v_{m+1} = -d_{m+1} + d_m + \ldots + d_1$ and $|v_i|=2$ for all $2\leq i \leq m$, or $m+2 = 4$, $v_3 = -d_3 + d_2 + d_0$, and either $v_2 = -d_2 + d_1$ or $v_2 = -d_2 + d_1 + d_0$. 

Case III.3.a.i: $v_{m+1} = -d_{m+1} + d_m + \ldots + d_1$.

Suppose first that $v_{m+1} = -d_{m+1} + d_m + \ldots + d_1$ and $|v_i| =2$ for all $2\leq i \leq m$. Let $|w_j|\geq 3$ with $w_j \cdot -e_1 = -1$. If $A_j \cap A_1 = \emptyset$, then $w_j = -e_j + d_0$, but then $(v_1, w_j, v_{m+2}, v_{m+1}, v_1)$ is an incomplete cycle. If $A_j \cap A_1 = \{m\}$, then $A_j \cap \{1, \ldots, m\} = \{m\}$ or else $2\leq w_j \cdot v_{m+1}$, but then either $A_j = \{m\}$ and $(v_1, \ldots, v_m, w_j, v_{m+1}, v_1)$ is an incomplete cycle or $w_j = -e_j + d_0 + d_m$, but then $w_j \pitchfork v_1$, $w_j \dag v_{m+1}$, and $v_{m+1} \dag v_1$, which is absurd. If $m+1 \in A_j \cap A_1$, then $m+1 = \max(A_j)$ or else $1 \leq w_j \cdot w_1$ with equality if and only $A_j \cap A_1 = \{m+1, n\}$, but then $(v_1, \ldots, v_m, w_1, w_j, v_{m+1}, v_1)$ is an incomplete cycle since $m \not \in A_j \cap A_1$. If $A_j \cap A_1 = \{m+1\}$ then $A_j = \{m+1\}$ and $(v_1, v_{m+1}, w_j, v_{m+2}, v_1)$ is an incomplete cycle, and if $A_j \cap A_1 = \{m, m+1\}$, then $(v_1, \ldots, v_m, w_1, w_j, v_{m+2}, v_1)$ is an incomplete cycle. If $\max(A_j) \geq m+2$, then, in fact, either $A_j = \{n\}$, which is absurd since $v_{m+2} \prec v_1$, $v_{m+2} \sim (v_{m+3} + \cdots + v_n)$, $(v_{m+3} + \ldots + v_n)\cdot v_1 = 0$, and $w_j \sim (v_{m+3} + \cdots +  v_n$, but $w_j \cdot v_1 = 0$, or $w_j = -e_j + d_0 + d_n$, in which case either $n\geq m+3$ and there is a sign error between $w_j$ and $v_{m+2}$ mediated by $v_1$, or $n = m+2$ and $w_j \cdot v_{m+2} = 0$. Since we must have $|w_j|\geq 3$ and $w_j \cdot -e_1 = -1$ for at least one $j \in \{2,3,5\}$, we may assume that $n = m+2$. 

Case III.3.a.i.$\alpha$: $|w_5| = 2$.

If $|w_5| = 2$, then either $|w_6| = |w_7| = |w_8| = 2$, or else there is an incomplete cycle or $w_6 = -e_6 + d_0 + d_{m+2}$, in which case there is a sign error between $w_1$ and $w_6$ mediated by $v_1 + \ldots + v_m$. It follows that $w_2$ and $w_3$ are unloaded and $|w_2|$, $|w_3|\geq 3$, or else there is a claw at $w_1$, but then $w_2 \cdot w_3 = 2$, which is absurd.

Case III.3.a.i.$\beta$: $w_5 = -e_5 + d_0 + d_{m+2}$.

If $w_5 = -e_5 + d_0 + d_{m+2}$, then $|w_2| = |w_3| =2$ or else $A_j = \{0, m+2\}$ for some $j \in \{2,3\}$ and $w_j \cdot w_5 = 2$, but then $(w_1;v_m, w_2,w_3)$ is a claw.

Case III.3.a.ii: $m+2 = 4$, $v_3 = -d_3 + d_2 + d_0$, and $v_2 = -d_2 + d_1$.

Suppose that $m+2 = 4$, $v_3 = -d_3 + d_2 + d_0$, and $v_2 = -d_2 + d_1$, and let $|w_j|\geq 3$ with $w_j \cdot -e_1 = -1$. If $0 \in A_j$, then either $3\in A_j$ or $4 \in A_j$, or else $(v_1, v_3, w_j, w_4, v_1)$ is an incomplete cycle. If $\{0,3\} \subset A_j$, then $4 \in A_j$ or else $w_j \cdot v_4 = 2$, but then $w_j \dag w_1$ and $w_j \dag v_4$, so $(v_1, v_2, w_1, w_j, v_4, v_1)$ is an incomplete cycle. If $\{0,4\} \subset A_j$, then $3\not \in A_j$ or else $(v_1, v_2, w_1, w_j, v_4)$ is an incomplete cycle, but then there is a sign error between $w_j$ and $v_3$ mediated by $v_1$. If $1 \in A_j$, then $2 \in A_j$, and so $3 \in A_j$ since $\sigma_1 + \sigma_2 > \sigma_3$, and so $4 \in A_j$ since $\sigma_1 + \sigma_2 + \sigma_3 > \sigma_4$, but then $(v_1;v_3,v_4, w_j)$ is a claw. If $A_j \cap A_1 = \{m\}$, then $A_j = \{m\}$ and $(v_1, v_2, w_j, v_3, v_1)$ is an incomplete cycle. If $m+1 \in A_j$, then, as in Case III.3.a.i, $m+1 = \max(A_j)$, but then $A_j = \{m+1\}$, or else $m+2 \in A_j$, but then $(v_1, v_3, w_j, v_4, v_1)$ is an incomplete cycle. It follows then that either $A_j = \{n\}$, in which case either $n = 4$ and $w_j \dag v_4$ but $w_j \cdot v_1 = 0$, which is absurd, or $n \geq 5$ and $(v_5 + \ldots + v_n + w_j) \dag v_4$ but $(v_5 + \ldots + v_n + w_j) \cdot v_1 = 0$. Conclude that $|w_2| = |w_3| = |w_5| =2$, so $(w_1;w_2,w_3,w_5)$ is a claw.

Case III.3.a.iii: $m + 2 = 4$, $v_3 = -d_3 + d_2 + d_0$, and $v_2 = -d_2 + d_1 + d_0$. 

Suppose that $m+2 = 4$, $v_3 = -d_3 + d_2 + d_0$, and $v_2 = -d_2 + d_1 + d_0$, and let $|w_j|\geq 3$ with $w_j \cdot -e_1 = -1$. If $0 \in A_j$, then either $(v_2, v_4, w_j)$ is a heavy triple; $2 \in A_j$, in which case $3 \in A_j$ or $w_j \cdot v_3 = 2$, in which case $1 \in A_j$ or else $2 = w_j \cdot v_1 < |w_j|-2$, but then there is a sign error between $w_j$ and $v_3$ mediated by $v_1$; or $4 \in A_j$, in which case $1 \in A_j$ or else there is a sign error between $w_j$ and $v_4$ mediated by $v_1$, but then $2 \in A_j$ or else $w_j \cdot v_2 = 2$, and so $3 \in A_j$ or else $w_j \cdot v_3 = 2$, but then there is a sign error between $w_j$ and $v_3$ mediated by $v_1$. If $1 \in A_j$, then $2 \in A_j$ or else $w_j \dag v_1$ or $w_j \pitchfork v_1$, which is absurd since $w_j \dag v_2$ and $v_2 \dag v_1$, then $3\in A_j$ since $\sigma_1 + \sigma_2 > \sigma_3$, then $4 \in A_j$ since $\sigma_1 + \sigma_2 + \sigma_3 > \sigma_4$, but then $2 \leq w_j \cdot w_1 < |w_j|-2$. If $2 \in A_j$, then $w_j \dag v_2$, so $3 \in A_j$ or else $(v_1, v_2, w_j, v_3, v_1)$ is an incomplete cycle, so $4 \in A_j$ since $\sigma_2 + \sigma_3 > \sigma_4$, but then $2 = w_j \cdot w_1$. If $3\in A_j$, then $w_j \dag v_3$ and $w_j \cdot v_1 = 0$, which is absurd. If $4\in A_j$, then $w_j \dag v_4$ and $w_j \cdot v_1 = 0$, which is absurd. If $n > 4$ and $n \in A_j$, then $A_j = \{n\}$, and $(v_5 + \ldots + v_n + w_j) \dag v_4$ but $(v_5 + \ldots + v_n + w_j) \cdot v_1 = 0$, which is absurd. Conclude that $|w_2| = |w_3| = |w_5| = 2$, so $(w_1;w_2, w_3, w_5)$ is a claw.

Case III.3.b: $|v_i| = 2$ for all $m+2\leq i \leq n$. 

Suppose that $|v_i| = 2$ for all $m+2 \leq i \leq n$ and let $|w_j| \geq 3$ with $w_j \cdot -e_1 = -1$. If $A_j \cap A_1 = \emptyset$ then either $w_j = -e_j + d_0$, in which case $(v_1, \ldots, v_m, w_1, w_j, v_1)$ is an incomplete cycle, or there exists some $2\leq i \leq m$ such that $w_j \cdot v_i = 1$, in which case either some subpath of the path $(v_1, \ldots, v_m, w_1, w_j, v_i, \ldots, v_1)$ is an incomplete cycle or $|v_m|\geq 3$ and $w_j \cdot v_m = 1$, in which case $(v_m, w_1, w_j)$ is a heavy triple. Conclude that $A_j \cap A_1 \neq \emptyset$. If $m \in A_j$, then $A_j \cap A_1 = \{m\}$ or else either $A_j \cap A_1 = \{m, n\}$, and since $n > m+1$, $(v_1, \ldots, v_m, w_j, v_n, v_{n-1}, \ldots, v_{m+1}, \ldots, v_1)$ is an incomplete cycle, or $2 \leq w_j \cdot w_1$, which is absurd. It follows then that $w_j \cdot v_{m+1} = 1$. If $v_{m+1} \sim v_1$, then $v_{m+1} \prec v_1$ or else either $w_j \pitchfork v_1$ and $(v_{m+1};v_1, v_{m+2}, w_j)$ is a claw, or $w_j \dag v_1$, $w_j \dag v_{m+1}$, and $v_{m+1} \dag v_1$, which is absurd. If $v_{m+1} \prec v_1$, then $w_j \dag v_1$, so $\{0, 1\} \subset A_j$ or else either $(v_{m+1};v_1, v_{m+2}, w_j)$ is a claw or $(v_1, v_{m+1}, w_j)$ is a negative triangle, but then $m+1 \in A_j$ since $\sigma_0 + \sigma_1 + \sigma_m > \sigma_{m+1}$, which is absurd. Conclude that $v_{m+1} \cdot v_1 =0$, hence $v_{m+1} = -d_{m+1} + d_m + \ldots + d_{l}$ for some $l < m$ and $v_{m+1} \dag v_l$. Then either $l \in A_j$, which is absurd since then $w_j \cdot v_{m+1} = 2$, or $(w_j;v_l,v_{m+1}, w_j)$ is a claw. Conclude that $n \in A_j$, and, in fact, that $A_j \cap A_1 = \{n\}$ or else either $w_j \cdot w_1 = 1$ and $(v_1, \ldots, v_m, w_1, w_j, v_{n-1}, \ldots, v_{m+1}, \ldots, v_1)$ is an incomplete cycle, or $2\leq w_j \cdot w_1$.

Case III.3.b.i: $|w_5| = 2$. 

If $|w_5| = 2$, then $|w_6| = |w_7|= |w_8|= 2$ or else the subgraph induced by $\{v_1, \ldots, v_n, w_1, w_5,$ $w_6, w_7, w_8\}$ contains an incomplete cycle. It follows that $w_2$ and $w_3$ are unloaded, and that $|w_2|$, $|w_3|\geq 3$ or else there is a claw at $w_1$, so $\{n\} = A_1 \cap A_2 \cap A_3$, hence $(v_{m+1}, w_2, w_3)$ is a heavy triple.

Case III.b.ii: $|w_5|\geq 3$.

If $|w_5|\geq 3$, then $\{n\} = A_1 \cap A_5$, and $|w_j|\geq 3$ for some $j \in \{2,3\}$ with $w_j \cdot -e_1 = -1$, so $\{n\} = A_1 \cap A_j \cap A_5$, and $(v_{m+1}, w_j, w_5)$ is a heavy triple.
\end{proof}

\begin{prop}\label{prop:v1tight}
If $v_1$ is tight, then $\sigma = (1, 2, \ldots, 2) \in \mathbb Z^{n+1}$ and one of the following holds:
\begin{enumerate}
    \item $s^* = (2n, 0, 1, 0, 2, 0, 0, 0)$,
    \item $s^* = (2n, 0, 2, 0, 1, 0, 0, 0)$,
    \item $s^* = (2n, 1, 0, 0, 2, 0, 0, 0)$,
    \item $s^* = (2n, 1, 2, 0, 0, 0, 0, 0)$,
    \item $s^* = (2n, 2, 0, 0, 1, 0, 0, 0)$, or
    \item $s^* = (2n, 2, 1, 0, 0, 0, 0, 0)$.
\end{enumerate}
\end{prop}

\begin{proof}
By Lemma \ref{lem:v1tightw1}, if $v_1$ is tight, then $|v_i| = 2$ for all $2\leq i \leq n$ and $w_1 = -e_1 + d_1 + \ldots + d_n$. If $|w_j|\geq 3$, then either $w_j = -e_j + d_0$, $A_j = \{n\}$, or $w_j = -e_j + d_0 + d_n$, in which case there is a sign error between $w_j$ and $v_2 + \ldots + v_n$ mediated by $v_1$. 

Case I: $|w_5| = 2$.

If $|w_5| = 2$, then $|w_6| = |w_7| = |w_8| = 2$ or else the subgraph induced by $\{v_1, \ldots, v_n, w_1, w_5,$ $w_6, w_7, w_8\}$ contains an incomplete cycle. It follows that $w_2$ and $w_3$ are unloaded and $|w_2|$, $|w_3| \geq 3$ or else there is a claw at $w_1$, so for $\{j,k\} = \{2,3\}$, $w_j = -e_j + d_0$ and $w_k = -e_k + d_n$. If $w_4$ is loaded, then $w_4|_{E_8} = -e_4 + e_2 + e_1$ and either $w_2 = -e_2 + d_0$, in which case $n \in A_4$ or else $s^*_4 = |\sigma|_1 + 1$, but then $(w_5;w_1,w_4,w_6)$ is a claw or $(w_1, w_4, w_5)$ is a negative triangle, or $w_2 = -e_2 + d_n$, in which case either $n \in A_4$ and either $(w_5;w_1, w_4, w_6)$ is a claw or $(w_1, w_4, w_5)$ is a negative triangle, or $(v_1, \ldots, v_n, w_2, w_4, w_5, w_1, v_1)$ is an incomplete cycle. Conclude that $w_4$ is unloaded, hence either $|w_4| = 2$, $w_4 = -e_4 + d_n$, in which case $(v_1, \ldots, v_n, w_4, w_1, v_1)$ is an incomplete cycle, or $w_4 = -e_4 + d_0$, in which case either $j = 2$ and there is a sign error between $w_4$ and $w_2$ induced by $v_1$ or $j = 3$ and $(v_1;w_1,w_3, w_4)$ is a claw. Conclude that $|w_4| = 2$, hence $s^* = (2n, 1, 2, 0, 0, 0, 0, 0)$ or $s^* = (2n, 2, 1, 0, 0, 0, 0, 0)$. 

Case II: $w_5 = -e_5 + d_0$.

If $w_5 = -e_5 + d_0$, then $|w_j|\geq 3$ for some $j \in \{2,3\}$ with $w_j \cdot -e_1 = -1$ or else $|w_2| = |w_3| = 2$ and $(w_1; v_1, w_2, w_3)$ is a claw. If $w_j = -e_j + d_0$, then there is a sign error between $w_j$ and $w_5$ mediated by $v_1$, so $A_j = \{n\}$. If $j = 2$ and $w_2$ is loaded, then $|w_3| = 2$ and $w_2 = -e_2 + e_4 + d_n$, so $(v_1, \ldots, v_n, w_2, w_3, w_1, v_1)$ is an incomplete cycle. Conclude that $w_j = -e_j + d_n$. It follows that $|w_6| = |w_7| = |w_8| = 2$ or else, for $l = \min \{k \in \{6,7,8\} \colon |w_k|\geq 3\}$, either $w_l = -e_l + d_n$ and $(v_1, w_j, w_l)$ induces a heavy triple, or $w_l = -e_l + d_0$ and there is a sign error between $w_5$ and $w_l$ mediated by $v_1$ if $l \geq 7$ or $l = 6$ and $(v_1;v_2, w_5, w_6)$ is a claw. Then, letting $\{j,k\} = \{2,3\}$, it follows that $w_k$ is unloaded, hence $|w_k| = 2$, since if $w_k = -e_k + d_0$ then there is a sign error between $w_k$ and $w_5$ mediated by $v_1$, and if $w_k = -e_k +d_n$, then $(v_1, w_j,w_k)$ induces a heavy triple. Suppose that $w_4$ is loaded and $j = 2$; thus, $w_4|_{E_8} = -e_4 + e_2 + e_1 + e_5$, or else $s^*_4 \leq |\sigma|_1 + 1$, and $0 \not \in A_4$ or else $(w_6;w_4,w_5, w_7)$ is a claw, hence either $w_4 = -e_4 + e_2 + e_1 + e_5$ and $(v_1, \ldots, v_n, w_2, w_4, w_5, v_1)$ is an incomplete cycle or $w_4 = -e_4 + e_2 + e_1 + e_5 + d_n$ and $(v_1, \ldots, v_n, w_4, w_1, v_1)$ is an incomplete cycle. Suppose instead that $w_4$ is loaded and $j = 3$; then $w_4|_{E_8} = -e_4 + e_1 + e_5$, so either $0 \in A_4$ and $(w_6;w_4, w_5, w_7)$ is a claw or $w_4 = -e_4 + e_1 + e_5 + d_n$ and $(v_1, \ldots, v_n, w_4, w_2, w_1, v_1)$ is an incomplete cycle. Conclude that $w_4$ is not loaded. If $w_4 = -e_4 + d_0$ then there is a sign error between $w_4$ and $w_5$ mediated by $v_1$, and if $w_4 = -e_4 + d_n$, then $(v_1, \ldots, v_n, w_4, w_1, v_1)$ is an incomplete cycle. Conclude that $|w_4| = 2$, hence $s^* = (2n, 2, 0, 0, 1, 0, 0, 0)$ or $s^* = (2n, 0, 2, 0, 1, 0, 0, 0)$.

Case III: $w_5 = -e_5 + d_n$. 

If $w_5 = -e_5 = d_n$, then $|w_j|\geq 3$ for some $j \in \{2,3\}$ with $w_j \cdot -e_1 = -1$ or else $|w_2| = |w_3| = 2$ and $(w_1; v_1, w_2, w_3)$ is a claw. If $A_j = \{n\}$, then $(v_1, w_j, w_5)$ induces a heavy triple, so $w_j = -e_j + d_0$. It follows that $|w_6| = |w_7| = |w_8| = 2$ or else, for $l = \min \{k \in \{6,7,8\} \colon |w_k|\geq 3\}$, either $w_l = -e_l + d_0$ and there is a sign error between $w_j$ and $w_l$, or $w_l = -e_l = d_n$ and either $l = 6$ and $(v_n;v_{n-1}, w_5, w+l)$ is a claw or $l \geq 7$ and $(v_1, v_5, v_l)$ induces a heavy triple. Letting $\{j,k\} = \{2,3\}$, it follows that $w_k$ is unloaded, so $|w_k| = 2$ since if $w_k = -e_k + d_0$, then there is a sign error between $w_2$ and $w_3$ mediated by $v_1$ and if $w_k = -e_k + d_n$ then $(v_1, w_k, w_5)$ induces a heavy triple. Suppose that $w_4$ is loaded and that $j = 2$. Then $w_4|_{E_8} = -e_4 + e_2 + e_1 + e_5$ or else $s^*_4 \leq |\sigma|_1 + 1$, and so $w_4 = -e_4 + e_2 + e_1 + e_5$ or else $n \in A_4$ and $(w_6;w_4,w_5, w_8)$ is a claw, but then $(v_1, \ldots, v_n, w_5, w_4, w_2, w_1, v_1)$ is an incomplete cycle. Suppose instead that $w_4$ is loaded and $j = 3$; then $w_4|_{E_8} = -e_4 + e_1 + e_5$ and $n \in A_4$ or else $s^*_4 \leq |\sigma|_1 + 1$, but then $(w_6;w_4, w_5, w_7)$ is a claw. Conclude that $w_4$ is unloaded, hence $|w_4| = 2$ or else $w_4 = -e_4 + d_n$ and $(v_1, \ldots, v_n, w_4, v_1)$ is an incomplete cycle, or $w_4 = -e_4 + d_0$ and either $j = 2$ and $v_1$ mediates a sign error between $w_4$ and $w_2$ or $j = 3$ and $(v_1;v_2, w_3, w_4)$ is a claw. Hence, $s^* = (2n, 1, 0, 0, 2, 0, 0, 0)$ or $s^* = (2n, 0, 1, 0, 2, 0, 0, 0)$.
\end{proof}

\subsection{When $G(\mathcal V)$ is disconnected.} 
First we recall some basic observations about changemaker bases whose intersection graphs are disconnected. 

\begin{lem}[Lemma 5.1 of \cite{Gre13}]
A changemaker lattice has at most two indecomposable summands. If it has two indecomposable summands, then there exists an index $r>1$ for which $v_r = -d_r + \sum_{i=0}^{r-1} d_i$, $|v_i| = 2$ for all $1\leq i <r$, and $v_r$ and $v_1$ belong to separate summands. \qed
\end{lem}

\begin{lem}[Lemma 5.2 of \cite{Gre13}]
All intervals in $\mathcal V$ are just right. In particular, they are unbreakable. \qed
\end{lem}

These two lemmas are all we will use from \cite{Gre13} in this section, but the reader should note that Section 5 of \cite{Gre13} contains a complete classification of changemaker bases whose intersection graphs are disconnected. 

\begin{lem}\label{lem:discvw1big}
If $G(\mathcal V)$ is disconnected, then $|w_1|\geq 3$.
\end{lem}
\begin{proof}
Suppose, by way of contradiction, that $|w_1| = 2$. Lest $(w_1;w_2,w_3,w_5)$ be a claw, we are in one of three scenarios: either $w_j \sim w_k$ for some $j,k\in \{2,3,5\}$ with $w_j \cdot w_1 = w_k \cdot w_1 = -1$, $w_2$ is loaded and $w_2 \cdot w_1 = 0$, or $w_3$ is loaded and $w_3 \cdot w_1 = 0$. 

Case I: $w_j \sim w_k$ for some $j, k\in \{2,3,5\}$ with $w_j \cdot w_1 = w_k \cdot w_1 = -1$. 

If $w_j$ (or $w_k$) has more than one neighbor in $\mathcal V$, then $w_j$ has two neighbors $v_{i_1}$ and $v_{i_2}$, and $(v_{i_1}, v_{i_2}, w_j)$ is a triangle, or else $(w_j;v_{i_1}, v_{i_2}, w_1)$ is a claw. Since every $v_i \in \mathcal V$ is just right, if $(v_{i_1}, v_{i_2}, w_j)$ is a triangle with $i_1 < i_2$, then $|v_i| = 2$ for all $1 \leq i \leq i_1$, $r = i_1 + 1$, $|v_i| = 2$ for all $r<i<i_2$, $v_{i_2} = -d_{i_2} + d_{i_2 -1} + \ldots + d_{i_1}$, $w_j \cdot v_{i_1} = -1$ and $w_j \cdot v_{i_2} = 1$, or else the subgraph induced by $\{v_1, \ldots, v_n, w_j\}$ contains a heavy triple. It follows that $i_2 = i_1 + 2$ and $\{i_1, r, i_2\}\subset A_j$, or else the subgraph induced by $\{v_1, \ldots, v_n, w_j\}$ is connected, hence the subgraph induced by $\{v_1, \ldots, v_n, w_j, w_k\}$ contains an incomplete cycle. It follows that $A_j = \{i_1, \ldots, n\}$ and $|v_i| = 2$ for all $i_2 + 1 \leq i \leq n$. But then $w_k \dag v_r$ and $w_k \cdot v_i = 0$ for all $i \in \{1, \ldots,r-1, r+1 = i_2, \ldots, n\}$ or else the subgraph induced by $\{v_1, \ldots, v_n, w_j, w_k\}$ contains an incomplete cycle. Conclude that $w_j$ and $w_k$ each have a single neighbor in $\mathcal V$, respectively $v_{i_j}$ and $v_{i_k}$, and, moreover, $v_{i_j}$ and $v_{i_k}$ belong to distinct components of $G(\mathcal V)$. Without loss of generality, we may assume that $A_j = \{m, \ldots, n\}$, $A_k = \{n\}$, $|v_{m+1}|\geq 3$ and $v_{m+1} \cdot v_m = 0$, and $|v_i| = 2$ for all $m+1\leq i \leq n$. 

Case I.1: $j = 5$.

If $j = 5$, then $w_6\not \sim w_5$, so $w_6 \cdot w_5 = 0$ since $-1 \leq w_6 \cdot w_5 \leq |w_5|-3$, so $|w_6|\geq 3$, or else $(w_5;v_m,w_j, w_6)$ is a claw or $(w_j, w_5, w_6)$ is a heavy triple. Then, either $A_6 \cap A_5 = \{m\}$, in which case $w_6 = -e_6 + d_m$, in which case $(v_m, w_6, v_{m+1}, \ldots, v_n, w_k, w_1, w_5, v_m)$ is an incomplete cycle, or $w_6 = -e_6 + d_n$, in which case $(v_{m+1}, w_k, w_6)$ is a heavy triple.

Case I.2: $j = 3$.

If $j = 3$, then $|w_4| \geq 3$ or else $(w_3;v_m, w_1, w_4)$ is a claw. Then, either $w_4$ is loaded, or $w_4$ is unloaded and $w_4 \cdot w_3 = 0$, since in this case $-1\leq w_4 \cdot w_3 \leq |w_3|-3$. 

Case I.2.a: $w_4$ is loaded. 

First observe that if $w_4$ is loaded, then $w_4$ is not tight, for then $w_4 \sim v_1$ and $v_r \prec w_4$, thus the subgraph induced by $\{v_1, \ldots, v_n, w_1, w_3, w_4, w_k\}$ contains an incomplete cycle. It follows that $w_4$ is unbreakable, so $w_4 \cdot w_3 \in \{-1, 0, 1\}$ since $-1 \leq w_4 \cdot w_3 \leq |w_3|-2$. Suppose first that $w_4 \cdot w_3 = -1$. Then $w_4|_{E_8} = -e_4 + e_2$ or else $w_4 \cdot w_3 \geq 0$. Then, either $k = 2$, in which case $n \in A_4$ or else $w_4 \cdot w_2 = -2$, which is absurd since $|w_4|\geq 4$ and $w_4$ is unbreakable, but then $w_4 \cdot w_3 \geq 0$, or $k = 5$, in which case $A_4 \subset \{1, \ldots, m\}$ or else $n \in A_4$ and $s^*_4 - (s^*_1 + s^*_5) \geq 0$. But then, since every element of $\mathcal V$ is just right, there is some $l \leq m$ such that $|v_l|\geq 3$, $w_4 \cdot v_l = -1$, $w_4 \cdot v_{l+1} = 1$, and $v_{l}$ and $v_{l+1}$ belong to distinct components of $G(\mathcal V)$, so the subgraph induced by $\{v_1, \ldots, v_n, w_1, w_3, w_4, w_5\}$ contains an incomplete cycle. Conclude that $w_4 \cdot w_3 \neq -1$. Suppose instead that $w_4 \cdot w_3 = 0$. We must then have $w_4|_{E_8} = -e_4 + e_2$ or else $w_4 \cdot -e_3 = 0$, in which case $A_4 \subset \{1, \ldots, m\}$ and there is an incomplete cycle as before. It follows that either $w_4 = -e_4 + e_2 + d_m$, in which case $(v_m, w_4, w_{m+1}, \ldots, v_n, w_k, w_1, w_3, v_m)$ is an incomplete cycle, or $w_4 = -e_4 + e_2 + d_n$, in which case $(v_{m+1}, w_4, w_k)$ is a heavy triple. Suppose lastly that $w_4 \cdot w_3 = 1$. Then either $n = m+1$ and $w_4 = -e_4 + e_2 + d_m +d_{m+1}$ and $(w_3, w_k, w_4)$ is a heavy triple, or $w_4|_{E_8} = -e_4 + e_2 + e_1 + \ldots$, in which case either $A_4 = \{m\}$ and there is an incomplete cycle as before, or $A_4 = \{n\}$, and either $(v_{m+1}, w_k, w_4)$ is a heavy triple or there is a claw at $v_n$. Conclude that $w_4$ is not loaded.

Case I.2.b: $w_4$ is not loaded. 

If $w_4$ is not loaded, then we arrive at a contradiction in the same way as in Case I.1, \emph{mutatis mutandi}.

Case I.3: $j = 2$. 

If $j = 2$, then $k = 3$ or $k = 5$.

Case I.3.a: $k = 5$.

If $k = 5$, then $n \in A_6$ or else $(w_5; v_n, w_1, w_6)$ is a claw, but then $w_6 \cdot w_2 \geq 1$, so $w_6 \dag w_2$, and therefore $w_6 \cdot v_m = -1$ or else $(w_2;v_m, w_1, w_6)$ is a claw, but then $m \in A_6$, so $2\leq w_6, w_2$, which is absurd.

Case I.3.b: $k = 3$.

If $k = 3$, then $|w_4| \geq 3$ or else $(w_3;v_n, w_1, w_4)$ is a claw. Then, either $w_4$ is loaded, or $w_4$ is unloaded and $w_4 \cdot w_3 = 0$, since in this case $-1\leq w_4 \cdot w_3 \leq |w_3|-3$.

Case I.3.b.i: $w_4$ is loaded. 

First observe that if $w_4$ is loaded, then $w_4$ is not tight, for then $w_4 \sim v_1$ and $v_r \prec w_4$, thus the subgraph induced by $\{v_1, \ldots, v_n, w_1, w_2, w_4, w_3\}$ contains an incomplete cycle. It follows, as before, that $w_4 \cdot w_3 \in \{-1, 0, 1\}$. If $w_4 \cdot w_3 = -1$, then, since $n \not \in A_4$, either $w_4 \cdot w_2 = -2$, which is absurd, or $w_4 \cdot w_2 = -1$, in which case $w_4 = -e_4 + e_2 + d_m$ and $(w_2, w_3, w_4)$ is a negative triangle. If $w_4 \cdot w_3 = 0$, then either $w_4 = -e_4 + e_2 + d_n$, in which case there is a claw at $v_n$, or $w_4|_{E_8} = -e_4 + e_2 + e_1 + e_5 + \ldots$, so $A_4 \subset \{1, \ldots, m\}$ and we run into the same problem as in Case II.1. If $w_4 \cdot w_3 = 1$, then $w_4|_{E_8} = -e_4 + e_2 + e_1 + e_5 + \ldots$ and $A_4 = \{n\}$, but then $(v_{m+1}, w_3, w_4)$ is a heavy triple. Conclude that $w_4$ is not loaded.

Case I.3.b.ii: $w_4$ is unloaded.

If $w_4$ is not loaded, then we arrive at a contradiction in the same way as in Case I.3.a, \emph{mutatis mutandi}.

Note that in the remaining cases, i.e. where either $w_2$ is loaded and $w_2 \cdot w_1 = 0$ or $w_3$ is loaded and $w_3 \cdot w_1 = 0$, the resolution of Case I above implies that $w_j \not \sim w_k$ for $\{j,k\} \subset \{2,3,5\}$ with $w_j \cdot w_1 = w_k \cdot w_1 = -1$. In particular this means that either $|w_j| = 2$ or $|w_k| = 2$ for $w_j$, $w_k$ unloaded, or else, without loss of generality, we may assume that $G(\{v_1, \ldots, v_n, w_j\})$ is connected, hence the subgraph induced by $\{v_1, \ldots, v_n, w_1,w_j,w_k\}$ contains an incomplete cycle. This further implies that $|w_5| = 2$ since $|w_2|$, $|w_3|\geq 3$ if $w_2$ or $w_3$ is loaded, and we arrive at a contradiction: since $w_2$ or $w_3$ is loaded and $|w_5| = 2$, $|w_l|\geq 3$ for some $l \in \{6,7,8\}$, so, taking $l$ to be minimal, the subgraph induced by $\{v_1, \ldots, v_n, w_1, w_2, w_3, w_5, \ldots, w_l\}$ contains an incomplete cycle. 
\end{proof}

\begin{lem}\label{lem:discvw1nottight}
    If $G(\mathcal V)$ is disconnected, then $w_1$ is not tight.
\end{lem}

\begin{proof}
    If $G(\mathcal V)$ is disconnected and $w_1$ is tight, then $|w_5|\geq 3$ or else $(w_1;v_1,v_r, w_5)$ is a claw. We must furthermore have $|w_5| = 3$; if $|w_5| = 4$, then $1 \leq w_5 \cdot w_1 < |w_5| -2$, and therefore $|w_5| = 4$ and $w_5 \dag w_1$, but then $w_5 \cdot v_r = 1$ or else the subgraph induced by $\{v_1, \ldots, v_n, w_1, w_5\}$ contains an incomplete cycle, or $(v_1, w_1, w_5)$ is a negative triangle, or $(v_r, w_1, w_5)$ is a negative triangle, so we must have $w_5 \cdot v_{r-1} = -1$ or else $|w_5|\geq 5$, a contradiction. But then $w_5 = -e_5 + d_k$ and $w_5 \cdot w_1 = 0$, so $|v_i| = 2$ for all $1\leq i \leq k$, but then $r = k+1$ and $(v_1, \ldots, v_{r-1}, w_5, v_r, w_1, v_1)$ is an incomplete cycle. 
\end{proof}

\begin{lem}
    If $G(\mathcal V)$ is disconnected, then $w_j$ is not tight for $j \in \{2,3,5\}$. 
\end{lem}

\begin{proof}
    If $G(\mathcal V)$ is disconnected and $w_j$ is tight for some $j \in \{2,3,5\}$ with $w_j \cdot -e_1 = -1$, then we arrive at a contradiction as in the Lemma \ref{lem:discvw1nottight}. Suppose now that $j \in \{2,3\}$, $w_j \cdot -e_1 = 0$, and $w_j$ is tight. Since by Lemma \ref{lem:discvw1big} we must have $|w_1| \geq 3$ it follows that $w_1 \pitchfork w_2$. We must furthermore have either $w_1 = -e_1 + d_n$ or $w_1 = -e_1 + d_m + \ldots + d_n$ for some $m$ with $|v_{m+1}|\geq 3$ and $|v_i| = 2$ for all $m+2 \leq i \leq n$ or else the subgraph induced by $\{v_1, \ldots, v_n, w_1, w_j\}$ contains an incomplete cycle. We must furthermore have $|w_5|\geq 3$ or else $|w_5| = 2$ and, since $w_j$ is loaded, letting $l$ be minimal in $\{6,7,8\}$ with $|w_l|\geq 3$, the subgraph induced by $\{v_1, \ldots, v_n, w_1, w_j, w_5, \ldots, w_l\}$ contains an incomplete cycle, so $w_5 \pitchfork w_2$. Now, fix $m\in \{1, \ldots, n\}$ such that $|v_{m+1}| \geq 3$ and $|v_i| = 2$ for all $m+2 \leq i \leq n$. We must have either $w_1 = -e_1 + d_m + \ldots + d_n$ and $w_5 = -e_5 + d_n$ or $w_1 = -e_1 + d_n$ and $w_5 = -e_5 + d_m + \ldots + d_n$. Letting $\{j,k\} = \{2,3\}$ with $w_j$ loaded, we must have that either $A_k = A_5 = \{n\}$, in which case $(v_{m+1}, w_k, w_5)$ is a heavy triple, or $A_k = A_1 = \{n\}$, in which case there is a claw at $v_n$, or $A_k\cap A_1 = \emptyset$ and either $n = m+1$, $w_3 = -e_3 + d_m$, and $(v_{m+1}, w_1, w_3)$ is a negative triangle, or the subgraph induced by $\{v_1, \ldots, v_n, w_1, w_j, w_k, w_5\}$ contains an incomplete cycle.
\end{proof}

\begin{lem}
If $G(\mathcal V)$ is disconnected, then there are two $j, k \in \{2,3,5\}$ with $w_j \cdot -e_1 = w_k \cdot -e_1 = -1$ and $|w_j|$, $|w_k|\geq 3$.
\end{lem}

\begin{proof}
Since $|w_1|\geq 3$, we know that $w_1$ has at least one neighbor in $\mathcal V$. If there are not two distinct $j, k \in \{2,3,5\}$ with $w_j \cdot e_1 = w_k \cdot e_1 = -1$ and $|w_j|$, $|w_k|\geq 3$, then $|w_5| = 2$, $w_j$ is loaded, and $|w_k|\geq 3$ for $\{j,k\} = \{2,3\}$, or else there is a claw at $w_1$. Furthermore, $w_1$ has a unique neighbor in $\mathcal V$, or else there is a claw at $w_1$. It follows that either $w_1 = -e_1 + d_n$ or $w_1 = -e_1 + d_m + \ldots + d_n$, $|v_{m+1}|\geq 3$, and $|v_i| = 2$ for all $m+2\leq i \leq n$. Since $w_j$ is loaded, we must have $|w_l|\geq 3$ for some $l \in \{6,7,8\}$, hence $|w_6|\geq 3$ or else, letting $l$ be minimal, either $w_l \dag w_1$ and $(w_1, w_5, w_6, \ldots, w_l, w_1)$ is an incomplete cycle, or $A_l \cap A_1 = \emptyset$, so $G(\{v_1, \ldots, v_n, w_l\})$ is connected, hence the subgraph induced by $\{v_1, \ldots, v_n, w_1, w_5, w_6, \ldots, w_l\}$ contains an incomplete cycle. It follows furthermore that $w_6 \dag w_1$ and the unique neighbor of $w_6$ in $\mathcal V$ lies in a different component than the neighbor of $w_1$ in $\mathcal V$, or else the subgraph induced by $\{v_1, \ldots, v_n, w_1, w_5, w_l\}$ contains an incomplete cycle, hence either $w_1 = -e_1 + d_m + \ldots + d_n$ and $w_6 = -e_6 + d_n$ or $w_1 = -e_1 + d_n$ and $w_6 = -e_6 + d_m + \ldots + d_n$. Now, either $A_k \cap A_1 = A_3 \cap A_6 = \emptyset$, in which case the subgraph induced by $\{v_1, \ldots, v_n, w_1, w_k, w_6\}$ contains an incomplete cycle, or $A_k \cap A_1 = \{m\}$, in which case $(v_m, w_k, v_{m+1}, \ldots, v_n, w_6, w_5, w_1, v_m)$ is an incomplete cycle, or $A_k \cap A_1 = \{n\}$, and either $w_1 = -e_1 + d_n$ and there is a claw at $v_n$, or $w_6 = -e_6 + d_n$ and $(v_{m+1}, w_k, w_6)$ is a heavy triple, or $|A_k \cap A_1| = 2$, in which case $n = m+1$, $w_k = -e_k + d_m + d_{m+1}$, $w_1 = -e_1 + d_m + d_{m+1}$, and $w_6 = -e_6 + d_{m+1}$, in which case $(w_1, w_k, w_6)$ is a heavy triple.
\end{proof}

\begin{lem}\label{lem:discvemptyint}
Let $j, k \in \{2,3,5\}$. If $w_j\cdot -e_1 = w_k \cdot -e_1 = -1$, and neither $w_j$ nor $w_k$ is tight, then $A_j \cap A_k = \emptyset$. 
\end{lem}

\begin{proof}
If $A_j \cap A_k \neq \emptyset$, then $|A_j \cap A_k| = 1$ and $w_j \dag w_k$. Without loss of generality, and lest there be a heavy triple, either $w_j = -e_j + d_{r-1} + \ldots + d_n$, $w_k = -e_k + d_{r-1}$, and $|v_i| = 2$ for all $r+1 \leq i \leq n$, or $w_j = -e_j + d_m + \ldots + d_n$, $w_k = -e_k + d_n$, $|v_{m+1}| \geq 3$, $|v_i|=2$ for all $m+2 \leq i \leq n$.

Case I: $w_j = -e_j + d_{r-1} + \ldots + d_n$, $w_k = -e_k + d_{r-1}$, and $|v_i| = 2$ for all $s+1 \leq i \leq n$.

If we are in Case I, then either $w_1 = -e_1 + d_{4-1}$, in which case $(v_{s-1}, w_1, v_r, w_k, v_{r-1})$ is an incomplete cycle, or $n = r$ and $w_1 = -e_1 + d_{r-1} + d_r$, in which case $(v_{r-1}, w_1, w_j, w_k, v_{r-1})$ is an incomplete cycle, or $n = r = 2$ and $w_1 = -e_1 + d_0 + d_1 + d_2$ in which case $(v_r, w_1, w_j, w_k, v_r)$ is an incomplete cycle. 

Case II: $w_j = -e_j + d_m + \ldots + d_n$, $w_k = -e_k + d_n$, $|v_{m+1}| \geq 3$, $|v_i|=2$ for all $m+2 \leq i \leq n$.

If we are in Case II, then either $w_1 = -e_1 + d_{m}$, in which case $(v_{m+1}, w_1, w_k)$ is a heavy triple, or $w_1 = -e_1 + d_n$, in which case $n = r$ or else there is a claw at $v_n$, or $n = r$, and $w_1 = -e_1 + d_r + d_{r+1}$, or else $A_1\cap \{m,\ldots, n\} = \emptyset$, in which case the subgraph induced by $\{v_1, \ldots, v_n,w_1,w_j,w_k\}$ contains an incomplete cycle

Case II.1: $n = r$ and $w_1 = -e_1 + d_r$.

Suppose that $n = r$ and $w_1 = -e_1 + d_r$. Let $l \in \{1, \ldots, 8\} \setminus \{1,j,k\}$. If $|A_l| \geq 1$, then $A_l = \{r-1, r\}$ or else $w_l \dag v_r$ and either $(v_r;w_1,w_k, w_l)$ is a claw or $(v_r, w_1, w_l)$ or $(v_r, w_1, w_k)$ is a heavy triple, or $w_l$ is tight and $(v_1, \ldots, v_{r-1}, w_j, w_k, v_r, w_l, v_1)$ is an incomplete cycle. It follows that either $|w_5| = 2$, or $j = 5$, or $k = 5$, or else $5 \not \in \{j,k\}$ and $w_5 \cdot w_j = 2$, which is absurd.

Case II.1.a: $|w_5| = 2$.

If $|w_5| = 2$, then $|w_6| = |w_7| = |w_8| = 2$ or else the subgraph induced by $\{v_1, \ldots, v_n, w_1, w_2,$ $w_3,w_5, w_6, w_7, w_8\}$ contains an incomplete cycle, thus $w_2$ and $w_3$ are unloaded and $\{2,3\}= \{j,k\}$. It follows that $|w_4|\geq 3$ or else there is a claw at $w_3$. Suppose that $w_4$ is unloaded. Then $w_4 = -e_4 + d_{r-1} + d_r$, so $j = 3$ or else $w_4\cdot w_2 = 2$, but then $(v_r, w_1, w_4, w_2, v_r)$ is an incomplete cycle. If $w_4$ is loaded, then $w_4|_{E_8} = -e_4 + e_2 + e_1 + d_{r-1} + d_r$, but then $(w_5;w_1, w_4, w_6)$ is a claw. 

Case II.1.b: $j = 5$.

If $j = 5$, then $w_5 = -e_5 + d_{r-1} + d_r$, so $|w_6|\geq 3$ or else $(w_5;v_{r-1}, w_k,w_6)$ is a claw, but then $w_6 = -e_6 + d_{r-1}$ or else either $n \in A_6$ and $(w_1, w_k, w_6)$ is a heavy triple or $n \not \in A_6$ and $2 \leq w_6 \cdot v_r$, which is absurd, but then $(v_{r-1}, w_6, v_r, w_k, w_5, v_{r-1})$ is an incomplete cycle. 

Case II.1.c: $k = 5$.

If $k = 5$, then $w_5 = -e_5 + d_r$, so $|w_6|\geq 3$ or else $(w_5;v_{r-1}, w_j, w_6)$ is a claw, but then $n \in A_6$ or else the subgraph induced by $\{v_1, \ldots, v_n, w_j,w_5, w_6\}$ contains an incomplete cycle, so $w_6 = -e_6 + d_n$ or else $2\leq w_6 \cdot w_j$, but then $(v_r, w_1, w_6)$ is a heavy triple. 

Case II.2: $n = r$ and $w_1 = -e_1 + d_{r-1} + d_r$.

Suppose that $n = r$ and $w_1 = -e_1 + d_{r-1} + d_r$. Let $l \in \{1, \ldots, 8\} \setminus \{1,j,k\}$. If $A_l = \{r\}$, then $w_l \dag v_r$ and either $w_l \dag w_k$, in which case $(v_r, w_k,w_l)$ is a heavy triple, $w_l \dag w_j$, in which case $(v_r,w_l,w_j,w_k,v_r)$ is an incomplete cycle, $w_l \dag w_1$, in which case $(v_r, w_l, w_1, w_j, w_k, v_r)$ is an incomplete cycle. If $A_l = \{r-1\}$, then $(v_{r-1}, w_l, v_r, w_k, w_j,w_1,v_{r-1})$ is an incomplete cycle. If $A_l = \{r-1\}$, then $w_l \dag v_{r-1}$ and $(w_1,w_j,w_l)$ is a heavy triple. If $r= 2$ and $A_l = \{0,1,2\}$, then $w_l \dag v_2$ and either $w_l \dag w_1$ or $w_l \dag w_j$, and in either case the subgraph induced by $\{v_1, \ldots, v_n, w_1,w_j,w_k,w_l\}$ contains an incomplete cycle. Conclude that $|A_l| = 0$ for $l \not \in \{1,j,k\}$. 

Case II.2.a: $|w_5| = 2$.

If $|w_5| = 2$, then $\{2,3\} = \{j,k\}$, and $w_2$ and $w_3$ are unloaded. It follows that there is a claw at $w_3$ unless $|w_4|\geq2$, in which case $j = 2$ and $w_4 = -e_4 + e_2 + e_1$, or else either $|A_4|\geq 1$ or $w_4$ is loaded and $s^*_4 \leq |\sigma|_1 + 1$, which is absurd, but then $(w_1,w_2, w_4)$ is a heavy triple. 

Case II.2.b: $j = 5$.

If $j = 5$, then $(w_5;v_{r-1}, w_k, w_6)$ is a claw.

Case II.2.c: $k = 5$.

If $k = 5$, then $(w_5;v_r,w_k,w_6)$ is a claw.
\end{proof}

\begin{prop}\label{prop:discv}
    If $G(\mathcal V)$ has two components then $\sigma = (1, 1, 2, \ldots, 2) \in \mathbb Z^{n+1}$ and either $s^*=(2n-1, 1, 2, 0, 0, 0, 0, 0)$ or $s^* = (2n-1, 1, 0, 0, 2, 0, 0, 0)$.
\end{prop}

\begin{proof}
    Note first that the Lemma \ref{lem:discvemptyint} ensures that $n\not \in A_j \cap A_k$, so, without loss of generality, let us assume that $n \not \in A_j$, hence there is some $m \in A_j$ such that $m+1 \not \in A_j$. Notice that if $m \in A_j$ but $m+1\not \in A_j$, then $w_j \cdot v_{m+1} \geq 1$ and $|v_{m+1}|\geq 3$, so $w_j \cdot v_{m+1} = 1$. Then either $w_j \cdot v_m = -1$, or $v_m = -d_m + d_{m-1} + \ldots + d_l$ for some $l <m-1$ with $|v_{l+1}|\geq 3$ and $|v_i| = 2$ for all $l+2\leq i\leq m-1$, $v_{m+1} = -d_{m+1} + d_m + d_{m-1}$, and $A_j \cap \{l, \cdots, m+1\} = \{l,m\}$, in which case $w_j \cdot v_{l+1} = 1$, so $(v_{l+1}, v_{m+1}, w_j)$ is a heavy triple. It follows that either $A_j  = \{m\}$ or $n \in A_j$. Supposing that $A_j = \{m,\}$, note that since $G(\{v_1, \ldots, v_n, w_j\})$ is connected, $w_k$ has only a single neighbor in $\mathcal V$ or else there is a heavy triple or an incomplete cycle, hence $n \in A_k$. We must therefore have that $A_j = \{m\}$ and either $A_k = \{n\}$ or $A_k = \{m', \ldots, n\}$ for some $m < m'$ with $|v_{m'+1}|\geq 3$ and $|v_i| = 2$ for all $m' + 2 \leq i \leq n$. Since $G(\{v_1, \ldots, v_n, w_j, w_k\})$ is connected, we must have $m \in A_1$, $|A_1 \cap A_k| = 1$, and $w_1$ is just right, or else $G(\{v_1, \ldots, v_n, w_1, w_j, w_k\})$ contains a heavy triple or an incomplete cycle. It follows that $A_k = \{n\}$ and therefore $w_1 = -e_1 + d_m + \ldots + d_n$, but then $r = 2$ or else $(v_{r-1};v_{r-2}, w_1, w_j)$ is a claw. Furthermore, we must have $|v_i| = 2$ for all $3\leq i \leq n$. Hence, $\sigma = (1,1, 2, \ldots, 2) \in \mathbb Z^{n+1}$.

    It follows that $|A_l| = 0$ for all $l \in \{1, \ldots, 8\} \setminus \{1,j,k\}$ or else either $w_l$ is tight and $(V_1,w_j,v_2,w_l,v_1)$ is an incomplete cycle or $w_l \sim v_1$, in which case either $(w_1,w_j,w_l)$ is a heavy triple or $(v_1;w_1,w_j,w_l)$ is a claw, or $w_l \dag v_2$, in which case either $(w_j,w_k,w_l)$ is a heavy triple or $n = 2$ and $(v_2;w_j,w_k,w_l)$ is a claw or $n > 2$ and $(v_2;v_3,w_j,w_l)$ is a claw, or $w_l \sim v_n$ and either $(w_j,w_k,w_l)$ is a heavy triple or $n = 2$ and $(v_2;w_j,w_k,w_l)$ is a claw or $n > 2$ and $(v_2;v_3, w_k,w_l)$ is a claw. It follows that $j \neq 5$ or else $(w_5;v_1,v_2,w_6)$ is a claw. If $j = 3$, then $w_4$ is loaded and $w_4 = -e_4 + e_1 + e_k$, in which case $(w_1,w_k,w_4)$ is a heavy triple, or else $(w_3;v_1,v_2,w_4)$ is a claw. Conclude that $j = 2$ and that $w_4$ is unloaded, hence $|w_4| = 2$, in the same breath. Therefore, $s^* = (2n-1, 1, 2, 0, 0, 0, 0, 0)$ or $s^* = (2n-1, 1, 0, 0, 1, 0, 0, 0)$.   
\end{proof}

\section{Main Results}\label{sec:main}
\begin{table} 
\caption{$\Lambda(p, -k^2 \ \mod p) \hookrightarrow E_8\oplus \mathbb Z^{j-1},\  j \geq 1, \ \sigma=(1,1,\ldots,1)$.}
\begin{tabular}{|c|c|c|c|c|}\hline
 Tange Type & $p$& $k$& $s^*$ & Proposition \\ \hline
 A$_{1-}$ & $14j^2 - 7j + 1$ & $-7j + 2$ & $(0, 1, j-1, 0, 0, 0, 0, 0)$ & \ref{prop:alloness10}\\ \hline
 A$_{1+}$ & $14j^2 + 7j + 1$& $7j + 2$  & $(0,1,j-1,1,0,0,0,0)$ & \ref{prop:alloness10}\\ \hline
 A$_{2-}$ & $20j^2 - 15j + 3$& $-5j + 2$& $(0,1,0,0,j-1,0,0,0)$ & \ref{prop:alloness10}\\ \hline
 A$_{2+}$ & $20j^2 + 15j + 3$& $5j + 2$ &$(0,1,0,0,j-1,1,0,0)$ & \ref{prop:alloness10}\\ \hline
 B$_-$ & $30j^2 -9j + 1$& $-6j + 1$& $(j-1,1,0,1,0,0,0,0)$ & \ref{prop:alloness1full} \\ \hline
 B$_+$ & $30j^2 + 9j + 1$& $6j + 1$& $(j-1,0,0,1,1,0,0,0)$ & \ref{prop:alloness1full}\\\hline
 C$_{1-}$ & $42j^2 - 23j + 3$& $-7j + 2$& $(0,j,j-1,1,0,0,0,0)$& \ref{prop:alloness10}\\ \hline
 C$_{1+}$ & $42j^2 + 23j + 3$& $7j + 2$& $(1,j,j-1,0,0,0,0,0)$ & \ref{prop:alloness11}\\ \hline
 C$_{2-}$ & $42j^2 - 47j + 13$& $-7j + 4$& $(0,j,j-1,0,0,0,0,0)$& \ref{prop:alloness10}\\ \hline
 C$_{2+}$ & $42j^2 + 47j + 13$& $7j + 4$& $(1, j, j-1, 1, 0, 0, 0, 0)$ & \ref{prop:alloness11}\\ \hline
 D$_{1-}$ & $52j^2 - 15j + 1$& $-13j + 2$& $(0,j, 0, 0, j-1,1,0,0)$& \ref{prop:alloness10} \\ \hline
 D$_{1+}$ & $52j^2 + 15j + 1$& $13j + 2$& $(1, j, 0, 0, j-1, 0 ,0,0)$ & \ref{prop:alloness11}\\ \hline
 D$_{2-}$ & $52j^2 - 63j + 19$& $-13j + 8$&$(0,j,0,0,j-1,0,0,0)$& \ref{prop:alloness10} \\ \hline
 D$_{2+}$ & $52j^2 + 63j + 19$& $13j + 8$& $(1,j,0,0,j-1,1,0,0)$ & \ref{prop:alloness11}\\ \hline
 E$_{1-}$ & $54j^2 - 15j + 1$ & $-27j + 4$& $(j-1,0,0,j,1,0,0,0)$& \ref{prop:alloness1full} \\ \hline
 E$_{1+}$ & $54j^2 + 15j + 1$ & $27j + 4$& $(j-1, 1, 1, j, 0, 0, 0, 0)$& \ref{prop:alloness1full}\\ \hline
 E$_{2-}$ & $54j^2 - 39j + 7$ & $-27j + 10$& $(j-1, 1, 0, j, 0, 0, 0, 0)$& \ref{prop:alloness1full}\\\hline
 E$_{2+}$ & $54j^2 + 39j + 7$ & $27j + 10$ & $(j-1,0,1,j,1,0,0,0)$& \ref{prop:alloness1full}\label{table:allones}\\\hline

 \end{tabular}
\end{table}

\begin{table}
\caption{$\Lambda(p, -k^2 \ \mod p)\hookrightarrow E_8\oplus \mathbb Z^j,\ j \geq 1, \ \sigma=(1,1,\ldots,1,j)$.}
\begin{tabular}{|c|c|c|c|c|}
\hline
 Tange Type & $p$& $k$& $s^*$ & Proposition \\ \hline
 F$_{1-}$ & $69j^2 - 17j + 1$ & $-23j + 3$ &$(j-1,j,0,0,1,0,0,0)$ &\ref{prop:vttight}\\ \hline
 F$_{1+}$ & $69j^2 + 17j + 1$ & $23j + 3$ &$(j-1,j+1,1,0,0,0,0,0)$ &\ref{prop:vttight}\\ \hline
 F$_{2-}$ & $69j^2 - 29j + 3$ & $-23j + 5$ & $(j-1,j,1, 0 ,0,0,0,0)$&\ref{prop:vttight}\\ \hline
 F$_{2+}$ & $69j^2 + 29j + 3$ & $23j + 5$ & $(j-1,j+1,0,0,1,0,0,0)$&\ref{prop:vttight}\\ \hline
 G$_{1-}$ & $85j^2 - 19j + 1$ & $-17j + 2$& $(j-1, 0, j, 0, 1, 0, 0, 0)$&\ref{prop:vttight}\\ \hline
 G$_{1+}$ & $85j^2 + 19j + 1$ & $17j + 2$&$(j-1, 1, j+1, 0, 0, 0, 0, 0)$&\ref{prop:vttight}\\ \hline
 G$_{2-}$ & $85j^2 - 49j + 7$ & $-17j + 5$& $(j-1,1,j,0,0,0,0,0)$&\ref{prop:vttight}\\ \hline
 G$_{2+}$ & $85j^2 + 49j + 7$ & $17j + 5$&$(j-1, 0, j+1, 0, 1, 0, 0, 0)$&\ref{prop:vttight}\\ \hline
 H$_{1-}$ & $99j^2 - 35j + 3$ & $-11j + 2$& $(j-1,0,1,0,j,0,0,0)$&\ref{prop:vttight}\\ \hline
 H$_{1+}$ & $99j^2 + 35j + 3$ &$11j + 2$&$(j-1, 1, 0, 0, j+1, 0, 0, 0)$ &\ref{prop:vttight}\\ \hline
 H$_{2-}$ & $99j^2 - 53j + 7$ & $-11j + 3$&$(j-1, 1, 0, 0, j, 0, 0, 0)$&\ref{prop:vttight}\\ \hline
 H$_{2+}$ & $99j^2 + 53j + 7$ &$11j + 3$&$(j-1, 0, 1, 0, j+1, 0, 0, 0)$ &\ref{prop:vttight}\label{table:vntight}\\ \hline
 \end{tabular}
 \end{table}

 \begin{table}
\caption{$\Lambda(p, -k^2 \ \mod p)\hookrightarrow -E_8\oplus \mathbb Z^j,\ j \geq 1, \ \sigma=(1,2,2,\ldots,2)$.}
\begin{tabular}{|c|c|c|c|c|}
\hline
 Tange Type & $p$& $k$& $s^*$ & Proposition \\ \hline
  I$_{1-}$ & $120j^2 - 16j + 1$ & $-12j + 1$&$(2(j-1),1, 2, 0, 0,0,0,0)$ & \ref{prop:v1tight} \\ \hline
 I$_{1+}$ & $120j^2 + 16j + 1$&$12j + 1$&$(2(j-1),1, 0, 0, 2, 0, 0, 0)$ & \ref{prop:v1tight}\\\hline
 I$_{2-}$ & $120j^2 - 20j +1$ & $-20j + 2$& $(2(j-1),2,0,0,1,0,0,0)$ & \ref{prop:v1tight}\\ \hline
 I$_{2+}$ & $120j^2 + 20j +1$ & $20j + 2$ &$(2(j-1),0,2,0,1,0,0,0)$& \ref{prop:v1tight}\\ \hline
 I$_{3-}$ & $120j^2 - 36j + 3$ & $-12j + 2$ &$(2(j-1),2,1,0,0,0,0,0)$& \ref{prop:v1tight}\\ \hline
 I$_{3+}$ & $120j^2 + 36j + 3$ & $12j + 2$&$(2(j-1), 0, 1, 0, 2, 0, 0, 0)$ & \ref{prop:v1tight}\label{table:v1tight}\\ \hline
 \end{tabular}
 \end{table}

 \begin{table}
\caption{$\Lambda(p, -k^2 \ \mod p)\hookrightarrow -E_8\oplus \mathbb Z^j,\ j \geq 1, \ \sigma=(1,1,2,\ldots,2)$.}
\begin{tabular}{|c|c|c|c|c|}
\hline
 Tange Type & $p$& $k$& $s^*$ & Proposition\\ \hline
 J$_-$ & $120j^2-104j+22$& $-12j + 5$& $(2j-1, 1, 2, 0, 0, 0, 0, 0)$ & \ref{prop:discv}\\ \hline
 J$_+$ & $120j^2+104j+22$& $12j + 5$ & $(2j-1, 1, 0, 0, 2, 0, 0, 0)$ & \ref{prop:discv}\label{table:discv}\\ \hline
 \end{tabular}
 \end{table}
\subsection{How to read the tables in this section.} In \cite{Tan09} and \cite{Tan10}, Tange provided a tabulation of simple knots in lens spaces admitting integer surgeries to $\mathcal P$. There is one lens space on Tange's list of surgeries, $L(191, 157)$, that our tables do not account for. The linear lattice $\Lambda(191,157)$ has rank 8, and embeds in $E_8 \oplus \mathbb Z$ as the orthogonal complement to the $E_8$-changemaker $\tau = (s, (1))$ with $s^* = (1, 1, 1, 1, 0, 0, 0, 0)$, which falls out of the purview of our analysis of $E_8$-changemakers in $E_8 \oplus \mathbb Z^{n+1}$ with $n \geq 2$. With the exception of the lone simple knot in the lens space $L(191,157)$, each of Tange's 19 families is a family of simple knots $K_j \subset L(p_j,q_j)$ representing the class $k_j \in \mathbb Z/ p_j \mathbb Z \cong H_1(L(p_j, q_j))$, parametrized by $j \in \mathbb Z \setminus \{0\}$, where $p_j$ is a quadratic polynomial of $j$, $k_j$ is a linear polynomial in $j$, and $q_j$ is the residue of $-k_j^2 \ \mod p_j$. After characterizing linear $E_8$-changemaker lattices, it seems even more surprising to the author that the Tange knot admitting a surgery to $L(191, 157)$ does not fit in to an infinite family of knots like all the other Tange knots.

From the perspective of the linear lattices bounded by the lens spaces accounted for on Tange's list, each of Tange's families naturally splits into two subfamilies: a $-$-family and a $+$-family, depending on whether $j \leq -1$ or $j \geq 1$, respectively. We have recorded the naming convention according to \cite{Tan10}, $p_j$, $k_j$, the $E_8$-changemaker whose orthogonal complement is isomorphic to $\Lambda(p_j, q_j)$, and the proposition in which the $E_8$-changemaker makes an appearance in Tables 1--4, which are separated by changemaker tails. 

The attentive reader will observe that these 38 families of $E_8$-changemakers do not account for all forty-four families of lattices identified in Section \ref{sec:identifying}. In addition to these families, there are $6$ families of orthogonal sums of pairs of linear lattices which embed as the orthogonal complements to $E_8$-changmakers $\tau = (s, (1,\ldots, 1)) \in E_8 \oplus \mathbb Z^{n+1}$ given by $s^* = (n+1, 1, 0, 0, 0, 0, 0, 0)$ and $s^* = (n+1, 0, 1, 0, 1, 0, 0, 0)$, which correspond to surgeries on cables of the exceptional fiber of order $-2$, $s^* = (n+1, 0, 1, 0, 0, 0, 0, 0)$ and $s^* = (n+1, 1, 0, 0, 1, 0, 0, 0)$, which correspond to surgeries on cables of the exceptional fiber of order $3$, and $s^* = (n+1, 0, 0, 0, 1, 0, 0, 0)$ and $s^* = (n+1, 1, 1, 0, 0, 0, 0, 0)$, which correspond to surgeries on cables of the exceptional fiber of order $5$. Thus, we have accounted for all linear lattices and orthogonal pairs of linear lattices which are also $E_8$-changemaker lattices.

\subsection{Proofs of the main theorems.}

\begin{proof}[Proof of Theorem \ref{thm:twosummandhardeightembedding}]
Note that none of the decomposable lattices identified in Section \ref{sec:identifying} have $\Lambda(2,1)$ summands. It follows that if $\Lambda(p,q) \oplus \Lambda(2,1) \cong (\tau)^\perp$ for some $E_8$-changemaker $\tau \in E_8 \oplus \mathbb Z^{n+1}$, then $n \in \{-1, 0, 1\}$. In fact, we cannot have $n = 1$, or else the isomorphism $\Lambda(p,q) \oplus \Lambda(2,1)$ must take $\Lambda(2,1)$ to $(d_1-d_0)$, but then $A_j = \{0, 1\}$ and $w_j$ is just right for all $j \in \{1, \ldots, 8\}$ with $|w_j|\geq 3$, so $|w_j| \geq 3$ for at most one $j \in \{2,3,5\}$, so $(w_1; w_2, w_3, w_5)$ is a claw. Furthermore, we cannot have $n = 0$, since then $\sigma = (1)$ and for all $j \in \{1, \ldots, 8\}$ either $|w_j| \geq 3$ or $w_j \sim w_{j'}$ for some $j'$ such that $e_j$ and $e_{j'}$ are adjacent in the $E_8$ Dynkin diagram. A computer search of the 1003 non-zero $E_8$-changemakers in $E_8$ reveals only two whose orthogonal complements are isomorphic to $\Lambda(p,q) \oplus \Lambda(2,1)$: either $s^* = (1, 0, 0, 1, 0, 0, 0, 0)$, in which case $\Lambda(p,q)$ is given by the Gram matrix
$$\begin{bmatrix}
    2 & -1 & 0 & 0 & 0 & 0\\
    -1 & 4 & -1 & 0 & 0 & 0\\
    0 & -1 & 2 & -1 & 0 & 0\\
    0 & 0 & -1 & 2 & -1 & 0\\
    0 & 0 & 0 & -1 & 2 & -1\\
    0 & 0 & 0 & 0 & -1 & 2
\end{bmatrix},$$
and one readily certifies that $(p,q) = (27,16)$, or $s^* = (0, 0, 1, 0, 0, 0, 0, 0)$, in which case $\Lambda(p,q)$ is given by the Gram matrix
$$\begin{bmatrix}
    2 & -1 & 0 & 0 & 0 & 0\\
    -1 & 2 & -1 & 0 & 0 & 0\\
    0 & -1 & 2 & -1 & 0 & 0\\
    0 & 0 & -1 & 2 & -1 & 0\\
    0 & 0 & 0 & -1 & 2 & -1\\
    0 & 0 & 0 & 0 & -1 & 2
\end{bmatrix},$$
and one readily certifies that $(p,q) = (7,6)$. Note that $$0 = \max\{\langle \mathfrak c, \tau\rangle \colon \langle \mathfrak c , \mathfrak c\rangle = \text{rk}(E_8) -4d(\mathcal P) = 0\},$$ so $2g(K) = 2p$ for any knot $K \subset \mathcal P$ with $K(2p) \cong L(p,q)\#L(2,1)$.    
\end{proof}

\begin{proof}[Proof of Theorem \ref{thm:integersurgeryrealization}]
Every family of linear lattices admitting an $E_8$-changemaker embedding, i.e. those families captured in Propositions \ref{prop:alloness10} (cf. Table \ref{table:allones}), \ref{prop:alloness11} (cf. Table \ref{table:allones}), \ref{prop:alloness1full} (cf. Table \ref{table:allones}), \ref{prop:vttight} (cf. Table \ref{table:vntight}), \ref{prop:v1tight} (cf. Table \ref{table:v1tight}), \ref{prop:discv} (cf. Table \ref{table:discv}), is bounded by a lens space realized by surgery on a Tange knot in $\mathcal P$.
\end{proof}

\subsection{Characterizing $C$ and $E$.} Theorem \ref{thm:twosummandhardeightembedding} alone does not give us Theorem \ref{thm:main}; rather, it puts us in a position to leverage the following theorems, due, in order, to Baker, Rasmussen (building on work of Ni), and Hedden, to great effect. We are grateful to John Baldwin for pointing out as much. 

\begin{thm}[Theorem 1.1 of \cite{Bak06}]\label{thm:pplusoneoverfour}
Let $K \subset L(p,q)$ be a knot whose exterior $M_K$ fibers over $S^1$ with fiber surface $F$ where $\partial F$ is connected, and let $g(K) = g(F)$. If $g(K) \leq (p+1)/4$, then $K$ is one-bridge with respect to a Heegaard torus of $L(p,q)$. \qed
\end{thm}

The following theorem follows from the combination of Theorem 4.3 and Proposition 4.5 of \cite{Ras07} (cf. \cite{Ni}).

\begin{thm}\label{thm:pplusoneovertwo}
Let $K \subset L(p,q)$ be a knot that admits a surgery to an integer homology sphere L-space. If $g(K) < (p+1)/2$, then rk $\widehat{HFK}(L(p,q), K) = p$. \qed
\end{thm}

For each homology class $\alpha \in H_1(L(p,q); \mathbb Z)$, the simple knot $K_\alpha$ representing $\alpha$ satisfies $\text{rk }\widehat{HFK}(L(p,q), K) = p$; those familiar with knot Floer homology can readily see this by observing that the knot $K_\alpha$ admits a genus one doubly pointed Heegaard diagram obtained by simply placing an extra basepoint in the standard genus one singly-pointed Heegaard diagram $\mathcal H$ of $L(p,q)$ such that $\widehat{CF}(\mathcal H) \cong \widehat{HF}(L(p,q))$. Hedden observed that simple knots are the only one-bridge knots in lens spaces that are \emph{Floer simple}---that is, they have minimal rank knot Floer homology as they satisfy $\text{rk}\widehat{HFK}(L(p,q), K) = p$.

\begin{thm}[Theorem 1.9 of \cite{Hed}]\label{thm:onebridgefloersimpleimpliessimple}
If $K \subset L(p,q)$ is primitive, $g(K) \leq \frac{p+1}{4}$, and rk $\widehat{HFK}(L(p,q), K) = p$, then $K$ is simple. \qed
\end{thm}

We will, in addition, make use of one more lemma, a variation of \cite[Lemma 2.2]{Ras07}. For any knot $K$ in a rational homology 3-sphere $Y$, the meridian of $K$, $\mu$, is still well-defined, though since $K$ may not be null-homologous in $Y$ there is no canonical choice of a Seifert longitude. However, we fix a longitude of $K$ and denote it by $\lambda$,  oriented so that $[\lambda]\cdot [\mu]=1$ with respect to the orientation on $\partial M_K$ induced by $Y$. Recall that, by the half-lives-half-dies principle, there is a unique slope $\gamma$ on $\partial M_K$ such that some integer multiple of it bounds in $M_K$.  Let $\alpha = a\cdot [\mu] + p\cdot [\lambda] \in H_1(\partial M_K;\mathbb Z)$ be a primitive homology class represented by such a curve. The number $p$ is well-defined---it is the \emph{order} of $[K]$ in $H_1(Y;\mathbb Z)$. Replacing $[\lambda]$ by $[\lambda] + [\mu]$ has the effect of replacing $a$ by $a-p$, so the value $a \ \mod p$ is an invariant of $K$---and so is the quantity $a/p \ \mod 1$, referred to as the \emph{self-linking number} of $K$ and denoted $K \cdot K$. The self-linking number of $K$ depends only on the class $[K] \in H_1(Y;\mathbb Z)$, and $[nK]\cdot[nK] \equiv n^2[K]\cdot[K]\ \mod 1$.

A \emph{distance} $n$ surgery on $K$ is a manifold $Y'$ obtained by Dehn filling $M_K$ along a curve representing $k\cdot[\mu] + n\cdot[\lambda]$ for some $k \in \mathbb Z$. More geometrically, $Y'$ is obtained by distance $n$ surgery on $K \subset Y$ if the curve $\beta$ along which $M_K$ is Dehn filled to obtain $Y'$ has minimum geometric intersection number $n$ with $\mu$. From this perspective, the relation of being obtained by distance $n$ surgery is clearly symmetric, as $\beta$ becomes the meridian of the surgery dual knot $K'\subset Y'$ when $M_K$ is Dehn filled along $\beta$.  

\begin{lem}[(cf. Lemma 2.2 of \cite{Ras07})]\label{lem:distancensurgery}
Let $K \subset Y$ be a knot in a rational homology sphere with $H_1(Y; \mathbb Z) \cong \mathbb Z/p\mathbb Z$. Then K has a distance $n$ surgery $Y'$ which is an integer homology sphere if and only if $[K]$ generates $H_1(Y)$ and its self-linking number $a/p$ is congruent to $\pm n'/p \ \mod 1$, where $nn' \equiv 1 \ \mod p$.
\end{lem}
\begin{proof}[Sketch of proof.] The proof of this lemma follows the proof of \cite[Lemma 2.2]{Ras07} up until the final paragraph which discusses the map $A: H_1(T^2; \mathbb Z) \to H_1(S^1 \times D^2; \mathbb Z) \oplus H_1(M_K; \mathbb Z)$ in the Mayer-Vietoris sequence associated to the decomposition $Y' = (S^1 \times D^2) \cup_{T^2} M_K$. Since $H_2(Y';\mathbb Z)$ and $H_1(Y'; \mathbb Z)$ are both assumed to be $(0)$, $A$ is an isomorphism. The map $H_1(T^2; \mathbb Z) \to H_1(S^1 \times D^2; \mathbb Z)\cong \mathbb Z$ is given by $x \mapsto x\cdot [\beta]$, where $[\beta] = k\cdot [\mu] + n\cdot [\lambda]$. On the other hand, the map $H_1(T^2; \mathbb Z) \to H_1(M_K; \mathbb Z) \cong \mathbb Z$ is given by $x \mapsto x \cdot [\alpha]$, where $[\alpha] = a \cdot [\mu] + p \cdot [\lambda]$. Therefore, with respect to the basis $([\mu], [\lambda])$ on $H_1(T^2;\mathbb Z)$, the matrix for the map $A$ is given by
\begin{equation}
    M_A = \begin{bmatrix}
    -n & k\\
    -p & a
    \end{bmatrix}.
\end{equation}
In order for $A$ to be an isomorphism we must be able to choose a $k$ so that $\det(M_A) = \pm 1$, which is possible if and only if $na \equiv \pm 1 \ \mod p$.

\end{proof}

\begin{proof}[Proof of Theorem \ref{thm:main} and Theorem \ref{thm:mainrephrased}]
For any knot $\kappa \subset \mathcal P$ with $\kappa(p/2) \cong L(p,q)$, the cabling construction yields a knot $K$---which is isotopic to a simple closed curve representing $p\cdot [\mu] + 2 \cdot [\lambda] \in H_1(\partial M_K, \mathbb Z)$, and which we call the \emph{$(p,2)$-cable} of $\kappa$---such that $K(2p) = \kappa(p/2) \# L(2,p) \cong L(p,q) \# L(2,1)$. Crucially, since we must have $2p \geq 2g(K)$, it follows that every knot $K$ with $K(2p) \cong L(p,q) \#L(2,1)$ arises as the cable of some knot $\kappa$ with a non-integer surgery either to $L(p,q)$ or $L(2,1)$ by Matignon--Sayari \cite{MS03}. We may furthermore deduce that $\kappa(p/2) \cong L(p,q)$, since in order for $\kappa(2/p)$ $(p\geq 3)$ to be an L-space we must have $g(\kappa) = 0$, and so $\kappa$ bounds a disk in $\mathcal P$ and is therefore unknotted in a 3-ball; in this case $\kappa(2/p)\cong \mathcal P \# L(2,1)$ and is thus never a lens space. 

In the case that $K$ is the $(p,2)$-cable of $\kappa$, an elementary calculation shows that
\begin{equation}\label{eq:firstgenuseq}
    2g(K) -1 = 2p + 2(2g(\kappa)-1) - p.
\end{equation} On the other hand, if $K(2p) = L(2,1)\#L(p,q)$, then $(p,q) = (7,6)$ or $(27,16)$ and $g(K) = p$ by Theorem \ref{thm:twosummandhardeightembedding}. Then, (\ref{eq:firstgenuseq}) reads
\begin{equation}
    g(\kappa) = (p+1)/4,
\end{equation}
and since $\kappa(p/2) \cong L(p,q)$, by Theorems \ref{thm:pplusoneoverfour}, \ref{thm:pplusoneovertwo}, and \ref{thm:onebridgefloersimpleimpliessimple} and the preceding paragraph, it follows that the surgery dual to $\kappa$, $\kappa^* \subset L(p,q)$, is simple.

By Lemma \ref{lem:distancensurgery}, we must furthermore have that the self-linking number of $\kappa^*$ is $\pm 2'/p \ \mod 1$. Explicit computation shows that $\pm 2 \ \mod 7$ are the only solutions to $x^2 \equiv \pm 4 \ \mod 7$ and that $\pm 11 \ \mod 27$ are the only solutions to $x^2 \equiv \pm 14 \ \mod 27$, so, up to orientation reversal, $E^*$ is the unique simple knot in $L(7,6)$ with a distance 2 integer homology sphere surgery, and $C^*$ is the unique simple knot in $L(27,16)$ with a distance 2 integer homology sphere surgery. Therefore, if $K(2p) \cong L(2,1) \# L(p,q)$, $K$ is either the $(7,2)$-cable of $E$ or the $(27,2)$-cable of $C$.
\end{proof}


\begin{thebibliography}{99}

\bibitem{Bak06} K. Baker: \textit{Small genus knots in lens spaces have small bridge number}, Algebr. Geom. Topol. \textbf{6} (2006), 1519--1621.

\bibitem{Bak20}K. L. Baker: \textit{The {P}oincar\'{e} homology sphere, lens space surgeries, and some
              knots with tunnel number two}, Pacific J. Math. \textbf{305} (2020), no. 1, 1--27.

\bibitem{Ber}J. Berge. \textit{Some knots with surgeries yielding lens spaces}, 2018, available at \url{https://arxiv.org/abs/1802.09722}.

\bibitem{BZ96}S. Boyer and X. Zhang: \textit{Finite {D}ehn surgery on knots}, J. Amer. Math. Soc. \textbf{9} (1996), no. 4, 1005--1050.

\bibitem{Cau21}J. Caudell: \textit{Three lens space summands from the Poincar\'e homology sphere}, 2021, available at \url{https://arxiv.org/abs/2101.01256}.

\bibitem{Eft09}E. Eftekhary: \textit{Seifert fibered homology spheres withtrivial Heegaard Floer homology}, 2009, available at \url{https://arxiv.org/abs/0909.3975}.

\bibitem{Fro96} K. Fr\o yshov: \textit{The Seiberg-Witten equations and four-manifolds with boundary}, Math. Res. Lett. \textbf{3} (1996), no. 3, 373--390.

\bibitem{Ger} L. J. Gerstein: \textit{Nearly unimodular quadratic forms}, Ann. of Math. (2) \textbf{142}, no. 3 (1995), 597--610.

\bibitem{Gib} J. Gibbons: \textit{Deficiency symmetries of surgeries in {$S^3$}}, Int. Math.Res. Not. (2015), 12126--12151.

\bibitem{GAS86} F. Gonz\'alez-Acu\~na and H. Short: \textit{Knot surgery and primeness}, Math.  Proc. Cambridge Philos. Soc. \textbf{99} (1986), no. 1, 89--102.

\bibitem{Gre13}J. E. Greene: \textit{The lens space realization problem}, Ann. of Math. (2) \textbf{177} (2013), no. 2, 449--511.

\bibitem{Gre15}J. E. Greene: \textit{L-space surgeries, genus bounds, and the cabling conjecture}, J. Differential Geom. \textbf{100} (2015), no. 3, 491--506.

\bibitem{Hed} M. Hedden: \textit{On Floer homology and the Berge conjecture on knots admitting lens space surgeries}, Trans. Amer. Math. Soc. \textbf{2} (2011), no. 2, 949--968.

\bibitem{Hof98} J. A. Hoffman: \textit{There are no strict great {$x$}-cycles after a reducing or {$P^2$} surgery on a knot}, J. Knot Theory Ramifications \textbf{7} (1998), no. 5, 549--569.

\bibitem{Kir10} R. Kirby: Problems in low-dimensional topology (2010). \url{math.berkeley.edu/~kirby/problems.ps.gz}.

\bibitem{Lic62} W. B. R. Lickorish: \textit{A representation of orientable combinatorial {$3$}-manifolds}, Ann. of Math. (2) \textbf{76} (1962), 531--540.

\bibitem{MS03} D. Matignon and N. Sayari: \textit{Longitudinal slope and {D}ehn fillings}, Hiroshima Math. J. \textbf{33} (2003), no. 1, 127--136.

\bibitem{McC} D. McCoy: \textit{Surgeries, sharp 4-manifolds and the {A}lexander polynomial}, Algebr. Geom. Topol. \textbf{21} (2021), no. 5, 2649--2676.

\bibitem{Ni} Y. Ni: \textit{Link Floer homology detects the Thurston norm}, Geom. Topol. \textbf{13} (2009), no. 5, 2991--3019.

\bibitem{OS03}P. Ozsv\'{a}th and Z. Szab\'{o}: \textit{Absolutely graded {F}loer homologies and intersection forms for four-manifolds with boundary}, Adv. Math. \textbf{173} (2003), no. 2, 179--261.

\bibitem{OS05}P. Ozsv\'{a}th and Z. Szab\'{o}: \textit{On knot {F}loer homology and lens space surgeries}, Topology \textbf{44} (2005), no. 6, 1281--1300.

\bibitem{Ras07}J. Rasmussen: \textit{Lens space surgeries and L-space homology spheres}, 2007, available at \url{https://arxiv.org/abs/0710.2531}.

\bibitem{Sca18}C. Scaduto: \textit{On definite lattices bounded by a homology 3-sphere and Yang-Mills instanton Floer theory}, 2018, available at \url{https://arxiv.org/abs/1805.07875}.

\bibitem{Tan09}M. Tange: \textit{Lens spaces given from {$L$}-space homology 3-spheres}, Experiment. Math. \textbf{18} (2009), no. 3, 285--301.

\bibitem{Tan10}M. Tange: \textit{A complete list of lens spaces constructed by Dehn surgery I}, 2010, available at \url{https://arxiv.org/abs/1005.3512}.


\bibitem{Wal60}A. H. Wallace: \textit{Modifications and cobounding manifolds}, Canadian J. Math. \textbf {12} (1960), 503--528.
\end{thebibliography}
\end{document}